\documentclass[
	final,
	paper=a4,
	pagesize=auto,
	fontsize=10pt,
	version=last,
]{scrartcl}

\usepackage[utf8]{inputenc}
\usepackage[
	english,
]{babel}

\KOMAoptions{%
   DIV=10,
}%
\KOMAoptions{%
   numbers=noenddot
}%
\KOMAoptions{
   twoside=semi,  
   twocolumn=false, 
   headinclude=false,%
   footinclude=false,%
   mpinclude=false,%
   headlines=2.1,%
   headsepline=true,%
   footsepline=false,%
   cleardoublepage=empty 
}%
\KOMAoptions{%
   parskip=relative, 
}%
\KOMAoptions{%
}%
\KOMAoptions{%
}%
\KOMAoptions{%
   bibliography=totoc 
}%
\KOMAoptions{%
}%
\KOMAoptions{%
}%
\KOMAoptions{%
}%

\makeatletter
\newcommand{\LoadPackagesNow}{}
\newcommand{\LoadPackageLater}[2][]{%
   \g@addto@macro{\LoadPackagesNow}{%
      \usepackage[#1]{#2}%
   }%
}
\makeatother

\usepackage{xspace}
\usepackage{xargs}
\usepackage{ifthen}
\usepackage{etoolbox}

\usepackage{calc}
\usepackage[
	margin=3cm,
	]{geometry}
\usepackage{framed}
\usepackage{mdframed}
\usepackage{pbox}
\usepackage{enumitem}
\usepackage{array}
\usepackage{multirow}
\usepackage{longtable}
\usepackage{hhline}
\usepackage{afterpage}
\usepackage{pdflscape}
\usepackage{rotating}
\usepackage{booktabs}
\usepackage{setspace}
\usepackage[
   bottom,      
   stable,      
   ragged,      
   multiple,    
   flushmargin,
   hang,
]{footmisc}

\usepackage{mathpazo}

\usepackage[%
   headsepline,
   automark,
   komastyle,
]{scrpage2}

\clearscrheadings
\clearscrplain
\lehead{\pagemark}
\rohead{\pagemark}
\rehead{\headmark}
\lohead{\shortauthor}
\automark[section]{section} 

\usepackage{relsize}
\usepackage[normalem]{ulem}      
\usepackage{soul}		           
\usepackage{url}
\usepackage{lipsum}

\usepackage[table]{xcolor}
\usepackage{tikz}
\usetikzlibrary{arrows,shapes,calc,decorations.pathreplacing,intersections,quotes,angles,positioning,fit,petri,through}
\usepackage[%
]{graphicx}

\usepackage{epstopdf}
\usepackage{wrapfig}
\usepackage{caption}
\usepackage{subcaption}
\usepackage[
   tbtags,    
   sumlimits,  
   nointlimits, 
   namelimits, 
   reqno,     
]{amsmath} %

\usepackage{amsfonts}
\usepackage{mathrsfs} 
\usepackage{dsfont}
\usepackage{amssymb}
\usepackage{units}
\LoadPackageLater{amsthm}
\LoadPackageLater{thmtools}
\usepackage[fixamsmath,disallowspaces]{mathtools}
\mathtoolsset{showonlyrefs}
\mathtoolsset{centercolon=true}
\usepackage{bm} 
\makeatletter
\g@addto@macro\bfseries{\boldmath}
\makeatother
\allowdisplaybreaks[4]
\numberwithin{equation}{section}

\usepackage[toc]{appendix}

\usepackage{csquotes}
\usepackage[
	style=alphabetic,
	sorting=nyt,
	giveninits=true,
	maxbibnames=99,
	backend=biber,
	maxalphanames=5,
]{biblatex}

\renewbibmacro{in:}{}
\DeclareFieldFormat{pages}{#1}

\usepackage[textsize=tiny,english,colorinlistoftodos,
	disable,
	]{todonotes}

\definecolor{pdfurlcolor}{rgb}{0,0,0.6}
\definecolor{pdffilecolor}{rgb}{0.7,0,0}
\definecolor{pdflinkcolor}{rgb}{0,0,0.6}
\definecolor{pdfcitecolor}{rgb}{0,0,0.6}
\usepackage[
  colorlinks=true,         
  urlcolor=pdfurlcolor,    
  filecolor=pdffilecolor,  
  linkcolor=pdflinkcolor,  
  citecolor=pdfcitecolor,  %
  raiselinks=true,			 
  breaklinks,              
  verbose,
  hyperindex=true,         
  linktocpage=true,        
  hyperfootnotes=false,     
  bookmarks=true,          
  bookmarksopenlevel=1,    
  bookmarksopen=true,      
  bookmarksnumbered=true,  
  bookmarkstype=toc,       
  plainpages=false,        
  pageanchor=true,         
  pdfdisplaydoctitle=true, 
  pdfstartview=FitH,       
  pdfpagemode=UseOutlines, 
  pdfpagelabels=true,           
  pdfpagelayout=OneColumn, 
]{hyperref}

\LoadPackagesNow

\newcommand{\ifargdef}[3][{}]{\ifthenelse{\equal{#2}{}}{#1}{#3}}

\newlength{\hangind}
\newcommand{\myhangindent}[1]{\settowidth{\hangind}{\widthof{#1}}\hangindent=\the\hangind}

\setkomafont{section}{\Large\rmfamily\bfseries}
\setkomafont{subsection}{\large\rmfamily\bfseries}
\setkomafont{subsubsection}{\rmfamily\bfseries}
\setkomafont{paragraph}{\rmfamily\bfseries}
\setkomafont{pageheadfoot}{\smaller\scshape\rmfamily}

\usepackage{tabularx}
\hypersetup{
	pdfauthor={Martin Genzel and Gitta Kutyniok},
	pdftitle={The Mismatch Principle},
	pdfsubject={Article},
	pdfcreator={PDF-LaTeX},
}

\newenvironment{properties}[2][2em]
{\begin{enumerate}[label={\textsc{(#2\arabic*)}},leftmargin=#1]}
{\end{enumerate}}

\newenvironment{listing}
{\begin{itemize}[itemindent=0em,leftmargin=1.2em]}
{\end{itemize}}

\newenvironment{rmklist}
{\begin{enumerate}[label={(\arabic*)},itemindent=2em,leftmargin=0em]}
{\end{enumerate}}

\newenvironment{deflist}
{\begin{enumerate}[label={(\arabic*)}]}
{\end{enumerate}}

\renewcommand{\qedsymbol}{$_\blacksquare$}

\providecommand{\qedhere}{\hfill\qedsymbol}

\newtheoremstyle{claim}
	{\topsep}{\topsep}%
	{\itshape}
	{}
	{}
	{}
	{.5em}
	{{\bfseries\boldmath\thmname{#1} \thmnumber{#2}} \thmnote{(#3)}}

\newtheoremstyle{definition}
	{\topsep}{\topsep}%
	{}
	{}
	{}
	{}
	{.5em}
	{\textbf{\thmname{#1} \thmnumber{#2}} \thmnote{(#3)}}
	
\newtheoremstyle{algorithm}
	{\topsep}{\topsep}%
	{}
	{}
	{\bfseries\boldmath}
	{}
	{\newline}
	{\thmname{#1} \thmnumber{#2} \thmnote{(#3)}}

\mdfdefinestyle{boxed}{innertopmargin =6pt, splittopskip = \topskip, skipbelow=6pt, skipabove=12pt}

\mdfdefinestyle{emphframe}{linecolor=black, innertopmargin = 3pt, splittopskip = \topskip, skipbelow=6pt, skipabove=6pt}

\declaretheorem[style=claim,numberwithin=section]{theorem}
\declaretheorem[style=claim,sibling=theorem]{lemma}
\declaretheorem[style=claim,sibling=theorem]{corollary}
\declaretheorem[style=claim,sibling=theorem]{proposition}
\declaretheorem[style=definition,sibling=theorem]{definition}

\declaretheorem[style=definition,sibling=theorem]{assumption}

\declaretheorem[style=definition,sibling=theorem,qed=$\Diamond$]{remark}
\declaretheorem[style=definition,sibling=theorem,qed=$\Diamond$]{example}
\declaretheorem[style=algorithm,sibling=theorem,%
	preheadhook={\begin{mdframed}[style=emphframe] \setcounter{mpfootnote}{\value{footnote}}},%
	postfoothook=\end{mdframed}]{experiment}
\declaretheorem[style=algorithm,sibling=theorem,%
	preheadhook={\begin{mdframed}[style=emphframe] \setcounter{mpfootnote}{\value{footnote}}},%
	postfoothook=\end{mdframed}]{algorithm}
\declaretheorem[style=definition,sibling=theorem,%
	preheadhook={\begin{mdframed}[style=boxed] \setcounter{mpfootnote}{\value{footnote}}},%
	postfoothook=\end{mdframed}]{recipe}

\newcommand{\opleft}[1]{\mathopen{}\left#1}
\newcommand{\opright}[1]{\right#1\mathclose{}}
\newcommandx{\braces}[4]{%
\ifstrequal{#3}{normal}{#1#4#2}{%
\ifstrequal{#3}{auto}{\left#1#4\right#2}{%
\ifstrequal{#3}{opauto}{\opleft#1#4\opright#2}{%
#3#1#4#3#2}}}%
}
\newcommandx{\opannot}[3][3=\downarrow]{\stackrel{\mathclap{\substack{#1 \\ #3 \vspace{2pt}}}}{#2}}
\newcommandx{\lineannot}[3][3=\rightarrow]{\mathllap{\boxed{\text{\textsmaller{#1}}} #3} #2}
\newcommandx{\multilineannot}[4][4=\rightarrow]{\mathllap{\boxed{\parbox{#1}{\RaggedRight\textsmaller{#2}}} #4} #3}
 
\newcommand{\R}{\mathbb{R}} 
\newcommand{\eps}{\varepsilon} 

\renewcommand{\iff}{\Leftrightarrow} 
\renewcommand{\implies}{\Rightarrow} 
\newcommand{\suchthat}[1][normal]{\ifstrequal{#1}{normal}{\mid}{#1|}} 

\newcommand{\setcompl}[1]{#1^c} 
\newcommand{\cardinality}[1]{\abs{#1}} 
\newcommand{\intersec}{\cap} 
\newcommand{\boundary}[1]{\partial#1} 

\newcommandx{\intvcl}[3][1=normal]{\braces{[}{]}{#1}{#2, #3}} 
\newcommandx{\intvop}[3][1=normal]{\braces{(}{)}{#1}{#2, #3}} 
\newcommandx{\intvclop}[3][1=normal]{\braces{[}{)}{#1}{#2, #3}} 
\newcommandx{\intvopcl}[3][1=normal]{\braces{(}{]}{#1}{#2, #3}} 
\DeclareMathOperator*{\argmin}{argmin} 
\DeclareMathOperator{\sign}{sign}
\newcommandx{\abs}[2][1=normal]{\braces{\lvert}{\rvert}{#1}{#2}} 
\newcommandx{\ceil}[2][1=normal]{\braces{\lceil}{\rceil}{#1}{#2}} 
\newcommandx{\floor}[2][1=normal]{\braces{\lfloor}{\rfloor}{#1}{#2}} 
\newcommandx{\round}[2][1=normal]{\braces{[}{]}{#1}{#2}} 
\newcommandx{\der}[1]{D^{#1}} 
\newcommandx{\gradient}{\nabla} 
\newcommandx{\partder}[4][1={},4={}]{\frac{\partial^{#4} #2}{\partial #3^{#4}}\ifargdef{#1}{\Big|_{#1}}} 
\newcommandx{\integ}[4][1={},2={}]{\int_{#1}^{#2} #3 \, #4} 
\newcommandx{\asympffaster}[2][1=normal]{o\braces{(}{)}{#1}{#2}} 
\newcommandx{\asympfaster}[2][1=normal]{O\braces{(}{)}{#1}{#2}} 
\newcommandx{\asympeq}[2][1=normal]{\Theta\braces{(}{)}{#1}{#2}} 
\newcommandx{\asympsslower}[2][1=normal]{\omega\braces{(}{)}{#1}{#2}} 
\newcommandx{\asympslower}[2][1=normal]{\Omega\braces{(}{)}{#1}{#2}} 

\DeclareMathOperator{\Id}{Id} 
\newcommandx{\norm}[2][1=normal]{\braces{\|}{\|}{#1}{#2}} 
\renewcommandx{\sp}[3][1=normal]{\braces{\langle}{\rangle}{#1}{#2, #3}} 
\newcommandx{\End}[2][2={}]{\mathcal{L}\opleft( #1 \ifargdef{#2}{, #2} \opright)} 
\newcommand{\orthcompl}[1]{{#1}^\perp} 
\DeclareMathOperator{\ran}{ran} 
\DeclareMathOperator{\spann}{\operatorname{span}} 
\newcommand{\T}{\mathsf{T}} 
\newcommand{\psinv}[1]{#1^{\dagger}} 
\renewcommand{\vec}[1]{\boldsymbol{#1}} 

\newcommandx{\measure}[2][1=normal]{\operatorname{vol}\braces{(}{)}{#1}{#2}} 
\DeclareMathOperator{\supp}{supp} 
\newcommandx{\Leb}[3][1={},3=normal]{L^{#2}\ifargdef{#1}{\braces{(}{)}{#3}{#1}}{}} 
\newcommandx{\Lebnorm}[4][1=normal,3={2},4={}]{\norm[#1]{#2}_{#3}} 
\renewcommandx{\l}[3][1={},3=normal]{\ell^{#2}\ifargdef{#1}{\braces{(}{)}{#3}{#1}}} 
\newcommandx{\lnorm}[4][1=normal,3={2},4={}]{\norm[#1]{#2}_{#3}} 
\newcommandx{\Smooth}[4][1={},3={},4=normal]{C_{#3}^{#2}\ifargdef{#1}{\braces{(}{)}{#4}{#1}}} 
\newcommandx{\Schwartz}[2][1={},2=normal]{\mathscr{S}\ifargdef{#1}{\braces{(}{)}{#2}{#1}}} 
\newcommandx{\Schwartzpoly}[2][1=normal]{\braces{\langle}{\rangle}{#1}{\abs[#1]{#2}} } 
\newcommandx{\Tempdistr}[2][1={},2=normal]{\mathscr{S}'\ifargdef{#1}{\braces{(}{)}{#2}{#1}}} 
\newcommandx{\distrinp}[3][1=normal]{\braces{\langle}{\rangle}{#1}{#2, #3}} 
\newcommandx{\ft}[3][1=default,2=auto]{
\ifstrequal{#1}{default}{\widehat{#3}}{
\ifstrequal{#1}{long}{{\braces{(}{)}{#2}{#3}}^{\wedge}}{}}} 
\newcommandx{\ift}[3][1=default,2=auto]{
\ifstrequal{#1}{default}{\check{#3}}{
\ifstrequal{#1}{long}{{\braces{(}{)}{#2}{#3}}^{\vee}}{}}} 

\renewcommand{\v}{\vec{v}}

\newcommand{\func}{f}
\newcommand{\funcclass}{\mathcal{F}}
\newcommand{\fobs}{g}
\newcommand{\Fobs}{G}
\newcommand{\scalfac}{\mu}
\newcommand{\x}{\vec{x}}

\newcommand{\lat}{\vec{s}}

\newcommand{\latsig}{{\vec{v}}}
\newcommand{\latnoise}{{\vec{n}}}
\newcommand{\latpv}{\vec{z}}
\newcommand{\targetlatpv}{\target{\latpv}}
\newcommand{\trulatpv}{\tru{\latpv}}
\newcommand{\xpv}{\vec{\beta}}

\newcommand{\atoms}{\vec{A}}
\newcommand{\atrafo}{\vec{M}}

\newcommand{\p}{{p}}
\renewcommand{\d}{{d}}
\newcommand{\U}{\vec{U}}
\newcommand{\D}{\vec{D}}
\newcommand{\tru}[1]{{#1}_0}
\newcommand{\solu}[1]{\hat{#1}}
\newcommand{\target}[1]{{#1}^\natural}
\newcommand{\soluexp}[1]{{#1}^\ast}

\newcommand{\actsupp}{\mathcal{S}}
\newcommand{\sset}{K}
\newcommand{\ssetalt}{L}
\newcommand{\tset}{T}
\newcommand{\modelcovar}[2][]{\rho_{#1}(#2)}
\newcommand{\modeldev}[2][]{\sigma_{#1}(#2)}

\newcommand{\loss}{\mathcal{L}}
\newcommand{\exloss}{\mathcal{E}}

\newcommand{\lossemp}[1][{}]{\bar{\mathcal{L}}_{#1}}

\newcommand{\multiplterm}[1]{\mathcal{M}(#1)}
\newcommand{\quadrterm}[1]{\mathcal{Q}(#1)}
\newcommand{\vdir}{\vec{v}} 
\newcommand{\I}[1]{\vec{I}_{#1}}
\newcommand{\vnull}{\vec{0}}
\newcommand{\h}{\vec{h}}
\newcommand{\proj}{\vec{P}}

\newcommand{\vunit}{\vec{e}}

\newcommandx{\prob}[2][1={},2=normal]{\mathbb{P}\ifargdef{#1}{\braces{[}{]}{#2}{#1}}}

\newcommandx{\mean}[2][1={},2=normal]{\mathbb{E}\ifargdef{#1}{\braces{[}{]}{#2}{#1}}}
\newcommandx{\var}[2][1={},2=normal]{\mathbb{V}\ifargdef{#1}{\braces{[}{]}{#2}{#1}}}

\newcommand{\distributed}{\sim}
\newcommandx{\Normdistr}[3][1=normal]{\mathcal{N}\braces{(}{)}{#1}{#2, #3}} 
\newcommandx{\normsubg}[2][1=normal]{\norm[#1]{#2}_{\psi_2}} 
\newcommand{\subgparam}{\kappa} 
\newcommand{\gaussian}{\vec{g}} 

\newcommand{\Covmatr}{\vec{\Sigma}} 

\newcommandx{\anorm}[3][1=normal,3={\sset}]{\norm[#1]{#2}_{#3}} 
\newcommandx{\opnorm}[2][1=normal]{\norm[#1]{#2}_{\operatorname{op}}} 

\newcommandx{\ball}[2][1={},2={}]{B_{#1}^{#2}} 
\renewcommand{\S}{\mathbb{S}} 
\newcommand{\meanwidth}[2][{}]{w_{#1}(#2)} 
\newcommand{\effdim}[2][{}]{w_{#1}^2(#2)} 
\newcommand{\conic}{\wedge} 
\newcommand{\cone}[1]{\operatorname{cone}(#1)} 

\graphicspath{{images/}}

\begin{document}

\renewcommand*{\thefootnote}{\fnsymbol{footnote}}
\pagestyle{scrheadings}

\begin{center}
	\bfseries\larger[3]{The Mismatch Principle: The Generalized Lasso Under Large Model Uncertainties}
\end{center}

\vspace{1\baselineskip}
\begin{addmargin}[2em]{2em}
\begin{center}
\noindent{\normalsize\bfseries{Martin Genzel\footnote{Technische Universit\"at Berlin, Department of Mathematics, Berlin, Germany; E-Mail:~\href{mailto:genzel@math.tu-berlin.de}{\texttt{genzel@math.tu-berlin.de}}} \qquad Gitta Kutyniok\footnote{Technische Universit\"at Berlin, Department of Mathematics and Department of Electrical Engineering \& Computer Science, Berlin, Germany; University of Troms\o, Department of Physics and Technology, Troms\o, Norway; E-Mail:~\href{mailto:kutyniok@math.tu-berlin.de}{\texttt{kutyniok@math.tu-berlin.de}}}}}
%
%
\end{center}


\vspace{1\baselineskip}
{\smaller
\noindent\textbf{Abstract.}
We study the estimation capacity of the generalized Lasso, i.e., least squares minimization combined with a (convex) structural constraint.
While Lasso-type estimators were originally designed for noisy linear regression problems, it has recently turned out that they are in fact robust against various types of model uncertainties and misspecifications, most notably, non-linearly distorted observation models.
This work provides more theoretical evidence for this somewhat astonishing phenomenon.
At the heart of our analysis stands the mismatch principle, which is a simple recipe to establish theoretical error bounds for the generalized Lasso.
The associated estimation guarantees are of independent interest and are formulated in a fairly general setup, permitting arbitrary sub-Gaussian data, possibly with strongly correlated feature designs; in particular, we do not assume a specific observation model which connects the input and output variables.
Although the mismatch principle is conceived based on ideas from statistical learning theory, its actual application area are (high-dimensional) estimation tasks for semi-parametric models.
In this context, the benefits of the mismatch principle are demonstrated for a variety of popular problem classes, such as single-index models, generalized linear models, and variable selection. 
Apart from that, our findings are also relevant to recent advances in quantized and distributed compressed sensing. 

\vspace{.5\baselineskip}

\vspace{.5\baselineskip}
\noindent\textbf{MSC classes.}
68T37, 60D05, 90C25, 62F30, 62F35

}
\end{addmargin}
\newcommand{\shortauthor}{Genzel and Kutyniok: The Mismatch Principle}

\renewcommand*{\thefootnote}{\arabic{footnote}}
\setcounter{footnote}{0}


\thispagestyle{plain}

\section{Introduction}
\label{sec:intro}

\subsection{Motivation: From Statistical Learning to Semi-Parametric Estimation}
\label{subsec:intro:statlearn}

One of the key objectives in statistical learning and related fields is to study \emph{sampling processes} that arise from a random pair $(\x, y)$ in $\R^\p \times \R$ whose joint probability distribution $\mu$ is \emph{unknown}. In this context, the random vector $\x = (x_1, \dots, x_\p) \in \R^\p$ is typically regarded as a collection of \emph{input variables} or \emph{features}, whereas $y \in \R$ corresponds to an \emph{output variable} that one would like to predict from $\x$.\footnote{There exist many synonyms for $\x$ and $y$. We collected some of them in Table~\ref{tab:results:terminology}, which summarizes the most important notations and terminology of this work.}
More precisely, the main goal is to select a function $\soluexp\func \colon \R^\p \to \R$ from a certain \emph{hypothesis class} $\funcclass \subset L^2(\R^\p,\mu_{\x})$ such that the expected risk is minimized:
\begin{equation}\label{eq:intro:explossmin-general}
	\min_{\func \in \funcclass} \ \mean[(y - \func(\x))^2].
\end{equation}
A solution $\soluexp\func \in \funcclass$ to \eqref{eq:intro:explossmin-general} is then called an \emph{expected risk minimizer} and yields the \emph{optimal} estimator (approximation) of $y$ in $\funcclass$, with respect to the mean squared error.

However, it is not possible to directly solve \eqref{eq:intro:explossmin-general} in practice because the underlying probability measure $\mu$ of $(\x, y)$ is unknown.
Instead, one is merely given a finite amount of \emph{observed data}
\begin{equation}\label{eq:intro:datasamples}
	(\x_1, y_1), \dots, (\x_n, y_n) \in \R^\p \times \R
\end{equation}
where each of these pairs is an independent random sample of $(\x,y)$. This limitation suggests to consider the empirical analog of \eqref{eq:intro:explossmin-general}, which is well-known as \emph{empirical risk minimization} (\emph{ERM}):
\begin{equation}\label{eq:intro:emplossmin-general}
	\min_{\func \in \funcclass} \ \tfrac{1}{n}\sum_{i = 1}^n (y_i - \func(\x_i))^2.
\end{equation}
Such types of optimization problems have been extensively studied in statistical learning theory during the last decades, e.g., see \cite{vapnik1998learning,cucker2007learning,shalev2014understanding} for comprehensive overviews.
One of the primary concerns in this field of research is to establish (non-asymptotic) bounds on the \emph{estimation error}
\begin{equation}\label{eq:intro:esterr}
	\mean{}_{\x}[(\solu\func(\x) - \soluexp\func(\x))^2]
\end{equation}
with $\solu\func \in \funcclass$ being a minimizer of \eqref{eq:intro:emplossmin-general}.\footnote{Note that $\solu\func$ is actually a random element of $\funcclass$, since it depends on the sample set $\{(\x_i, y_i)\}_{i = 1}^n$. The subscript $\x$ in \eqref{eq:intro:esterr} indicates that the expectation is computed only with respect to $\x$ and we condition on $\{(\x_i, y_i)\}_{i = 1}^n$. In particular, $\solu\func$ is not random in this situation.}
In particular, it has turned out that, in many situations of interest, empirical risk minimization constitutes a \emph{consistent} (or \emph{asymptotically unbiased}) estimator of $\soluexp\func$, in the sense that $\mean{}_{\x}[(\solu\func(\x) - \soluexp\func(\x))^2] \to 0$ in probability as $n \to \infty$.
Provided that the sample size $n$ is sufficiently large, one can therefore expect in such case that $\solu\func$ acquires the predictive capacity of the expected risk minimizer $\soluexp\func$.

In this paper, we only focus on those hypothesis classes that contain \emph{linear} functions, i.e.,
\begin{equation}
	\funcclass = \{ \x \mapsto \sp{\x}{\xpv} \suchthat \xpv \in \sset \}
\end{equation}
for a certain \emph{convex} subset $\sset \subset \R^\p$, which is referred to as the \emph{hypothesis set} or \emph{constraint set}.
While this choice of $\funcclass$ is actually one of the simplest examples to think of, the associated empirical risk minimization program has prevailed as a standard approach in modern statistics and signal processing. Indeed, \eqref{eq:intro:emplossmin-general} just turns into a \emph{constrained least squares estimator} in the linear case:
\begin{equation} \label{eq:intro:klasso}\tag{$\mathsf{P}_{\sset}$}
	\min_{\xpv \in \sset} \ \tfrac{1}{n} \sum_{i = 1}^n (y_i - \sp{\x_i}{\xpv})^2,
\end{equation}
and the expected risk minimization problem \eqref{eq:intro:explossmin-general} takes the form
\begin{equation} \label{eq:intro:explossmin-linear}
	\min_{\xpv \in \sset} \ \mean[(y - \sp{\x}{\xpv})^2].
\end{equation}
The purpose of the \emph{parameter vector} $\xpv \in \sset$ is now to ``explain'' the observed data by a linear model.
At the same time, the hypothesis set $\sset$ imposes additional structural constraints that restrict the set of admissible models.
In this way --- if $\sset$ is appropriately chosen --- accurate estimation can even become possible in \emph{high-dimensional} scenarios where $n \ll \p$.
Perhaps the most popular example is the \emph{Lasso} \cite{tibshirani1996lasso}, which was originally designed for \emph{sparse} linear regression and employs a scaled $\l{1}$-ball as constraint set. 
As a tribute to this contribution, we will speak from now on of the \emph{generalized Lasso} when referring to \eqref{eq:intro:klasso}; this naming particularly facilitates the distinction between \eqref{eq:intro:klasso} and the more general problem of empirical risk minimization stated in \eqref{eq:intro:emplossmin-general}.

An appealing aspect of many results in statistical learning is that almost no specific assumptions on the output variable $y$ are made, except for some mild regularity.
While these theoretical findings in principle cover a wide class of estimation problems, the actual error bounds are however rather implicit and technical.
On the other hand, there is usually more knowledge available about the sampling process in concrete model settings, eventually allowing for more practicable guarantees.
This work is an attempt to bridge the gap between these two conceptions: our main results still take the abstract viewpoint of statistical learning, but also provide an easy-to-use toolbox to derive \emph{explicit error bounds} for \eqref{eq:intro:klasso}.
Before presenting further details of our approach in the next subsection, we wish to highlight two important aspects of the sample pair $(\x,y) \in \R^\p \times \R$ that will emerge over and over again in the remainder of this article:
\begin{deflist}
\item
	\emph{Correlated inputs.} In case $\x$ represents real-world data, the input variables are typically strongly correlated and affected by noise.
		In this context, it is useful to assume that $\x$ can be at least linearly factorized as follows:
		\begin{equation}\label{eq:intro:fvdecomp}
			\x = \atoms \lat 
		\end{equation}
		where $\lat = (s_1, \dots, s_\d)$ is a centered isotropic random vector\footnote{A random vector $\lat$ in $\R^\d$ is \emph{centered} if $\mean[\lat] = \vnull$ and it is \emph{isotropic} if $\mean[\sp{\lat}{\latpv}^2] = \lnorm{\latpv}^2$ for all $\latpv \in \R^\d$.} in $\R^\d$ and $\atoms \in \R^{\p \times \d}$ is a deterministic matrix, called the \emph{mixing matrix}; in fact, such an \emph{isotropic decomposition} does always exist as long as $\x$ is centered with a well-defined covariance matrix of rank $\d$ (see Proposition~\ref{prop:results:setup:modeldecomp}).
		Noteworthy, the statistical ``fluctuations'' of $\x$ are now completely determined by $\lat$, so that one may  phrase the probabilistic dependence of $y$ on $\x$ also in terms of $\lat$; and indeed, this is precisely what we will do henceforth.
		
		The \emph{linear factor model} of \eqref{eq:intro:fvdecomp} leads to the following important simplification of the estimation error in \eqref{eq:intro:esterr}:
		Let $\solu\xpv\in \R^\p$ be a solution to \eqref{eq:intro:klasso} and let $\soluexp\xpv\in \R^\p$ be an expected risk minimizer of \eqref{eq:intro:explossmin-linear}, i.e., $\soluexp\func(\cdot) \coloneqq \sp{\cdot}{\soluexp\xpv}$ solves \eqref{eq:intro:explossmin-general}. Setting $\solu\func(\cdot) \coloneqq \sp{\cdot}{\solu\xpv}$ and using the isotropy of $\lat$, we observe that
		\begin{align}
			\mean{}_{\x}[(\solu\func(\x) - \soluexp\func(\x))^2] &= \mean{}_{\x}[\sp{\x}{\solu\xpv - \soluexp\xpv}^2] = \mean{}_{\lat}[\sp{\atoms \lat}{\solu\xpv - \soluexp\xpv}^2] \\*
			&= \mean{}_{\lat}[\sp{\lat}{\atoms^\T\solu\xpv - \atoms^\T\soluexp\xpv}^2] = \lnorm{\atoms^\T\solu\xpv - \atoms^\T\soluexp\xpv}^2 \ . \label{eq:intro:estimerror}
		\end{align}
		This identity reflects a well-known issue of least squares estimators: both $\solu\xpv$ and $\soluexp\xpv$ can be highly non-unique due to (perfectly) correlated input variables. But the error measure of \eqref{eq:intro:estimerror} resolves this ``ambiguity,'' in the sense that one can still compare two parameter vectors by weighting them with $\atoms^\T$.
		However, it is worth emphasizing that \eqref{eq:intro:fvdecomp} and \eqref{eq:intro:estimerror} are primarily of theoretical interest because the mixing matrix $\atoms$ is often unknown in practice. In that situation, $\lat$ is not directly accessible, which explains why its components $s_1, \dots, s_\d$ are sometimes also referred to as \emph{latent variables} or \emph{latent factors}.
\item
	\emph{Semi-parametric models.} It is quite obvious that the linear hypothesis functions used in \eqref{eq:intro:klasso} are not capable of learning complicated non-linear output rules. More specifically, we cannot expect that $\solu\func(\x) = \sp{\x}{\solu\xpv}$ or $\soluexp\func(\x) = \sp{\x}{\soluexp\xpv}$ provide a good approximation of $y$. On the other hand, there exist many relevant scenarios where the output variable $y$ (approximately) follows a semi-parametric model and one is interested in estimating some unknown parameters, rather than predicting the actual value of $y$.
	In such a case, there is still hope that the outcome of the generalized Lasso \eqref{eq:intro:klasso} --- despite a large \emph{model uncertainty} --- allows us to deduce the parametric dependence between $y$ and $\x = \atoms \lat$, where we assume that \eqref{eq:intro:fvdecomp} holds.
	This rationale will prove particularly useful for the following two popular observation models, which will serve as running examples in this article:
	\begin{listing}
	\item
		\emph{Single-index models.} Let $\trulatpv \in \R^\d$ and assume that
		\begin{equation}\label{eq:intro:sim}
			y = \fobs(\sp{\lat}{\trulatpv})
		\end{equation}
		for a scalar function $\fobs \colon \R \to \R$ which can be non-linear, unknown, and random (independently of $\lat$).
		The goal is to construct an estimator of the unknown \emph{index vector} $\trulatpv$ (or at least of its direction $\trulatpv / \lnorm{\trulatpv}$) using \eqref{eq:intro:klasso}.
	\item
		\emph{Variable selection.} Let $\{k_1, \dots, k_S\} \subset \{1, \dots, \d\}$ and assume that
		\begin{equation}\label{eq:intro:vs}
			y = \Fobs(s_{k_1}, \dots, s_{k_S})
		\end{equation}
		for a function $\Fobs \colon \R^S \to \R$ which can be again non-linear, unknown, and random.
		Can we use the generalized Lasso \eqref{eq:intro:klasso} to extract the set of \emph{active variables} $\actsupp \coloneqq \{k_1, \dots, k_S\}$?
	\end{listing}
\end{deflist}

The above concerns give rise to several general issues that we would like to address in this work:
\begin{properties}[3em]{Q}
\item\label{quest:intro:estimation}
	\emph{Estimation.}
	Can one establish a non-asymptotic upper bound on the estimation error in \eqref{eq:intro:estimerror} that is controlled by the sample size $n$?
	When does $\solu\latpv \coloneqq \atoms^\T \solu\xpv \in \R^\d$ provide a consistent estimator of $\soluexp\latpv \coloneqq \atoms^\T \soluexp\xpv \in \R^\d$?
\item\label{quest:intro:approximation}
	\emph{Interpretability.}
	Does it always make sense to aim for an estimate of $\soluexp\latpv$ when one is interested in specific parameters of an observation model?
	In other words, does~$\soluexp\latpv$ carry the desired information?
	For example, one could ask whether $\soluexp\latpv$ is contained in $\spann{\{ \trulatpv \}}$ when assuming a single-index model \eqref{eq:intro:sim}, or whether $\soluexp\latpv$ is supported on $\actsupp$ when $y$ obeys \eqref{eq:intro:vs}.
\item\label{quest:intro:complexity}
	\emph{Complexity.} What role is played by the hypothesis set $\sset$, especially in  high-dimen\-sion\-al problems? How to exploit low-complexity features of the underlying observation model, for instance, if $S \ll \d$ in variable selection \eqref{eq:intro:vs}?
\end{properties}

\subsection{The Mismatch Principle at a Glance}
\label{subsec:intro:mismatch}

In order to highlight the main ideas of our theoretical approach and to avoid unnecessary technicalities, let us assume in this subsection that \eqref{eq:intro:fvdecomp} holds true with $\atoms = \I{\d}$, where $\I{\d} \in \R^{\d \times \d}$ denotes the identity matrix in $\R^\d$.
This simplification of the data model implies that $\x = \lat$ is an \emph{isotropic} random vector in $\R^\d$, and adopting the terminology of the previous paragraphs, we particularly have $\solu\xpv = \solu\latpv$, $\soluexp\xpv = \soluexp\latpv$, and $\x_i = \lat_i$ for $i = 1, \dots, n$.
For the sake of consistency (with the definitions of Section~\ref{sec:results}), we agree on writing $\lat$, $\lat_i$, $\solu\latpv$, $\soluexp\latpv$ instead of $\x$, $\x_i$, $\solu\xpv$, $\soluexp\xpv$, respectively. The generalized Lasso \eqref{eq:intro:klasso} then reads as follows:\footnote{When referring to \eqref{eq:intro:klasso} as an \emph{estimator}, we actually mean a \emph{minimizer} of \eqref{eq:intro:klasso}.}
\begin{equation} \label{eq:intro:klassoiso}\tag{$\mathsf{P}_{\sset}^{\text{iso}}$}
	\min_{\latpv \in \sset} \ \tfrac{1}{n} \sum_{i = 1}^n (y_i - \sp{\lat_i}{\latpv})^2.
\end{equation}

A key quantity of our framework is the so-called \emph{mismatch covariance} which is defined as
\begin{equation}
	\modelcovar{\targetlatpv} \coloneqq \lnorm[\big]{\mean[(y - \sp{\lat}{\targetlatpv}) \lat]} \ , \qquad \targetlatpv \in \R^\d.
\end{equation}
The name `mismatch covariance' is due to the fact that $\modelcovar{\targetlatpv}$ basically measures the covariance between the input variables $\lat \in \R^\d$ and the \emph{mismatch} $y - \sp{\lat}{\targetlatpv}$ that arises from approximating the (possibly non-linear) output $y$ by a linear function $\func(\lat) \coloneqq \sp{\lat}{\targetlatpv}$.
To better understand the meaning of this statistical parameter, it is very helpful to take a closer look at the following equivalence:
\begin{equation} \label{eq:intro:orthogprinciple}
	\modelcovar{\targetlatpv} = 0 \qquad \iff \qquad \forall \latpv \in \R^\d \colon \ \mean[(y - \sp{\lat}{\targetlatpv}) \sp{\lat}{\latpv}] = 0.
\end{equation}
The right-hand side of \eqref{eq:intro:orthogprinciple} in fact corresponds to the fundamental \emph{orthogonality principle} in linear estimation theory.
According to its classical formulation, the orthogonality principle states that the prediction error of the optimal estimator is orthogonal (uncorrelated) to every possible linear estimator $\func(\lat) = \sp{\lat}{\latpv}$ with $\latpv \in \R^\d$ (cf. \cite[Sec.~12.4]{kay1993fundamentals}).
This argument becomes more plausible when computing the gradient of the expected risk at $\targetlatpv \in \R^\d$,
\begin{equation}
	\gradient_{\latpv = \targetlatpv} \ \mean[(y - \sp{\lat}{\latpv})^2] = -2 \mean[(y - \sp{\lat}{\targetlatpv}) \lat],
\end{equation}
which implies that \eqref{eq:intro:orthogprinciple} is equivalent to $\func = \sp{\cdot}{\targetlatpv}$ being a critical point of the objective function in \eqref{eq:intro:explossmin-general}.
Consequently, if the mismatch covariance vanishes at $\targetlatpv \in \sset$, then $\targetlatpv$ is an expected risk minimizer in the sense that $\targetlatpv = \soluexp\latpv$.

However, we emphasize that \eqref{eq:intro:orthogprinciple} is only a \emph{sufficient} condition for $\targetlatpv = \soluexp\latpv$, since the search space of \eqref{eq:intro:klassoiso} is restricted to a certain hypothesis set $\sset \subset \R^n$.
There indeed exist several relevant model setups sketched further below where the traditional orthogonality principle fails to work.
In this light, the mismatch covariance can be considered as a refined concept that quantifies how far one is from fulfilling the optimality condition of \eqref{eq:intro:orthogprinciple}. 
Let us make this idea more precise by stating an informal version of one of our main results (see Theorem~\ref{thm:results:bounds:global} and Corollary~\ref{cor:appl:nonlinear:global}):
\begin{theorem}[informal]\label{thm:intro:maininformal}
	Let $y$ be a sub-Gaussian variable and let $\lat$ be a centered isotropic sub-Gaussian random vector in $\R^\d$ (cf. Subsection~\ref{subsec:intro:notation}\ref{item:intro:notation:probability}).
	Assume that $\sset \subset \R^\d$ is a bounded, convex subset and fix a vector $\targetlatpv \in \sset$.
	If $n \geq C \cdot \effdim{\sset}$, then every minimizer $\solu\latpv$ of \eqref{eq:intro:klassoiso} satisfies with high probability
	\begin{equation}\label{eq:intro:maininformal:bound}
		\lnorm{\solu\latpv - \targetlatpv} \leq C' \cdot \Big[ \Big(\frac{\effdim{\sset}}{n}\Big)^{1/4} + \modelcovar{\targetlatpv} \Big],
	\end{equation}
	where $\meanwidth{\sset}$ denotes the Gaussian mean width of the hypothesis set $\sset$ (cf. Definition~\ref{def:results:parameters:meanwidth}), and $C, C' > 0$ are model-dependent constants detailed in Section~\ref{sec:results}.
\end{theorem}

At first sight, this result appears somewhat odd, as it states an error bound for every choice of the \emph{target vector} $\targetlatpv \in \sset$.
But the actual significance of \eqref{eq:intro:maininformal:bound} clearly depends on the size of the mismatch covariance $\modelcovar{\targetlatpv}$. 
In particular, if $\modelcovar{\targetlatpv} > 0$, we do not even obtain a consistent estimate of $\targetlatpv$, i.e., the approximation error does not tend to $0$ as $n \to \infty$.
Following the above reasoning, it is therefore quite natural to select $\targetlatpv = \soluexp\latpv$, which yields the best possible outcome from Theorem~\ref{thm:intro:maininformal} and reflects the common strategy of statistical learning.

While this viewpoint primarily addresses the sampling-related issues of \ref{quest:intro:estimation} and \ref{quest:intro:complexity}, the problem of interpretability \ref{quest:intro:approximation} remains untouched. 
A precise answer to \ref{quest:intro:approximation} in fact strongly depends on the interplay between the structure of the (semi-parametric) output variable $y$ and the corresponding expected risk minimizer $\soluexp\latpv$.
For example, Plan and Vershynin recently verified in \cite{plan2015lasso} that a single-index model \eqref{eq:intro:sim} can be consistently learned via the generalized Lasso \eqref{eq:intro:klassoiso}, as long as the input data is \emph{Gaussian}. Although not stated explicitly, their finding is underpinned by the orthogonality principle of \eqref{eq:intro:orthogprinciple}, in the sense that one can always achieve $\modelcovar{\scalfac\trulatpv} = 0$ for an appropriate scalar factor $\scalfac \in \R$;
see Subsection~\ref{subsec:appl:nonlinear:sim} for more details.
Interestingly, the situation is very different for \emph{non-Gaussian} single-index models, where $\soluexp\latpv$ does not necessarily belong to $\spann\{\trulatpv\}$ anymore; see \cite[Rmk.~1.5]{plan2013onebit} and \cite{ai2014onebitsubgauss}.
But fortunately, the versatility of Theorem~\ref{thm:intro:maininformal} pays off at this point because it still permits to select a vector $\targetlatpv \in \tset \coloneqq \spann\{\trulatpv\}$. 
Thus, if there exists $\targetlatpv \in \tset \intersec \sset$ such that $\modelcovar{\targetlatpv}$ becomes sufficiently small, \eqref{eq:intro:maininformal:bound} turns into a meaningful recovery guarantee for single-index models.
A similar strategy applies to the problem of variable selection \eqref{eq:intro:vs} by considering $\targetlatpv \in \tset \coloneqq \{\latpv \in \R^\d \suchthat \supp(\latpv) \subset \actsupp \}$; see Subsection~\ref{subsec:appl:nonlinear:mim}.

These prototypical examples demonstrate that semi-parametric models often come along with a certain \emph{target set} $\tset \subset \R^\d$ containing all those vectors which allow us to extract the parameters of interest.
In other words, one would be satisfied as soon as an estimation procedure approximates any vector $\targetlatpv \in \tset$.
Transferring this notion to the general setup of Theorem~\ref{thm:intro:maininformal}, we simply assume that there exists such a target set $\tset \subset \R^\d$, which encodes the desired (parametric) information about the data pair $(\lat, y)$. The meaning of the term `information' is of course highly application-specific in this context, for instance, it is also reasonable to ask for the support of a (sparse) index vector $\trulatpv$ in \eqref{eq:intro:sim}, instead of its direction.
Therefore, we will leave $\tset$ unspecified in the following and treat it as an abstract component of the underlying observation model.
Note that, compared to the hypothesis set $\sset$, the target set $\tset$ is \emph{unknown} in practice, and to some extent, it plays the role of the \emph{ground truth} parameters that one would like to estimate.

Combining the concept of target sets with Theorem~\ref{thm:intro:maininformal}, we can now formulate an informal version of the \emph{mismatch principle}:
\begin{quote}\itshape
	Specify a target vector $\targetlatpv \in \tset \intersec \sset$ such that the mismatch covariance $\modelcovar{\targetlatpv}$ is minimized. Then invoke Theorem~\ref{thm:intro:maininformal} to obtain an error estimate for $\targetlatpv$ by a minimizer of \eqref{eq:intro:klassoiso}.
\end{quote}
This simple recipe includes the key aspects of \ref{quest:intro:estimation}--\ref{quest:intro:complexity}:
By enforcing $\targetlatpv \in \tset \intersec \sset$, we ensure that the target vector is consistent with our (parametric) assumptions on the output variable $y$ and satisfies the model hypothesis imposed by $\sset$ at the same time (see Figure~\ref{fig:intro:mismatchprinciple}).
The error bound of \eqref{eq:intro:maininformal:bound} then quantifies the compatibility of these desiderata by means of the mismatch covariance $\modelcovar{\targetlatpv}$ and the Gaussian mean width $\meanwidth{\sset}$.
This does not only provide a theoretical guarantee for the generalized Lasso \eqref{eq:intro:klassoiso}, but also more insight into its capacity to solve (or not) a specific estimation problem.
\begin{figure}
	\centering
	\tikzstyle{blackdot}=[shape=circle,fill=black,minimum size=1mm,inner sep=0pt,outer sep=0pt]
		\begin{tikzpicture}[scale=2,extended line/.style={shorten >=-#1,shorten <=-#1},extended line/.default=1cm]]
			\coordinate (K1) at (0,-.2);
			\coordinate (K2) at (.85,-.85);
			\coordinate (K3) at (.5,-1.75);
			\coordinate (K4) at (-.5,-1.5);
			\coordinate (K5) at (-.5,-.75);
			\coordinate (P0) at (0,-1);
			\path[name path=K1--K2] (K1) -- (K2);
			\path[name path=P0--X0] (P0) -- ++(65:1);
			\path[name intersections={of=P0--X0 and K1--K2,by=X0}];
			\coordinate (targetlatpv) at ($(X0)!.2!(P0)$);
			\coordinate (SpanEnd) at ($(X0)!-.65!(P0)$);
			\draw[thick] (SpanEnd) -- ($(X0)!2.4!(P0)$);
			\draw[thick,dashed] ($(X0)!-1!(P0)$) -- ($(X0)!2.75!(P0)$);
						
			\draw[fill=gray!20!white,thick] (K1) -- (K2) -- (K3) -- (K4) -- (K5) -- cycle;
			\begin{scope}
				\clip (K1) -- (K2) -- (K3) -- (K4) -- (K5) -- cycle;
				\draw[ultra thick,red] (SpanEnd) -- ($(X0)!2.75!(P0)$);
			\end{scope}
			
			\node at (barycentric cs:K1=1,K2=1,K3=8,K4=1,K5=1) {$\sset$};
			\node[blackdot, label={[label distance=-4pt]300:$\soluexp\latpv$}] at ($(targetlatpv) + (.1,-.15)$) {};
			\node[label={[label distance=-3pt]right:$\tset$}] at ($(X0)!-.35!(P0)$) {};
			\node[blackdot,label={[label distance=-3pt]180:$\targetlatpv$}] at (targetlatpv) {};
			
			\node (intersecAnchor) at ($(X0)!1.5!(P0)$) {};
			\node[left=.9 of intersecAnchor] (intersecLabel) {$\sset \intersec \tset$};
			\path[<-,shorten <=1pt,>=stealth,red] (intersecAnchor) edge (intersecLabel) ;
		\end{tikzpicture}
	\caption{Visualization of the mismatch principle. The mismatch covariance $\modelcovar{\cdot}$ is minimized on the intersection $\sset \intersec \tset$ (red part). Note that the minimizer $\targetlatpv \in \R^\d$ is not necessarily equal to the expected risk minimizer $\soluexp\latpv \in \R^\d$, and the latter might not even belong to $\sset \intersec \tset$.}
	\label{fig:intro:mismatchprinciple}
\end{figure}

Let us finally point out that the classical learning theory would simply suggest to consider $\tset = \R^\d$ and $\targetlatpv = \soluexp\latpv$.
By contrast, we allow for a restriction of $\tset$ to a certain family of \emph{interpretable} models according to the demand of \ref{quest:intro:approximation}.
It is particularly important to note that the mismatch principle does \emph{not} rely on knowledge of $\soluexp\latpv$.
Hence, the above approach is conceptually very different from the ``naive'' idea of first explicitly computing $\soluexp\latpv$ and then finding the closest point in a given target set $\tset$; see Remark~\ref{rmk:results:outofbox} for a more detailed discussion of this aspect.


\subsection{Main Contributions and Overview}
\label{subsec:intro:overview}

The major purpose of this work is to shed more light on the capability of the generalized Lasso to deal with \emph{non-linear} observation models.
A deeper understanding of these types of estimators is of considerable practical relevance because they established themselves as a benchmark method in many different application areas, such as machine learning, signal- and image processing, econometrics, or bioinformatics.
Their popularity is especially due to the fact that --- compared to more sophisticated non-convex methods --- the associated least squares problem \eqref{eq:intro:klasso} does not require any prior knowledge about the inputs and it can be efficiently solved for various choices of $\sset$ via convex programming.

But although conceptually quite simple, the performance of \eqref{eq:intro:klasso} is not fully understood to this day.
A particular shortcoming of many theoretical approaches in the literature is that they either focus on very restrictive model settings, e.g., noisy linear regression, or in stark contrast, do not make any structural assumptions on $y$.
As already pointed out in the previous subsections, our primary goal is develop a framework that connects these two opposite viewpoints and allows us to treat a large class of estimation problems in a systematic way.
The following list highlights the main achievements of this article in that respect:

\begin{listing}
\item
	\emph{Generic error bounds.}
	In Section~\ref{sec:results}, we make the approach of Subsection~\ref{subsec:intro:mismatch} more precise and provide technical definitions of the so-called \emph{mismatch parameters} and the \emph{Gaussian mean width} (see Subsection~\ref{subsec:results:parameters}).
	Our first main result, Theorem~\ref{thm:results:bounds:global} in Subsection~\ref{subsec:results:bounds}, formalizes the statement of Theorem~\ref{thm:intro:maininformal} and also allows for arbitrary mixing matrices in \eqref{eq:intro:fvdecomp}, meaning that the features of $\x \in \R^\p$ may be (even perfectly) correlated.
	In the course of our second main result, Theorem~\ref{thm:results:bounds:conic} in Subsection~\ref{subsec:results:conic}, we then revisit \ref{quest:intro:complexity} and refine the global Gaussian mean width $\meanwidth{\sset}$ in the first error term of \eqref{eq:intro:maininformal:bound}.
	Indeed, by using a \emph{localized} complexity parameter, it is possible to achieve the optimal error decay rate of $\asympfaster{n^{-1/2}}$ in terms of $n$.
	This improvement is of particular importance to \emph{high-dimensional} problems where the dimension of the data $\p$ is significantly larger that the actual sample size $n$.
	The proofs of our main results are postponed to Section~\ref{sec:proofs} and rely on an even more general estimation guarantee for the generalized Lasso, Theorem~\ref{thm:proofs:results:local}, which employs the so-called \emph{local} Gaussian mean width (see Subsection~\ref{subsec:proofs:results}); note that the more advanced statement of Theorem~\ref{thm:proofs:results:local} is omitted in the main part, as this would make the presentation of the article unnecessarily technical.
	But it is worth pointing out that all established results are novel and of independent interest, extending and improving previously known performance guarantees for the generalized Lasso; see also Remark~\ref{rmk:results:outofbox} and Remark~\ref{rmk:proofs:local:techniques} for more detailed discussions.
\item
	\emph{The mismatch principle.}
	A major consequence of our estimation guarantees in Section~\ref{sec:results} is the mismatch principle, which is formalized in Subsection~\ref{subsec:results:bounds} (see Recipe~\ref{rec:results:mismatch}).
	Its versatility is then demonstrated in Subsection~\ref{subsec:appl:nonlinear}, where it is applied to various types of semi-parametric observation models, including single-index models, generalized linear models, multiple-index models, and superimposed observations.
	While this leads to a rediscovery of some previous results from the literature, we will derive several new guarantees as well, e.g., for non-Gaussian single-index models and variable selection.
	Let us emphasize that these findings are also relevant to recent advances in \emph{quantized} and \emph{distributed compressed sensing}; see Subsection~\ref{subsec:appl:nonlinear:1bit} and Subsection~\ref{subsec:appl:nonlinear:msense}, respectively.
\item
	\emph{Correlated inputs and noisy data.}
	The importance of correlated input variables was already mentioned in the course of the factor model \eqref{eq:intro:fvdecomp}.
	We will return to these issues in Section~\ref{subsec:appl:correlated} and show that estimation via \eqref{eq:intro:klasso} may be still practicable if the mixing matrix $\atoms$ is just approximately known; see Subsection~\ref{subsec:appl:correlated:undercomplete}. 
	Another scenario of interest is input data contaminated by noise, so that \eqref{eq:intro:fvdecomp} holds with $\p < \d$; see Subsection~\ref{subsec:appl:correlated:overcomplete}.
	The mismatch principle again proves useful in this case, as the mismatch covariance naturally incorporates the signal-to-noise ratio of the underlying sampling process; in particular, this is another problem instance where the classical orthogonality principle does not remain valid (cf. \eqref{eq:intro:orthogprinciple}).
\end{listing}

Regarding related literature, our approach is inspired by the recent work of Plan and Vershynin \cite{plan2015lasso}, who studied the generalized Lasso \eqref{eq:results:bounds:klasso} with Gaussian single-index observations. In a certain sense, our work is an attempt to lift their findings onto a more abstract (statisical-learning-based) setup, which goes beyond the case of single-index models.
From a more technical perspective, an important difference to \cite{plan2015lasso} is that our proofs rely on more sophisticated uniform bounds for empirical quadratic and multiplier processes due to \cite{mendelson2016multiplier,liaw2016randommat}, allowing us to improve and generalize their performance guarantees into various directions.
A more detailed discussion of the related literature is provided in Section~\ref{sec:literature}. This does not only concern recent results in learning theory (Subsection~\ref{subsec:literature:statlearn}) but also advances in signal processing and compressed sensing (Subsection~\ref{subsec:literature:cs}) as well as strongly correlated designs (Subsection~\ref{subsec:literature:correl}).
Our concluding remarks and potential future research directions are presented in Section~\ref{sec:conlcusion}.
In fact, the findings of this work are amenable to many possible extensions and generalizations, such as general \mbox{(non-)}convex hypothesis classes, different types of loss functions, or heavy-tailed data.

Let us conclude this subsection with a brief clarification of the scope of our achievements:
\begin{remark}
	First, we wish to emphasize that the above mentioned contributions are not only of technical but also of conceptual nature.
	In this regard, the mismatch principle should be understood as a simple recipe to establish off-the-shelf guarantees for the generalized Lasso \eqref{eq:intro:klasso}.
	But despite this clear theoretical focus, our approach may have some interesting practical implications as well:
	While carefully tailored methods certainly perform better in specific model situations, our results provide evidence of why the ``naive'' strategy of \eqref{eq:intro:klasso} is still quite competitive. This particularly justifies why Lasso-type estimators do often yield a good initial guess of the true parameters, which can serve as initialization for more sophisticated algorithms.
	
	Although the formal setup of our approach fits well into the framework of statistical learning, it is helpful to bear in mind the following conceptual difference:
	Our goal is to explore what empirical risk minimization with \emph{linear} hypothesis functions (i.e., the generalized Lasso \eqref{eq:intro:klasso}) can learn about more complicated, possibly non-linear observation models. This type of misspecification precisely translates into the parameter estimation problem outlined in Subsection~\ref{subsec:intro:mismatch}.
	A major concern of traditional statistical learning, on the other hand, is to accurately \emph{predict} the value of an output variable. 
	In the case of non-linear models, it would be therefore more natural to select a hypothesis class consisting of appropriate \emph{non-linear} functions and then to investigate the predictive capacity of the resulting empirical risk minimization problem.
	This does also explain the important role of the prediction error in learning theory (see \eqref{eq:literature:statlearn:prederr}), which in turn is only of minor interest to our purposes.
\end{remark}

\subsection{Notations and Preliminaries}
\label{subsec:intro:notation}

Before proceeding with the main results, let us fix some notations and conventions that are frequently used in this work:
\begin{rmklist}
\item
	Vectors and matrices are denoted by lower- and uppercase boldface symbols, respectively. Their entries are indicated by subscript indices and lowercase letters, e.g., $\v = (v_1, \dots, v_\d) \in \R^\d$ for a vector and $\vec{M} = [m_{k,k'}] \in \R^{\d\times \d'}$ for a matrix.
	The (horizontal) \emph{concatenation} of $\vec{M} \in \R^{\d\times \d'}$ and $\tilde{\vec{M}} \in \R^{\d\times \d''}$ is denoted by $[\vec{M},\tilde{\vec{M}}] \in \R^{\d \times (\d' + \d'')}$. The \emph{identity matrix} in $\R^\d$ is $\I{\d} \in \R^{\d\times \d}$ and $\vunit_k \in \R^\d$ denotes the $k$-th \emph{unit vector}.
\item
	The \emph{support} of a vector $\v \in \R^\d$ is the index set of its non-zero components, i.e., $\supp(\v) \coloneqq \{ k \in \{1, \dots, \d\} \suchthat v_k \neq 0\}$, and we set $\lnorm{\v}[0] \coloneqq \cardinality{\supp(\v)}$. In particular, $\v$ is called \emph{$s$-sparse} if $\lnorm{\v}[0] \leq s$.
	For $q\geq 1$, the \emph{$\l{q}$-norm} of $\v$ is given by 
	\begin{equation}
		\lnorm{\v}[q] \coloneqq\begin{cases} 
		(\sum_{k=1}^\d \abs{v_k}^q)^{1/q}, &q < \infty \ ,\\
		\displaystyle\max_{k = 1, \dots, \d} \ \abs{v_k}, & q = \infty \ .
		\end{cases}  
	\end{equation}
	The associated \emph{unit ball} is denoted by $\ball[q][\d]\coloneqq\{\v \in \R^\d \suchthat \lnorm{\v}[q] \leq 1\}$ and the \emph{Euclidean unit sphere} is $\S^{\d-1} \coloneqq \{\v\in \R^\d \suchthat \lnorm{\v}=1\}$.
	The \emph{operator norm} of $\vec{M} \in \R^{\d \times \d'}$ is $\opnorm{\vec{M}} \coloneqq \sup_{\v' \in \S^{\d'-1}} \lnorm{\vec{M}\v'}$.
\item 
	Let $\ssetalt \subset \R^\d$ and $\v \in \R^\d$.
	The \emph{linear hull} of $\ssetalt$ is denoted by $\spann\ssetalt$, and its \emph{conic hull} is defined as $\cone{\ssetalt} \coloneqq \{\tau\v \suchthat \v \in \ssetalt,\ \tau \geq 0\}$.
	Moreover, we write $\ssetalt \pm \v \coloneqq \{ \tilde{\v} \pm \v \suchthat \tilde{\v} \in \ssetalt \}$.
	
	If $E \subset \R^\d$ is a linear subspace, the associated \emph{orthogonal projection} onto $E$ is denoted by $\proj_E \in \R^{\d \times \d}$, 
	and we write $\orthcompl{\proj}_{E} \coloneqq \I{\d} - \proj_E$ for the projection onto the orthogonal complement $\orthcompl{E} \subset \R^\d$. Moreover, if $E = \spann\{\v\}$, we use the short notations $\proj_{\v} \coloneqq \proj_{E}$ and $\orthcompl{\proj}_{\v} \coloneqq \orthcompl{\proj}_{E}$.
\item\label{item:intro:notation:probability}
	The \emph{expected value} is denoted by $\mean[\cdot]$ and we sometimes use a subscript to indicate that the expectation (integral) is only computed with respect to a certain random variable.
	Let $v$ be a real-valued random variable. Then $v$ is \emph{sub-Gaussian} if 
	\begin{equation}\label{eq:intro:subgnorm}
		\normsubg{v} \coloneqq \sup_{q \geq 1} q^{-1/2} (\mean[\abs{v}^q])^{1/q} < \infty \ ,
	\end{equation}
	and $\normsubg{\cdot}$ is called the \emph{sub-Gaussian norm} (cf. \cite{vershynin2012random}). The \emph{sub-Gaussian norm} of a random vector $\v$ in $\R^\d$ is then given by $\normsubg{\v} \coloneqq \sup_{\vec{u} \in \S^{\d-1}} \normsubg{\sp{\v}{\vec{u}}}$.
	Moreover, $\v$ is called \emph{isotropic} if $\mean[\v \v^\T] = \I{\d}$, or equivalently, if
	\begin{equation}\label{eq:intro:isotr}
		\mean[\sp{\v}{\vec{u}} \sp{\v}{\vec{u}'}] = \sp{\vec{u}}{\vec{u}'} \quad \text{for all $\vec{u}, \vec{u}' \in \R^\d$.}
	\end{equation}
	If $\tilde{v}$ is another random variable (not necessarily real-valued), we denote the expectation of $v$ conditional on $\tilde{v}$ by $\mean[v \suchthat \tilde{v}]$.
	Finally, we write $\v \distributed \Normdistr{\vnull}{\Covmatr}$ if $\v$ is a (centered) \emph{Gaussian random vector} with covariance matrix $\Covmatr \in \R^{\d \times \d}$.
\item
	The letter $C$ is reserved for a (generic) constant, whose value could change from time to time.
	We refer to $C > 0$ as a \emph{numerical constant} if its value does not depend on any other involved parameter.
	If an \mbox{(in-)equality} holds true up to a numerical constant $C$, we may simply write $A \lesssim B$ instead of $A \leq C \cdot B$, and if $C_1\cdot A\leq B \leq C_2 \cdot A$ for numerical constants $C_1,C_2 > 0$, we use the abbreviation $A\asymp B$.
\end{rmklist}

\section{Main Results}
\label{sec:results}

This part builds upon the key ideas from Subsection~\ref{subsec:intro:mismatch} and derives a formal version of the mismatch principle (Recipe~\ref{rec:results:mismatch}).
Starting with a definition of our model setup in Subsection~\ref{subsec:results:setup} and several technical parameters in Subsection~\ref{subsec:results:parameters}, our main results are presented in Subsection~\ref{subsec:results:bounds} (see Theorem~\ref{thm:results:bounds:global}) and Subsection~\ref{subsec:results:conic} (see Theorem~\ref{thm:results:bounds:conic}), respectively.
Note that this section rather takes an abstract viewpoint, whereas  applications and examples of our framework are elaborated in Section~\ref{sec:appl}.

\subsection{Formal Model Setup}
\label{subsec:results:setup}

Let us begin with a rigorous definition of the random sampling process considered in the introduction.
For the sake of convenience, we have summarized the our terminology in Table~\ref{tab:results:terminology}.
\begin{assumption}[Sampling process]\label{model:results:observations}
	Let $(\lat, y)$ be a joint random pair in $\R^\d \times \R$, where $y$ is a sub-Gaussian random variable and $\lat$ is a centered isotropic sub-Gaussian random vector in $\R^\d$ with $\normsubg{\lat} \leq \subgparam$ for some $\subgparam > 0$.
	Moreover, let $\atoms \in \R^{\p \times \d}$ be a (deterministic) matrix and define the input random vector in $\R^\p$ as $\x \coloneqq \atoms \lat$.
	The \emph{observed samples} $\{(\x_i, y_i)\}_{i = 1}^n$ are now generated as follows: Let $\{(\lat_i, y_i)\}_{i = 1}^n$ be independent copies of $(\lat, y)$. Then we set $\x_i \coloneqq \atoms \lat_i$ for $i = 1, \dots, n$. In particular, $\{(\x_i, y_i)\}_{i = 1}^n$ can be also regarded as independent samples of $(\x,y)$.
\end{assumption}

\begin{table}
\renewcommand{\arraystretch}{1.2}
\centering
\begin{tabularx}{\textwidth}{>{\raggedright}m{6cm}|m{4cm}|m{4cm}}
	\textbf{Term} & \textbf{Notation} & \textbf{Sampling notation} \\
 	\hline\hline
	{Output variable {\\ \smaller (dependent variable, response variable, label)}} & $y \in \R$ & $y_1, \dots, y_n$ \\ \hline
	{Latent variables {\\ \smaller (hidden variables, factors)}} & $\lat \in \R^\d$ & $\lat_1, \dots, \lat_n$ \\ \hline
	Mixing matrix {\\ \smaller (pattern matrix, factor loadings)} & $\atoms \in \R^{\p \times \d}$ & \\ \hline
	{Input variables {\\ \smaller (independent variables, explanatory variables, data variables, features, covariates, predictors)}} & $\x = \atoms \lat \in \R^\p$ & $\x_1 = \atoms\lat_1, \dots, \x_n = \atoms\lat_n$ \\ \hline
	{Observation {\\ \smaller (data, sample pair)}} & $(\x, y) = (\atoms\lat, y) \in \R^\p \times \R$ & $(\x_1, y_1), \dots, (\x_n, y_n)$ \\ \hline
	{Hypothesis set {\\ \smaller (constraint set, hypothesis class, models)}} & $\sset \subset \R^\p$ convex &  \\ \hline
	{Target set} & $\tset \subset \R^\d$ &  \\ \hline
	{(Latent) Parameter vectors} & $\xpv \in \sset$, \ $\latpv \in \atoms^\T \sset \subset \R^\d$ & 
\end{tabularx}
\caption{A summary of frequently used notations. The terms in parentheses are widely-used synonyms.}
\label{tab:results:terminology}
\end{table}

As usual in statistical learning, Assumption~\ref{model:results:observations} leaves the output variable $y$ largely unspecified and only makes a regularity assumption, namely that the tail of $y$ is sub-Gaussian.
But it is still useful to keep in mind our running examples from \eqref{eq:intro:sim} and \eqref{eq:intro:vs}, which indicate that $y$ typically depends on the latent factors $\lat$ according to a certain \mbox{(semi-)}parametric model.
Moreover, Assumption~\ref{model:results:observations} builds upon a simple linear factor model for the input variables, i.e., $\x = \atoms \lat$ with $\lat$ being isotropic.
The following proposition shows that this is actually not a severe restriction.
\begin{proposition}\label{prop:results:setup:modeldecomp}
	Let $\x$ be a centered random vector in $\R^\p$ whose covariance matrix is of rank $\d$. Then there exists a (non-unique) deterministic matrix $\atoms \in \R^{\p \times \d}$ and a centered isotropic random vector $\lat$ in $\R^\d$ such that $\x = \atoms \lat$ almost surely.
\end{proposition}
A proof of Proposition~\ref{prop:results:setup:modeldecomp} is given in Subsection~\ref{subsec:proofs:modeldecomp}.
Nevertheless, this statement of existence is only of limited practical relevance because the resulting latent factors are not necessarily linked to $y$ in a meaningful way. Indeed, while the above isotropic decomposition of $\x$ is highly non-unique, the actual output variable may exhibit a very specific structure, such as in \eqref{eq:intro:sim} or \eqref{eq:intro:vs}.
We therefore rather take the viewpoint that the factorization $\x = \atoms \lat$ is prespecified, say by a certain data acquisition process or a physical law.
In particular, our purpose is \emph{not} to learn the entries of $\atoms$, but to consider them as possibly unknown model parameters (cf. Subsection~\ref{subsec:literature:correl}).
Finally, let us point out that the situation of $\p < \d$ can be also of interest, e.g., when modeling noisy data as in Subsection~\ref{subsec:appl:correlated:overcomplete}.

\subsection{The Mismatch Parameters and Gaussian Mean Width}
\label{subsec:results:parameters}

We now introduce the technical ingredients of our general error bounds below.
Let us begin with a formal definition of the mismatch parameters, which are supposed to quantify the model mismatch that occurs when approximating the true output $y$ by a linear hypothesis function of the form $\func(\lat) \coloneqq \sp{\lat}{\targetlatpv}$.
\begin{definition}[Mismatch parameters]\label{def:results:parameters:mismatch}
	Let Assumption~\ref{model:results:observations} be satisfied and let $\targetlatpv \in \R^\d$.
	Then we define the following two parameters (as a function of $\targetlatpv$):
	\begin{alignat}{2}
		&\text{\emph{Mismatch covariance:}} && \qquad \modelcovar{\targetlatpv} \coloneqq \modelcovar[\lat,y]{\targetlatpv} \coloneqq \lnorm[\big]{\mean[(y - \sp{\lat}{\targetlatpv}) \lat]} \ , \\*[.25\baselineskip]
		&\text{\emph{Mismatch deviation:}} && \qquad \modeldev{\targetlatpv} \coloneqq \modeldev[\lat,y]{\targetlatpv} \coloneqq \normsubg{y - \sp{\lat}{\targetlatpv}} \ .
	\end{alignat}
\end{definition}
The meaning of the mismatch covariance was already discussed in the course of Subsection~\ref{subsec:intro:mismatch}: it is useful to regard $\modelcovar{\targetlatpv}$ as a measure of how close $\targetlatpv$ is to satisfying the optimality condition of \eqref{eq:intro:orthogprinciple}, which corresponds to the classical orthogonality principle.
By contrast, the mismatch deviation of $\targetlatpv$ captures the sub-Gaussian tail behavior of the actual model mismatch $y - \sp{\lat}{\targetlatpv}$, and thus, it measures how strongly $y$ deviates from the linear output model $\sp{\lat}{\targetlatpv}$.
The meaning of this parameter is best illustrated by the case of noisy linear observations, i.e., $y = \sp{\lat}{\targetlatpv} + e$, so that $\modeldev{\targetlatpv} = \normsubg{e}$ would basically measure the power of noise.
Note that Definition~\ref{def:results:parameters:mismatch} does only rely on the latent factors $\lat$, but not on the mixing matrix $\atoms$.
The role of $\atoms$ will become clear in Subsection~\ref{subsec:results:bounds} below when combining the concepts of target and hypothesis sets.

Our next definition introduces the (Gaussian) mean width, which is a widely-used notion of complexity for hypothesis sets (cf. \ref{quest:intro:complexity} at the end of Subsection~\ref{subsec:intro:statlearn}).
\begin{definition}[Mean width]\label{def:results:parameters:meanwidth}
	Let $\ssetalt \subset \R^\d$ be a subset and let $\gaussian \distributed \Normdistr{\vec{0}}{\I{\d}}$ be a standard Gaussian random vector.
	The \emph{(global) mean width} of $\ssetalt$ is given by
	\begin{equation}
		\meanwidth{\ssetalt} \coloneqq \mean\Big[\sup_{\h \in \ssetalt} \sp{\gaussian}{\h}\Big].
	\end{equation}
	Moreover, we define the \emph{conic mean width} of $\ssetalt$ as
	\begin{equation}
		\meanwidth[\conic]{\ssetalt} \coloneqq \meanwidth{\cone{\ssetalt} \intersec \S^{\d-1}}.
	\end{equation}
\end{definition}
These parameters originate from geometric functional analysis and convex geometry (e.g., see \cite{gordon1985gaussian,gordon1988escape,giannopoulos2004asymptotic}), but they do also appear in equivalent forms as \emph{Talagrand’s $\gamma_2$-functional} in stochastic processes (cf. \cite{talagrand2014chaining}) or as \emph{Gaussian complexity} in learning theory (cf. \cite{bartlett2003complexity}).
The benefit of the Gaussian mean with in the statistical analysis of various signal estimation problems then emerged more recently, e.g., see \cite{mendelson2007reconstruction,rudelson2008sparse,stojnic2009gordon,chandrasekaran2012geometry,plan2013robust,amelunxen2014edge,tropp2014convex,vershynin2014estimation,oymak2016sharpmse}. We refer the reader to these works for a more extensive discussion and basic properties.
In what follows below, we will treat the (conic) mean width as an abstract quantity that measures the size of the hypothesis set $\sset$ in \eqref{eq:intro:klasso} and thereby determines the approximation rate of our error bounds.
Some concrete examples of hypothesis sets (related to sparsity priors) are presented in Subsection~\ref{subsec:appl:nonlinear:sset}.
Finally, it is worth mentioning that the conic mean width is essentially equivalent to the notion of \emph{statistical dimension}, which is well known for characterizing the phase transition behavior of many convex programs with Gaussian input data; see \cite{amelunxen2014edge} and particularly Proposition~10.2 therein.

\subsection{General Error Bounds and the Mismatch Principle}
\label{subsec:results:bounds}

For the sake of convenience, let us restate the generalized Lasso estimator which is in the focus of our theoretical study:
\begin{equation} \label{eq:results:bounds:klasso}\tag{$\mathsf{P}_{\sset}$}
	\min_{\xpv \in \sset} \ \tfrac{1}{n} \sum_{i = 1}^n (y_i - \sp{\x_i}{\xpv})^2.
\end{equation}
\newcommand{\refklasso}[1]{(\hyperref[eq:results:bounds:klasso]{$\mathsf{P}_{#1}$})}%
Our first error bound for \eqref{eq:results:bounds:klasso} is based on the global mean width and extends Theorem~\ref{thm:intro:maininformal} with regard to arbitrary mixing matrices. Its proof is contained in Subsection~\ref{subsec:proofs:results} and is a consequence of a more general result of independent interest (Theorem~\ref{thm:proofs:results:local}), which employs the concept of local mean width.
\begin{theorem}[Estimation via \eqref{eq:results:bounds:klasso} -- Global version]\label{thm:results:bounds:global}
	Let Assumption~\ref{model:results:observations} be satisfied.
	Let $\sset \subset \R^\p$ be a bounded, convex subset and fix a vector $\targetlatpv \in \atoms^\T\sset \subset \R^\d$.
	Then there exists a numerical constant $C > 0$ such that for every $\delta \in \intvopcl{0}{1}$, the following holds true with probability at least $1 - 5\exp(- C \cdot \subgparam^{-4} \cdot \delta^2 \cdot n)$:\footnote{More precisely, the probability refers to the samples $\{(\x_i, y_i)\}_{i = 1}^n$, and we condition on this random set in the assertion of Theorem~\ref{thm:results:bounds:global}.}
	If the number of observed samples obeys
	\begin{equation}\label{eq:results:bounds:global:meas}
		n \gtrsim \subgparam^4 \cdot \delta^{-4} \cdot \effdim{\atoms^\T\sset},
	\end{equation}
	then every minimizer $\solu\xpv$ of \eqref{eq:results:bounds:klasso} satisfies
	\begin{equation}\label{eq:results:bounds:global:bound}
		\lnorm{\atoms^\T\solu\xpv - \targetlatpv} \lesssim \max\{1, \subgparam \cdot \modeldev{\targetlatpv}\} \cdot \delta + \modelcovar{\targetlatpv}.
	\end{equation}
\end{theorem}

The parameter $\delta$ can be regarded as an oversampling factor in the sense that it enables a trade-off between the number of required samples in \eqref{eq:results:bounds:global:meas} and the accuracy of \eqref{eq:results:bounds:global:bound}. 
Adjusting $\delta$ such that \eqref{eq:results:bounds:global:meas} holds true with equality, we can rephrase \eqref{eq:results:bounds:global:bound} as a more convenient error bound that explicitly depends on $n$:
\begin{equation}\label{eq:results:bounds:global:bound-decay}
	\lnorm{\atoms^\T\solu\xpv - \targetlatpv} \lesssim \max\{\subgparam, \subgparam^2 \cdot \modeldev{\targetlatpv}\} \cdot \Big(\frac{\effdim{\atoms^\T\sset}}{n}\Big)^{1/4} + \modelcovar{\targetlatpv}.
\end{equation}
This expression refines the statement of Theorem~\ref{thm:intro:maininformal} and reveals that first term in \eqref{eq:results:bounds:global:bound} is actually of the order $\asympfaster{n^{-1/4}}$. From this, we can particularly conclude that the first error term does not concern the consistency of the estimator, but rather controls its \emph{variance}. The mismatch covariance $\modelcovar{\targetlatpv}$, in contrast, can be interpreted as an \emph{(asymptotic) bias}, which explains why its size plays a crucial role when selecting an appropriate target vector $\targetlatpv$; see also Remark~\ref{rmk:results:outofbox}.
Remarkably, the individual parameters of \eqref{eq:results:bounds:global:bound-decay} are assigned to very different roles: while the mean width $\meanwidth{\atoms^\T\sset}$ only involves the (transformed) hypothesis set $\atoms^\T\sset$, the mismatch parameters $\modelcovar{\targetlatpv}$ and $\modeldev{\targetlatpv}$ do exclusively consider the output variable $y$ and the choice of target vector $\targetlatpv$. Apart from that, $\subgparam$ can be regarded as a model constant which bounds the sub-Gaussian tail of the input vector $\lat$.

Before further discussing the significance of the error bound \eqref{eq:results:bounds:global:bound} and deriving the mismatch principle (Recipe~\ref{rec:results:mismatch}), let us briefly investigate how the mixing matrix $\atoms$ affects the assertion of Theorem~\ref{thm:results:bounds:global}.
For this purpose, we first make the following important yet simple observation about the solution set of the generalized Lasso \eqref{eq:results:bounds:klasso}:
\begin{equation}\label{eq:results:bounds:transformklasso}
	\atoms^\T \cdot \Big(\argmin_{\xpv \in \sset} \ \tfrac{1}{n} \sum_{i = 1}^n (y_i - \sp{\x_i}{\xpv})^2 \Big) = \argmin_{\latpv \in \atoms^\T\sset} \ \tfrac{1}{n} \sum_{i = 1}^n (y_i - \sp{\lat_i}{\latpv})^2,
\end{equation}
which follows from the identity $\sp{\x_i}{\xpv} = \sp{\lat_i}{\atoms^\T\xpv}$ and substituting $\atoms^\T\xpv$ by $\latpv$.
As already pointed out in the course of \eqref{eq:intro:estimerror}, solving \eqref{eq:results:bounds:klasso} may lead to a highly non-unique minimizer, but by a linear transformation according to \eqref{eq:results:bounds:transformklasso}, it becomes related to an optimization problem that just takes \emph{isotropic} inputs $\lat_1, \dots, \lat_n \in \R^\d$.
This particularly explains why Theorem~\ref{thm:results:bounds:global} establishes an error bound on a parameter vector $\targetlatpv$ in the ``latent space'' $\R^\d$ and why the estimator actually takes the form $\solu\latpv \coloneqq \atoms^\T \solu\xpv$.
Moreover, \eqref{eq:results:bounds:transformklasso} indicates that, in order to better understand the estimation performance of \eqref{eq:results:bounds:klasso}, it essentially suffices to study
\begin{equation}\label{eq:results:bounds:klassoiso}
	\min_{\latpv \in \atoms^\T\sset} \ \tfrac{1}{n} \sum_{i = 1}^n (y_i - \sp{\lat_i}{\latpv})^2.
\end{equation}
The isotropy of data in \eqref{eq:results:bounds:klassoiso} indeed simplifies the statistical analysis, whereas the intricacy of the mixing matrix is now concealed by the transformed hypothesis set $\atoms^\T \sset$.
For that specific reason, the complexity of $\sset$ is measured by means of $\meanwidth{\atoms^\T\sset}$ in \eqref{eq:results:bounds:global:meas}, rather than $\meanwidth{\sset}$.

However, the original estimator \eqref{eq:results:bounds:klasso} is still of great practical interest because it neither requires knowledge of $\atoms$ nor of the $\lat_i$.
Even if $\atoms$ would be exactly known, an extraction of the latent factors can be difficult and unstable, implying that potential errors would propagate to the inputs of \eqref{eq:results:bounds:klassoiso}. 
Instead, \eqref{eq:results:bounds:transformklasso} suggests to first solve the ``unspoiled'' Lasso problem \eqref{eq:results:bounds:klasso} and then to appropriately reweight the outcome by $\atoms^\T$ as a post-processing step.
Our discussion on the mixing matrix is continued in Subsection~\ref{subsec:appl:correlated} and Subsection~\ref{subsec:literature:correl}, presenting several implications of Theorem~\ref{thm:results:bounds:global} and related approaches.
In this context, we demonstrate that estimation via \eqref{eq:results:bounds:klasso} is still feasible when $\atoms$ is only approximately known, and by carefully adapting the hypothesis set, one can even obtain error bounds for the untransformed minimizer $\solu\xpv$.

\subsubsection*{The Mismatch Principle}

Similar to the informal statement of Theorem~\ref{thm:intro:maininformal}, a key feature of Theorem~\ref{thm:results:bounds:global} is that it permits to freely select a target vector $\targetlatpv \in \atoms^\T\sset$.
This flexibility is an essential advantage over the traditional viewpoint of statistical learning theory, which would simply recommend to set $\targetlatpv = \soluexp\latpv \coloneqq \atoms^\T \soluexp\xpv$, where $\soluexp\func = \sp{\cdot}{\soluexp\xpv}$ is an expected risk minimizer of \eqref{eq:intro:explossmin-general}, see also \ref{quest:intro:estimation} at the end of Subsection~\ref{subsec:intro:statlearn}.
Indeed, according to \ref{quest:intro:approximation} and Subsection~\ref{subsec:intro:mismatch}, learning a semi-parametric model does also require a certain interpretability of the target vector, in the sense that it encodes the parameters of interest, e.g., an index vector (cf. \eqref{eq:intro:sim}) or the set of active variables (cf. \eqref{eq:intro:vs}).
For that purpose, we have introduced the abstract notion of a \emph{target set} $\tset \subset \R^\d$, which contains all those vectors that a user would consider as desirable outcome of an estimation procedure.\footnote{Note that the target set is part of the ``latent space'' $\R^\d$, which is due to the fact that we state semi-parameters output models for $y$ in terms of the latent factors $\lat \in \R^\d$ and not in terms of the correlated data $\x = \atoms\lat \in \R^\p$.}
It is therefore more natural to ask whether the generalized Lasso \eqref{eq:results:bounds:klasso} is capable of estimating any vector in $\tset$, rather than $\soluexp\latpv = \atoms^\T \soluexp\xpv$.
This important concern is precisely addressed by the mismatch principle:
\begin{recipe}[Mismatch principle]\label{rec:results:mismatch}
	Let Assumption~\ref{model:results:observations} be satisfied.	
	\begin{properties}[4em]{MP}
	\item\label{rec:results:mismatch:tset}
		\emph{Target set.} Select a subset $\tset \subset \R^\d$ containing all admissible target vectors associated with the model $(\lat,y)$.
	\item\label{rec:results:mismatch:sset}
		\emph{Hypothesis set.} Select a subset $\sset \subset \R^\p$ imposing structural constraints on the search space of \eqref{eq:results:bounds:klasso}.
	\item\label{rec:results:mismatch:modelcovar}
		\emph{Target vector.} Minimize the mismatch covariance on $\tset \intersec \atoms^\T \sset$, i.e., specify a minimizer $\targetlatpv$ of
		\begin{equation}\label{eq:results:mismatch:modelcovar:minimization}
			\min_{\latpv \in \tset \intersec \atoms^\T \sset} \modelcovar{\latpv}.
		\end{equation}
	\item\label{rec:results:mismatch:bound}
		\emph{Estimation performance.} Invoke \eqref{eq:results:bounds:global:bound} in Theorem~\ref{thm:results:bounds:global} to obtain a bound on the estimation error of $\targetlatpv$ via \eqref{eq:results:bounds:klasso}.
	\end{properties}
\end{recipe}
While the steps \ref{rec:results:mismatch:tset} and \ref{rec:results:mismatch:sset} are related to the underlying model specifications, the actual key step is \ref{rec:results:mismatch:modelcovar}: By restricting the objective set to $\tset \intersec \atoms^\T \sset$, we ensure that the target vector $\targetlatpv$ carries the desired (parametric) information and is compatible with the (transformed) hypothesis set $\atoms^\T \sset$ of the optimization problem \eqref{eq:results:bounds:klassoiso}.
Under these constraints every minimizer of the mismatch covariance then leads to the smallest possible bias term in \eqref{eq:results:bounds:global:bound}.
Meanwhile, the sampling-related term in \eqref{eq:results:bounds:global:bound} is only slightly affected by the size of the mismatch deviation $\modeldev{\targetlatpv}$.
In other words, the mismatch principle simply selects an admissible target vector in a such way that Theorem~\ref{thm:results:bounds:global} yields the best outcome.
Studying the significance of the resulting error bound eventually allows us to draw a conclusion on the ability of the generalized Lasso \eqref{eq:results:bounds:klasso} to solve the examined estimation problem or not.
We close this subsection with a useful simplification of Recipe~\ref{rec:results:mismatch} that is applicable whenever one has $\tset \subset \cone{\atoms^\T\sset}$:
\begin{remark}[How to apply the mismatch principle?]\leavevmode
\begin{rmklist}
\item\label{rmk:results:bounds:mismatchpractical}
	Invoking Recipe~\ref{rec:results:mismatch}\ref{rec:results:mismatch:modelcovar} can be a challenging step because $\modelcovar{\cdot}$ needs to be minimized on the intersection $\tset \intersec \atoms^\T \sset$.
	On the other hand, the hypothesis set $\sset$ can be freely chosen in practice.
	Thus, as long as $\tset \subset \cone{\atoms^\T\sset}$, one can simplify the mismatch principle as follows:
	\begin{quote}\itshape	
		First minimize the mismatch covariance on the target set $\tset$, i.e., set $\targetlatpv = \argmin_{\latpv \in \tset} \modelcovar{\latpv}$. Then specify a scaling factor $\lambda > 0$ for the hypothesis set $\sset$ such that $\targetlatpv \in \atoms^\T(\lambda\sset)$.
	\end{quote}
	The target vector $\targetlatpv$ is now admissible for step \ref{rec:results:mismatch:modelcovar}, when using $\lambda\sset$ as hypothesis set.
	Consequently, \eqref{eq:results:bounds:global:bound} states an error bound for the rescaled Lasso estimator \refklasso{\lambda\sset}, while the number of required samples is only enlarged by a quadratic factor of $\lambda^2$.
	To some extent, the above strategy allows us to decouple the tasks of minimizing the mismatch covariance (on $\tset$) and selecting an appropriate hypothesis set. This simplification will turn out to be very beneficial in Subsection~\ref{subsec:appl:nonlinear}, where we apply the mismatch principle to various types of output models but leave $\sset$ unspecified. 
	
	However, let us emphasize that the choice of $\sset$ and the scaling parameter $\lambda$ both remain important practical issues.
	In view of our approach, it is helpful to consider the hypothesis set as a means to exploit additional structural knowledge about the observation model.
	A typical example is a \emph{sparse} single-index model, for which the index vector $\trulatpv \in \R^\d$ in \eqref{eq:intro:sim} is unknown but its support is small. In this case, the uncertainty about $\trulatpv$ would be just captured by $\tset \coloneqq \spann\{\trulatpv\}$, whereas an $\l{1}$-constraint could serve as sparsity prior (cf. Subsection~\ref{subsec:appl:nonlinear:sset}).
	But nevertheless, finding a suitable value for $\lambda$ is not straightforward in practice and might require careful tuning of the estimator, since the target vector $\targetlatpv$ is still unknown in general.
\item \label{rmk:results:bounds:mpgeneric}
	Let us point out that the step \ref{rec:results:mismatch:bound} in Recipe~\ref{rec:results:mismatch} is not specifically tailored to Theorem~\ref{thm:results:bounds:global}, but one could also use the improved guarantees of Theorem~\ref{thm:results:bounds:conic} and Theorem~\ref{thm:proofs:results:local} instead (see below and Subsection~\ref{subsec:proofs:results}). In that sense, the mismatch principle is a \emph{generic} recipe describing general instructions to obtain error bounds for \eqref{eq:results:bounds:klasso}. In particular, it brings the concept of target sets into play --- which plays no role in Theorem~\ref{thm:results:bounds:global} per se --- thereby leading to \emph{interpretable} performance guarantees. On the other hand, let us emphasize once again that, since the target set is unknown in practice, the mismatch principle is rather a \emph{theoretical} tool that is intended for the statistical analysis of the generalized Lasso with semi-parametric observation models. \qedhere
\end{rmklist} \label{rmk:results:bounds}
\end{remark}

\subsection{Refined Error Bounds Based on the Conic Mean Width}
\label{subsec:results:conic}

The above discussion on Theorem~\ref{thm:results:bounds:global} indicates that the hypothesis set $\sset$ does only play a minor role for the estimation performance of the generalized Lasso \eqref{eq:results:bounds:klasso}, see also \ref{quest:intro:complexity} at the end of Subsection~\ref{subsec:intro:statlearn}. 
This fallacious impression is especially due to the fact that the complexity parameter $\meanwidth{\atoms^\T\sset}$ in \eqref{eq:results:bounds:global:meas} does not explicitly depend on $\targetlatpv$. The proposed sample size therefore remains widely unaffected by the choice of the target vector, but on the other hand, the involved oversampling factor of $\delta^{-4}$ is clearly suboptimal.
Our next main result tackles this issue and in fact achieves the optimal factor of $\delta^{-2}$.
The key difference to Theorem~\ref{thm:results:bounds:global} is that the complexity of the (transformed) hypothesis set $\atoms^\T\sset$ is now measured \emph{locally} at the target vector $\targetlatpv$.
\begin{theorem}[Estimation via \eqref{eq:results:bounds:klasso} -- Conic version]\label{thm:results:bounds:conic}
	Let Assumption~\ref{model:results:observations} be satisfied.
	Let $\sset \subset \R^\p$ be a convex subset and fix a vector $\targetlatpv \in \atoms^\T\sset \subset \R^\d$.
	Then there exists a numerical constant $C > 0$ such that for every $\delta \in \intvopcl{0}{1}$, the following holds true with probability at least $1 - 5\exp(- C \cdot \subgparam^{-4} \cdot \delta^2 \cdot n)$:
	If the number of observed samples obeys
	\begin{equation}\label{eq:results:bounds:conic:meas}
		n \gtrsim \subgparam^4 \cdot \delta^{-2} \cdot \effdim[\conic]{\atoms^\T\sset - \targetlatpv},
	\end{equation}
	then every minimizer $\solu\xpv$ of \eqref{eq:results:bounds:klasso} satisfies
	\begin{equation}\label{eq:results:bounds:conic:bound}
		\lnorm{\atoms^\T\solu\xpv - \targetlatpv} \lesssim \subgparam^{-1} \cdot \modeldev{\targetlatpv} \cdot \delta + \modelcovar{\targetlatpv}.
	\end{equation}
\end{theorem}
Analogously to \eqref{eq:results:bounds:global:bound-decay}, we can restate \eqref{eq:results:bounds:conic:bound} as a sample-dependent error bound:
\begin{equation}\label{eq:results:bounds:conic:bound-decay}
	\lnorm{\atoms^\T\solu\xpv - \targetlatpv} \lesssim \subgparam \cdot \modeldev{\targetlatpv} \cdot \Big(\frac{\effdim[\conic]{\atoms^\T\sset - \targetlatpv}}{n}\Big)^{1/2} + \modelcovar{\targetlatpv}.
\end{equation}
In certain special cases, this error estimate is even minimax optimal up to a constant, see \cite[Sec.~4]{plan2014highdim} for a detailed study in the setup of Gaussian single-index models.\footnote{More precisely, the optimality result of \cite[Sec.~4]{plan2014highdim} relies on the following argument: If the output $y$ would be linear in $\lat$ with independent Gaussian noise, no estimator can achieve a better rate. Thus, if $y$ obeys a non-linear (and therefore more ``complicated'') model, one cannot expect a better estimation performance than in the linear case.}
The price to pay for the improved decay rate of $\asympfaster{n^{-1/2}}$ in \eqref{eq:results:bounds:conic:bound-decay} is the presence of a more complicated complexity parameter. Indeed, using the conic mean width $\meanwidth[\conic]{\atoms^\T\sset - \targetlatpv}$ instead of its global counterpart $\meanwidth{\atoms^\T\sset}$ comes along with two difficulties:
\begin{listing}
\item
	\emph{Computability.} The conic mean width is a quite implicit parameter that is exactly computable only in special cases, e.g., see Subsection~\ref{subsec:appl:nonlinear:sset}.
	This issue is further complicated by the deformation of mixing matrix $\atoms$:
	Our primary goal is to exploit low-complexity features of $\targetlatpv$, whereas the actual hypothesis set $\sset$ merely restricts the search space of \eqref{eq:results:bounds:klasso}.
	Hence, in order to compute or accurately bound $\meanwidth[\conic]{\atoms^\T\sset - \targetlatpv}$, the impact of $\atoms^\T$ needs to be precisely understood. 
\item
	\emph{Stability.} The value of $\meanwidth[\conic]{\atoms^\T\sset - \targetlatpv}$ is highly sensitive to the position of $\targetlatpv$ within $\atoms^\T\sset$. For example, if the set $\atoms^\T\sset$ is full-dimensional and $\targetlatpv$ does not lie exactly on its boundary, one would simply end up with $\effdim[\conic]{\atoms^\T\sset - \targetlatpv} \asymp \d$. This leads to an overly pessimistic error bound in \eqref{eq:results:bounds:conic:bound-decay}, which does not reflect any benefit of using a ``small'' hypothesis set.
\end{listing}
The latter point particularly implies a challenging tuning problem: compared to the approach of Remark~\ref{rmk:results:bounds}\ref{rmk:results:bounds:mismatchpractical}, it does not suffice to just achieve $\targetlatpv \in \atoms^\T(\lambda\sset)$ for some $\lambda > 0$, but one even requires that $\targetlatpv \in \boundary{(\atoms^\T(\lambda\sset))}$.
Fortunately, such a perfect tuning is usually not necessary in practice and the generalized Lasso \eqref{eq:results:bounds:klasso} turns out to be quite stable under these types of model inaccuracies.
The above issue of stability is in fact rather an artifact of employing the conic mean width as complexity measure.
Based on the refined notion of \emph{local mean width} (see Definition~\ref{def:proofs:results:localmw}), we prove a more general version of Theorem~\ref{thm:results:bounds:conic} in Subsection~\ref{subsec:proofs:local} which does not suffer from this drawback.
This refined error estimate (Theorem~\ref{thm:proofs:results:local}) actually forms the basis for proving Theorem~\ref{thm:results:bounds:global} as well as Theorem~\ref{thm:results:bounds:conic}, and in principle, it allows us to establish more advanced results, e.g., on the stability of \eqref{eq:results:bounds:klasso}.
However, for the sake of brevity, we have decided to omit a detailed discussion on these extensions and refer the interested reader to \cite[Sec.~6.1~+~Lem.~A.2]{genzel2017cosparsity}.

According to Remark~\ref{rmk:results:bounds}\ref{rmk:results:bounds:mpgeneric}, Recipe~\ref{rec:results:mismatch}\ref{rec:results:mismatch:bound} is of course also applicable to Theorem~\ref{thm:results:bounds:conic}, but as pointed out above, the simplification of Remark~\ref{rmk:results:bounds}\ref{rmk:results:bounds:mismatchpractical} would only work poorly in this case.
To avoid unnecessary technicalities, the examples of the mismatch principle presented in Section~\ref{sec:appl} are therefore mainly based on an application of Theorem~\ref{thm:results:bounds:global}.
We close this section with a more technical discussion of the key differences between the above results and previously known error bounds for the (generalized) Lasso from the statistical learning theory and signal processing:
\begin{remark}\label{rmk:results:outofbox}
	Using the isotropy of the latent factors $\lat$ and the orthogonal decomposition
	\begin{equation}
		\lat = \sp{\lat}{\tfrac{\targetlatpv}{\lnorm{\targetlatpv}}} \tfrac{\targetlatpv}{\lnorm{\targetlatpv}} + \orthcompl{\proj}_{\targetlatpv} \lat,
	\end{equation}
	one can rephrase the mismatch covariance as follows:
	\begin{align}
		\modelcovar{\targetlatpv} &= \lnorm[\big]{\mean[(y - \sp{\lat}{\targetlatpv}) \lat]} = \lnorm[\big]{\mean[y \lat] - \mean[\sp{\lat}{\targetlatpv}\lat]} \\*
		&= \lnorm[\big]{\mean[y \lat] - \mean[\sp{\lat}{\targetlatpv} \big( \sp{\lat}{\tfrac{\targetlatpv}{\lnorm{\targetlatpv}}} \tfrac{\targetlatpv}{\lnorm{\targetlatpv}} + \orthcompl{\proj}_{\targetlatpv} \lat \big)]} \\
		&= \lnorm[\big]{\mean[y \lat] - \underbrace{\mean[\sp{\lat}{\targetlatpv}^2]}_{= \lnorm{\targetlatpv}^2} \tfrac{\targetlatpv}{\lnorm{\targetlatpv}^2} } = \lnorm[\big]{\mean[y \lat] - \targetlatpv } \ .
	\end{align}
	In other words, $\modelcovar{\targetlatpv}$ simply measures the Euclidean distance of $\targetlatpv$ to $\soluexp\latpv \coloneqq \mean[y \lat] \in \R^\d$; and as long as $\mean[y \lat] \in \atoms^\T \sset$, every parameter vector $\soluexp\xpv \in \sset$ with $\soluexp\latpv = \atoms^\T \soluexp\xpv$ forms an expected risk minimizer of \eqref{eq:intro:explossmin-linear}.
	Hence, a similar error bound as in Theorem~\ref{thm:results:bounds:global} and Theorem~\ref{thm:results:bounds:conic} could be derived by first applying the triangle inequality,
	\begin{equation}
		\lnorm{\atoms^\T\solu\xpv - \targetlatpv} \leq \lnorm{\atoms^\T\solu\xpv - \soluexp\latpv} + \lnorm{\soluexp\latpv - \targetlatpv} = \lnorm{\atoms^\T\solu\xpv - \soluexp\latpv} + \modelcovar{\targetlatpv},
	\end{equation}
	and then invoking a known guarantee from statistical learning theory, e.g., from \cite{mendelson2014learning}, to control the error term $\lnorm{\atoms^\T\solu\xpv - \soluexp\latpv}$.
	However, this simple strategy does not lead to the outcome of Theorem~\ref{thm:results:bounds:global} and Theorem~\ref{thm:results:bounds:conic} for two main reasons: firstly, we do not require that $\soluexp\latpv = \mean[y \lat] \in \atoms^\T \sset$; and secondly, we capture the complexity of the hypothesis set locally at $\targetlatpv$ in \eqref{eq:results:bounds:conic:meas}, which can be very different from an evaluation at $\soluexp\latpv$, as it is usually done.\footnote{Note that the second aspect does not become obvious from the ``global'' result of Theorem~\ref{thm:results:bounds:global}, whereas it becomes crucial when aiming at improved (optimal) decay rates such as in Theorem~\ref{thm:results:bounds:conic} and Theorem~\ref{thm:proofs:results:local}.}
	Consequently, the above performance guarantees are not off-the-shelf corollaries of previous findings in the statistical learning literature, but rather form a refinement of those.
	Nevertheless, it is still useful to bear in mind the equivalent formulation $\modelcovar{\targetlatpv} = \lnorm[\big]{\mean[y \lat] - \targetlatpv }$, as it clarifies the statistical role of the mismatch covariance, i.e., an asymptotic bias that arises when estimating $\targetlatpv$ instead of $\mean[y \lat]$.
\end{remark}

\section{Applications and Examples}
\label{sec:appl}

This section elaborates several applications of the general framework developed in Section~\ref{sec:results}.
In the first part (Subsection~\ref{subsec:appl:nonlinear}), we still focus on isotropic input data (i.e., $\atoms = \I{\p}$ in Assumption~\ref{model:results:observations}) and apply the mismatch principle to various types of semi-parametric output models. 
In particular, we revisit our two prototypical examples from the introduction, that is, single-index models \eqref{eq:intro:sim} and variable selection \eqref{eq:intro:vs}.
Subsection~\ref{subsec:appl:correlated} is then devoted to issues and challenges associated with correlated feature variables (i.e., $\atoms \neq \I{\p}$), distinguishing the underdetermined ($\p \geq \d$) and overdetermined case ($\p < \d$).

\subsection{Semi-Parametric Output Rules}
\label{subsec:appl:nonlinear}

In the following subsections, we investigate several popular examples of semi-parametric models. While the underlying output rules for $y$ actually rely on very different structural assumptions, the common goal is to study the capability of the generalized Lasso \eqref{eq:results:bounds:klasso} to estimate a certain set of ``ground truth'' parameters. In this context, the mismatch principle (Recipe~\ref{rec:results:mismatch}) turns out to be very useful, as it allows us to deal with each of the cases in a highly systematic way.

In order to keep the exposition simple and clear, we will only consider the situation of isotropic input vectors, i.e., Assumption~\ref{model:results:observations} is fulfilled with $\atoms = \I{\d}$.
Consequently, it holds that $\x = \lat$ and $\x_i = \lat_i$ for $i = 1, \dots, n$, and $\solu\xpv = \solu\latpv$.
Adopting the notation from Subsection~\ref{subsec:intro:mismatch}, we particularly agree on writing $\lat$, $\lat_i$, $\latpv$, $\solu\latpv$ instead of $\x$, $\x_i$, $\xpv$, $\solu\xpv$, respectively, so that the estimator \eqref{eq:results:bounds:klasso} takes the form
\begin{equation} \label{eq:appl:nonlinear:klassoiso}\tag{$\mathsf{P}_{\sset}^{\text{iso}}$}
	\min_{\latpv \in \sset} \ \tfrac{1}{n} \sum_{i = 1}^n (y_i - \sp{\lat_i}{\latpv})^2.
\end{equation}
For the sake of convenience, let us also restate Theorem~\ref{thm:results:bounds:global} in this specific setup (which is in fact the formal version of Theorem~\ref{thm:intro:maininformal}):
\begin{corollary}[Estimation via \eqref{eq:results:bounds:klasso} -- Isotropic data]\label{cor:appl:nonlinear:global}
	Let Assumption~\ref{model:results:observations} be satisfied with $\atoms = \I{\d}$.
	Let $\sset \subset \R^\d$ be a bounded, convex subset and fix a vector $\targetlatpv \in \sset$.
	Then there exists a numerical constant $C > 0$ such that for every $\delta \in \intvopcl{0}{1}$, the following holds true with probability at least $1 - 5\exp(- C \cdot \subgparam^{-4} \cdot \delta^2 \cdot n)$:
	If the number of observed samples obeys
	\begin{equation}\label{eq:appl:nonlinear:global:meas}
		n \gtrsim \subgparam^4 \cdot \delta^{-4} \cdot \effdim{\sset},
	\end{equation}
	then every minimizer $\solu\latpv$ of \eqref{eq:appl:nonlinear:klassoiso} satisfies
	\begin{equation}\label{eq:appl:nonlinear:global:bound}
		\lnorm{\solu\latpv - \targetlatpv} \lesssim \max\{1, \subgparam \cdot \modeldev{\targetlatpv}\} \cdot \delta + \modelcovar{\targetlatpv}.
	\end{equation}
\end{corollary}
Note that the subsequent analysis of output models does only concern step~\ref{rec:results:mismatch:tset} and step~\ref{rec:results:mismatch:modelcovar} of Recipe~\ref{rec:results:mismatch}. Indeed, according to the simplification of Remark~\ref{rmk:results:bounds}\ref{rmk:results:bounds:mismatchpractical}, it already suffices to analyze the mismatch covariance on the target set, whereas the hypothesis set can be left unspecified.
Some examples of $\sset$ are however discussed in Subsection~\ref{subsec:appl:nonlinear:sset} below.
Moreover, let us mention that in what follow, one could also invoke the mismatch principle in conjunction with the refined statements of Theorem~\ref{thm:results:bounds:conic} and Theorem~\ref{thm:proofs:results:local} (cf. Remark~\ref{rmk:results:bounds}\ref{rmk:results:bounds:mpgeneric}), but this would involve some additional technicalities distracting from the main ideas.

\subsubsection{Single-Index Models}
\label{subsec:appl:nonlinear:sim}

Let us begin with defining a single-index model, which was already informally introduced in \eqref{eq:intro:sim}:
\begin{assumption}[Single-index models]\label{model:appl:nonlinear:sim}
	Let Assumption~\ref{model:results:observations} be satisfied with $\atoms = \I{\d}$ and let
	\begin{equation}\label{eq:appl:nonlinear:sim}
		y = \fobs(\sp{\lat}{\trulatpv})
	\end{equation}
	where $\trulatpv \in \R^\d \setminus \{\vnull\}$ is an (unknown) \emph{index vector} and $\fobs \colon \R \to \R$ is a measurable scalar \emph{output function} which can be random (independently of $\lat$).\footnote{Note that we implicitly assume that $y$ is a sub-Gaussian variable.}
\end{assumption}
Since Assumption~\ref{model:appl:nonlinear:sim} imposes rather mild restrictions on the output function $\fobs$, \eqref{eq:appl:nonlinear:sim} covers many model situations of practical interest, e.g., noisy linear regression ($\fobs = \Id + e$ with noise $e$), $1$-bit compressed sensing ($\fobs = \sign$), linear classification, or phase retrieval ($\fobs = \abs{\cdot}$); see Example~\ref{ex:appl:nonlinear:sim} below and Subsection~\ref{subsec:appl:nonlinear:1bit} for more details on these applications.
It is worth mentioning that the name `single-index model' originates from the field of \emph{econometrics}, e.g., see \cite{ichimura1993sim,horowitz2009semi}; but the scope of single-index models is much wider by now, forming a very active research branch in machine learning and signal processing, with an increasing amount of works in the last years.
We refer to \cite[Sec.~6]{plan2014highdim} and \cite[Subsec.~1.2]{yang2017high} for a broader overview of the extensive literature on single-index models, also including some historical references; a more specific discussion of approaches closely related to ours can be found in Subsection~\ref{subsec:literature:cs}.

In the context of this article, the output function $\fobs$ plays the role of a (non-parametric) model uncertainty that is unknown to the estimator \eqref{eq:appl:nonlinear:klassoiso}.
If $\fobs$ is non-linear, we particularly cannot expect that the mean squared error in \eqref{eq:appl:nonlinear:klassoiso} is small, no matter what parameter vector $\latpv \in \sset$ is selected. Fortunately, this issue does not bother us to much because our primary goal is to learn the index vector $\trulatpv$, and not to predict the true response variable~$y$.
With this in mind, it appears quite natural to choose $\tset = \spann{\{\trulatpv\}}$ as target set. We emphasize that just using a singleton, e.g., $\tset = \{\trulatpv\}$, is not necessarily meaningful, since the magnitude of $\trulatpv$ might be ``absorbed'' by $\fobs$, like in the case of $\fobs = \sign$.
Hence, without any further assumptions on $\fobs$, the best we can expect is to recover a scalar multiple (or the direction) of $\trulatpv$.
Invoking the step~\ref{rec:results:mismatch:modelcovar} of Recipe~\ref{rec:results:mismatch} now leads to the following outcome:
\begin{proposition}\label{prop:appl:nonlinear:sim:mp}
	Let Assumption~\ref{model:appl:nonlinear:sim} be satisfied and set $\tset \coloneqq \spann\{\trulatpv\}$.
	Then $\targetlatpv \coloneqq \scalfac \trulatpv \in \tset$ with
	\begin{equation}
		\scalfac \coloneqq \tfrac{1}{\lnorm{\trulatpv}^2} \mean[\fobs(\sp{\lat}{\trulatpv})\sp{\lat}{\trulatpv}]
	\end{equation}
	minimizes the mismatch covariance on $\tset$ and we have
	\begin{equation}\label{eq:appl:nonlinear:sim:modelcovar}
		\modelcovar{\targetlatpv} = \lnorm[\big]{\mean[\fobs(\sp{\lat}{\trulatpv})\orthcompl{\proj}_{\trulatpv} \lat]} \ .
	\end{equation}
	In particular, if $\lat \distributed \Normdistr{\vnull}{\I{\d}}$, we have $\modelcovar{\targetlatpv} = 0$.
\end{proposition}
\begin{proof}
	Let us make the ansatz $\targetlatpv = \scalfac\trulatpv \in \tset$ where $\scalfac \in \R$ is specified later on.
	Using the orthogonal decomposition
	\begin{equation}
		\lat = \proj_{\trulatpv} \lat + \orthcompl{\proj}_{\trulatpv} \lat = \sp{\lat}{\tfrac{\trulatpv}{\lnorm{\trulatpv}}} \tfrac{\trulatpv}{\lnorm{\trulatpv}} + \orthcompl{\proj}_{\trulatpv} \lat,
	\end{equation}
	the mismatch covariance of $\targetlatpv$ simplifies as follows:
	\begin{align}
		& \modelcovar{\targetlatpv}^2 = \lnorm[\big]{\mean[(y - \sp{\lat}{\scalfac\trulatpv}) \lat]}^2 = \lnorm[\big]{\mean[(\fobs(\sp{\lat}{\trulatpv}) - \sp{\lat}{\scalfac\trulatpv}) \big( \sp{\lat}{\tfrac{\trulatpv}{\lnorm{\trulatpv}}} \tfrac{\trulatpv}{\lnorm{\trulatpv}} + \orthcompl{\proj}_{\trulatpv} \lat \big)][\big]}^2 \\*
		={} & \lnorm[\big]{\tfrac{1}{\lnorm{\trulatpv}^2} \mean[\fobs(\sp{\lat}{\trulatpv}) \sp{\lat}{\trulatpv}] \trulatpv - \tfrac{\scalfac}{\lnorm{\trulatpv}^2} \underbrace{\mean[ \sp{\lat}{\trulatpv}^2]}_{= \lnorm{\trulatpv}^2} \trulatpv}^2 + \lnorm[\big]{\mean[(\fobs(\sp{\lat}{\trulatpv}) - \sp{\lat}{\scalfac\trulatpv}) \orthcompl{\proj}_{\trulatpv} \lat][\big]}^2 \\
		={} & \Big(\tfrac{1}{\lnorm{\trulatpv}^2} \mean[\fobs(\sp{\lat}{\trulatpv}) \sp{\lat}{\trulatpv}] - \scalfac \Big)^2 \cdot \lnorm{\trulatpv}^2 + \lnorm[\big]{\mean[\fobs(\sp{\lat}{\trulatpv}) \orthcompl{\proj}_{\trulatpv} \lat][\big]}^2 \ , \label{eq:appl:nonlinear:sim:mp:simplif}
	\end{align}
	where we have used that $\mean[\sp{\lat}{\scalfac\trulatpv} \orthcompl{\proj}_{\trulatpv} \lat] = \vnull$ due to the isotropy of $\lat$.
	Since the second summand in \eqref{eq:appl:nonlinear:sim:mp:simplif} does not depend on $\scalfac$, the minimum of $\scalfac \mapsto \modelcovar{\scalfac\trulatpv}$ is indeed attained at $\scalfac = \mean[\fobs(\sp{\lat}{\trulatpv})\sp{\lat}{\trulatpv}] / \lnorm{\trulatpv}^2$.
	
	Finally, if $\lat \distributed \Normdistr{\vnull}{\I{\d}}$, the random vector $\orthcompl{\proj}_{\trulatpv} \lat$ is independent (not just uncorrelated) from $\sp{\lat}{\trulatpv}$, so that $\mean[\fobs(\sp{\lat}{\trulatpv}) \orthcompl{\proj}_{\trulatpv} \lat] = \vnull$.
\end{proof}
The above proof may serve as a template for applying the mismatch principle: First, make a parametric ansatz according to the target set and simplify the mismatch covariance as much as possible. Then select the free parameters such that the simplified expression is minimized (or at least gets small).
We will see below that this strategy works out for many other examples as well.

The Gaussian case of Proposition~\ref{prop:appl:nonlinear:sim:mp} leads to a very desirable situation, as it states that single-index models can be \emph{consistently} learned via the generalized Lasso \eqref{eq:appl:nonlinear:klassoiso} if $\lat \distributed \Normdistr{\vnull}{\I{\d}}$.
This conclusion precisely corresponds to a recent finding of Plan and Vershynin from \cite{plan2015lasso}, which relies on a calculation similar to \eqref{eq:appl:nonlinear:sim:mp:simplif}.
But beyond that, Proposition~\ref{prop:appl:nonlinear:sim:mp} even extends the results of \cite{plan2015lasso} to sub-Gaussian input variables, where the situation is in fact more complicated: In general, the mismatch covariance does not necessarily vanish at $\targetlatpv = \scalfac\trulatpv$, implying that \eqref{eq:appl:nonlinear:klassoiso} does not constitute a consistent estimator of any target vector in $\tset = \spann\{\trulatpv\}$. 
Consequently, according to Corollary~\ref{cor:appl:nonlinear:global}, the actual estimation capacity of \eqref{eq:appl:nonlinear:klassoiso} is specified by the expression $\modelcovar{\targetlatpv} = \lnorm{\mean[\fobs(\sp{\lat}{\trulatpv})\orthcompl{\proj}_{\trulatpv} \lat]}$, which in a certain sense measures the compatibility between the non-linear output model and the isotropic data vector. We will return to this issue in the course of variable selection (see Remark~\ref{rmk:appl:nonlinear:mim:vs-sim}), showing that asymptotically unbiased estimates can be still achieved as long as one is only aiming at the support of $\trulatpv$.

Let us conclude our discussion with some interesting model setups for which the outcome of Proposition~\ref{prop:appl:nonlinear:sim:mp} is well interpretable:
\begin{example}
\begin{rmklist}
\item\label{ex:appl:nonlinear:sim:rotinv}
	\emph{Rotationally invariant distributions.}
	This class of probability distributions is a natural generalization of Gaussian random vectors, according to which $\lat$ takes the form $r \vec{u}$, where the radius $r$ is sub-Gaussian on $\intvop{0}{\infty}$ and $\vec{u}$ is uniformly distributed on $\S^{\d-1}$.
	In the context of single-index models, this scenario was recently studied by Goldstein et al.\ in \cite{goldstein2016structured}, constructing a consistent estimator that is equivalent to \eqref{eq:results:bounds:klasso}.
	A simple calculation confirms their result by means of Proposition~\ref{prop:appl:nonlinear:sim:mp}: For every $\latpv \in \R^\d$, it holds that
	\begin{align}
		\mean[\fobs(\sp{\lat}{\trulatpv}) \sp{\orthcompl{\proj}_{\trulatpv}\lat}{ \latpv}] &= \mean{} \big[\mean[ \fobs(\sp{\lat}{\trulatpv}) \sp{\lat}{\orthcompl{\proj}_{\trulatpv} \latpv} \suchthat \sp{\lat}{\trulatpv}]\big] \\*
		&= \mean{} \big[\fobs(\sp{\lat}{\trulatpv}) \cdot \underbrace{\mean[ \sp{\lat}{\orthcompl{\proj}_{\trulatpv} \latpv} \suchthat \sp{\lat}{\trulatpv}]}_{= \tfrac{1}{\lnorm{\trulatpv}^2} \sp{\orthcompl{\proj}_{\trulatpv} \latpv}{\trulatpv} \sp{\lat}{\trulatpv} = 0}\big] = 0,
	\end{align}
	where we have used the law of total expectation and \cite[Cor.~1]{goldstein2016structured}.
	Hence, we conclude that $\modelcovar{\targetlatpv} = \lnorm{\mean[\fobs(\sp{\lat}{\trulatpv})\orthcompl{\proj}_{\trulatpv} \lat]} = 0$.
\item\label{ex:appl:nonlinear:sim:regression}
	\emph{Linear regression.} The most simple class of single-index models are linear observations of the form $y = \fobs(\sp{\lat}{\trulatpv}) \coloneqq \sp{\lat}{\trulatpv} + e$ with $e$ being independent, centered sub-Gaussian noise.
	In this case, it is not hard to see that Proposition~\ref{prop:appl:nonlinear:sim:mp} yields $\scalfac = 1$ and $\modelcovar{\trulatpv} = 0$.
	Note that, compared to non-linear output functions, $\scalfac$ does not depend on $\lnorm{\trulatpv}$ here.
	Due to $\modeldev{\trulatpv} = \normsubg{e}$, the error bound \eqref{eq:appl:nonlinear:global:bound} of Corollary~\ref{cor:appl:nonlinear:global} is  determined by the strength of noise, which is a well-known fact from the literature. In particular, one can achieve exact recovery in the noiseless case $e = 0$.
\item\label{ex:appl:nonlinear:sim:worstcase}
	\emph{Worst case error bounds.} While the latter two examples lead to desirable situations, there exist model configurations where learning the index vector fails. For instance, if the entries of $\lat$ are i.i.d.\ Rademacher distributed, recovery of certain extremely sparse signals from noiseless $1$-bit measurements is impossible ($\fobs = \sign$), regardless of the considered estimator (cf. \cite[Rmk.~1.5]{plan2013onebit}). 
	As an example, let us consider the two index vectors $\trulatpv = (1, 0, 0, \dots, 0) \in \R^\d$ and $\trulatpv' = (1, \tfrac{1}{2}, 0 ,\dots, 0) \in \R^\d$. Then it is not hard to see that $y = \sign(\sp{\lat}{\trulatpv}) = \sign(\sp{\lat}{\trulatpv'})$, implying that $\trulatpv$ and $\trulatpv'$ are indistinguishable on the basis of the observation model, although we have $\lnorm{\trulatpv - \trulatpv'} = 1/2$.
	Proposition~\ref{prop:appl:nonlinear:sim:mp} shows even more: since $\modelcovar{\trulatpv} = 0$ and $\modelcovar{\trulatpv'} = 1/2$, we can conclude that \eqref{eq:appl:nonlinear:klassoiso} would form a consistent estimator of $\trulatpv$ but not of $\trulatpv'$.
	
	Interestingly, it has turned out in a subsequent work \cite{ai2014onebitsubgauss} that these types of adversarial examples can be excluded by worst-case error bounds.
	Although the setup of \cite{ai2014onebitsubgauss} focuses on a simpler linear estimator, we believe that the techniques applied there could be also used to derive an appropriate upper bound on $\modelcovar{\targetlatpv}$.
	The assertion of Proposition~\ref{prop:appl:nonlinear:sim:mp} would then allow us to specify those index vectors which can(not) be accurately recovered by \eqref{eq:appl:nonlinear:klassoiso}.
	An alternative approach for $1$-bit measurements is presented in the next subsection, demonstrating that these limitations can be removed by \emph{dithered quantization}.
\item\label{ex:appl:nonlinear:sim:evenfunc}
	\emph{Even output functions.}
	If $\fobs$ is an even function and $\lat \distributed \Normdistr{\vnull}{\I{\d}}$, one ends up with $\scalfac = 0$ in Proposition~\ref{prop:appl:nonlinear:sim:mp}.
	Thus, Corollary~\ref{cor:appl:nonlinear:global} states that \eqref{eq:appl:nonlinear:klassoiso} simply approximates the null vector, which is clearly not what we are aiming at. This negative result nevertheless underpins the versatility of the mismatch principle because it reflects that linear hypothesis functions are not suited for problems like phase retrieval ($\fobs = \abs{\cdot}$).
	Based on a lifting argument, recent works \cite{yang2017stein,yang2017nips,thrampoulidis2018lifting} show that learning single-index models is still feasible via convex optimization for even output functions (or more generally, second-order output functions); but their method comes along with the price of a suboptimal sampling rate in the case of sparse recovery. Moreover, a related approach in \cite{yang2017misspecified} indicates that similar issues do also arise when using a non-convex estimator instead. \qedhere
\end{rmklist}
\label{ex:appl:nonlinear:sim}
\end{example}

\subsubsection{\texorpdfstring{$1$}{1}-Bit Observation Models and Dithering}
\label{subsec:appl:nonlinear:1bit}

As pointed out in the previous subsection, the generalized Lasso \eqref{eq:appl:nonlinear:klassoiso} does not always yield a consistent estimator for the index vector of a single-index model. This subsection demonstrates that such an undesirable behavior can be avoided by a \emph{dithering} step, which allows us to control the bias term $\modelcovar{\targetlatpv}$ by the sample size $n$.
While dithering is actually a well-known technique in quantized signal processing, e.g., see \cite{dabeer2006dithering,gray1993dithering,gray1998quantization}, its benefit recently also emerged in the field of compressed sensing \cite{knudson2016dithering,baraniuk2017dithering,xu2018dithering,dirksen2018robust,dirksen2018ditherstruct,thrampoulidis2018dithering}.
In the following, we will focus on the extreme situation of \emph{$1$-bit observations}, where $y = \fobs(\sp{\lat}{\trulatpv})$ with $\fobs \colon \R \to \{-1,+1\}$, i.e., each output carries only one single bit of information about the linear observation process $\lat \mapsto \sp{\lat}{\trulatpv}$.
Technically, this is just a special case of a single-index model (see Assumption~\ref{model:appl:nonlinear:sim}), but these models do in fact form an independent research subject referred to as \emph{$1$-bit compressed sensing} \cite{boufounos2008onebit}.

The simplest example of a quantization function is $\fobs = \sign$, which already appeared in Example~\ref{ex:appl:nonlinear:sim}\ref{ex:appl:nonlinear:sim:worstcase} and leads to \emph{perfect $1$-bit measurements}.
Geometrically, the associated response variable $y = \sign(\sp{\lat}{\trulatpv})$ then simply specifies on which side of the hyperplane $\orthcompl{\{\trulatpv\}}$ the input vector $\lat$ lies.
A slightly more advanced scenario is the presence of \emph{random bit-flips} after quantization, i.e., $y = \eps \cdot \sign(\sp{\lat}{\trulatpv})$ where $\eps$ is an independent \mbox{$\pm1$-valued} random variable with $q \coloneqq \prob[\eps = 1] \in \intvcl{\tfrac{1}{2}}{1}$.
In the Gaussian case $\lat \distributed \Normdistr{\vnull}{\I{\d}}$, Proposition~\ref{prop:appl:nonlinear:sim:mp} states that $\modelcovar{\scalfac\trulatpv} = 0$ with
\begin{equation}
	\scalfac = \tfrac{1}{\lnorm{\trulatpv}^2} \cdot \mean[\eps \cdot \abs{\sp{\lat}{\trulatpv}}] = \tfrac{1}{\lnorm{\trulatpv}^2} \cdot \mean[\eps] \cdot \mean[\abs{\sp{\lat}{\trulatpv}}] = \tfrac{1}{\lnorm{\trulatpv}} \cdot (2q - 1) \cdot \sqrt{\tfrac{2}{\pi}} \ .
\end{equation}
Thus, \eqref{eq:appl:nonlinear:global:bound} in Corollary~\ref{cor:appl:nonlinear:global} provides an upper bound on
\begin{equation}\label{eq:appl:nonlinear:1bit:boundrndflip}
	\lnorm{\solu\latpv - \scalfac\trulatpv} = \lnorm[\Big]{\solu\latpv - (2q - 1) \cdot \sqrt{\tfrac{2}{\pi}} \cdot \tfrac{\trulatpv}{\lnorm{\trulatpv}}} \ .
\end{equation}
This reflects the fact that recovery becomes impossible if $q = 1/2$, and more importantly, we can only expect an estimate of the direction of $\trulatpv$ but not of its magnitude.

Similar to this example situation, most theoretical approaches in the $1$-bit compressed sensing literature focus on the Gaussian case, e.g., see \cite{plan2013onebit,plan2013robust,jacques2013onebit}, and also \cite{dirksen2017onebitcirc} for an interesting extension to subsampled Gaussian circulant matrices. 
This restriction is particularly due to the difficulties with non-Gaussian input vectors described in Example~\ref{ex:appl:nonlinear:sim}\ref{ex:appl:nonlinear:sim:worstcase}.
For that reason, it was a somewhat astonishing finding by Dirksen and Mendelson \cite{dirksen2018robust} that a slight modification of the (perfect) $1$-bit measurement process suffices to construct a consistent estimator of $\trulatpv$ even in the sub-Gaussian case.
The key technique behind their achievement is \emph{dithering} which, in its most basic form, corresponds to a random shift of the quantization threshold.
The following adaption of Assumption~\ref{model:appl:nonlinear:sim} makes this more precise:
\begin{assumption}[Dithered $1$-bit observations]\label{model:appl:nonlinear:1bit}
	Let Assumption~\ref{model:results:observations} be satisfied with $\atoms = \I{\d}$. For $\Delta > 0$, we assume that
	\begin{equation}\label{eq:appl:nonlinear:1bit}
	y = \Delta \cdot \sign(\sp{\lat}{\trulatpv} + \tau^\Delta)
	\end{equation}
	where $\trulatpv \in \R^\d \setminus \{\vnull\}$ and $\tau^\Delta$ is uniformly distributed over the interval $\intvcl{-\Delta}{\Delta}$, independently of the latent factors $\lat$.
\end{assumption}
Compared to bit-flips \emph{after} quantization as considered above, a crucial feature of the output rule \eqref{eq:appl:nonlinear:1bit} is that the random fluctuations of $\tau^\Delta$ affect the quantization process itself:
sampling from \eqref{eq:appl:nonlinear:1bit} generates a set of measurements
\begin{equation}
y_i = \Delta \cdot \sign(\sp{\lat_i}{\trulatpv} + \tau_i^\Delta), \quad i = 1, \dots, n,
\end{equation}
with randomly chosen quantization thresholds $\tau_1^\Delta, \dots, \tau_n^\Delta$. 
We emphasize that the \emph{dithering parameter} $\Delta$ is often known and controllable in practice, whereas the realizations of the thresholds are typically unknown; the scaling factor $\Delta$ in \eqref{eq:appl:nonlinear:1bit} is therefore just a matter of convenience that simplifies the statement of Proposition~\ref{prop:appl:nonlinear:1bit:dithering} below.
More general background information about dithering is contained in \cite{thrampoulidis2018dithering} and the references therein. Moreover, one may consult \cite{dirksen2018robust} for a nice geometrical interpretation of dithered \mbox{$1$-bit} quantization, related to random hyperplane tessellations.

At first sight, the idea of using random thresholds might appear somewhat odd, as they form an extra source of noise compared to perfect quantization with $\fobs = \sign$. Hence, it is remarkable that by carefully selecting $\Delta$ in Assumption~\ref{model:appl:nonlinear:1bit} (as a function of $n$), one can gain control of the size of the mismatch covariance.
This strategy is formalized by the following result, which is a consequence of \cite[Lem.~4.1]{dirksen2018robust}.
\begin{proposition}\label{prop:appl:nonlinear:1bit:dithering}
	Let Assumption~\ref{model:appl:nonlinear:1bit} be satisfied and assume that $\lambda \geq \lnorm{\trulatpv}$.
	There exists a numerical constant $C > 0$ such that if $\Delta \geq C \cdot \subgparam \cdot \lambda \cdot \sqrt{\log(2n)}$, we have
	\begin{align}
	\modelcovar{\trulatpv} \lesssim \frac{\Delta}{\sqrt{n}} \qquad \text{and} \qquad \modeldev{\trulatpv} \lesssim \Delta.
	\end{align}
\end{proposition}
\begin{proof}
	Using that $\lnorm{\cdot} = \sup_{\vdir \in \S^{\d-1}} \sp{\cdot}{\vdir}$, the mismatch covariance at $\trulatpv$ takes the equivalent form
	\begin{equation}
	\modelcovar{\trulatpv} = \sup_{\vdir \in \S^{\d-1}} \mean[(\Delta \cdot \sign(\sp{\lat}{\trulatpv} + \tau^\Delta) - \sp{\lat}{\trulatpv}) \sp{\lat}{\vdir}].
	\end{equation}
	An application of \cite[Lem.~4.1]{dirksen2018robust} to this expression now yields the following upper bound:
	\begin{equation}
	\modelcovar{\trulatpv} \lesssim \mean[\sp{\lat}{\vdir}^2]^{1/2} \cdot \max\{\Delta, \normsubg{\sp{\lat}{\trulatpv}} \} \cdot e^{-C' \cdot \Delta^2 / \normsubg{\sp{\lat}{\trulatpv}}^2}, \quad \vdir \in \S^{\d-1},
	\end{equation}
	where $C' > 0$ is a numerical constant. Due to $\normsubg{\sp{\lat}{\trulatpv}} \leq \subgparam \cdot \lnorm{\trulatpv} \leq \subgparam \cdot \lambda$ and the assumption $\Delta \geq C \cdot \subgparam \cdot \lambda \cdot \sqrt{\log(2n)}$ with $C \coloneqq 1 / \sqrt{2C'}$, it follows that
	\begin{align}
	\modelcovar{\trulatpv} &\lesssim \max\{\Delta, \underbrace{\subgparam \cdot \lambda}_{\lesssim \Delta} \} \cdot e^{-C' \cdot \Delta^2 / (\subgparam \cdot \lambda)^2} \lesssim \Delta \cdot e^{-\log(2n) / 2} \lesssim \frac{\Delta}{\sqrt{n}} \ .
	\end{align}
	To bound the mismatch deviation, we observe that
	\begin{equation}
	\modeldev{\trulatpv} = \normsubg{\Delta \cdot \sign(\sp{\lat}{\trulatpv} + \tau^\Delta) - \sp{\lat}{\trulatpv}} \lesssim \Delta + \underbrace{\normsubg{\sp{\lat}{\trulatpv}}}_{\leq \subgparam \cdot \lambda \lesssim \Delta} \lesssim \Delta.
	\end{equation}
\end{proof}
The statistical intuition of Proposition~\ref{prop:appl:nonlinear:1bit:dithering} is the following: If $\Delta$ grows logarithmically with $n$, the bias term $\modelcovar{\trulatpv}$ gets reduced, with the price of a slightly increased deviation term $\modeldev{\trulatpv}$, controllable by $\Delta$. As the latter is associated with the variance of the Lasso estimator \eqref{eq:appl:nonlinear:klassoiso}, we can therefore conclude that the \emph{bias-variance trade-off} is well-behaved for dithered $1$-bit observations.

Let us now evaluate Corollary~\ref{cor:appl:nonlinear:global} with $\targetlatpv = \trulatpv$ and Assumption~\ref{model:appl:nonlinear:1bit}:
For the sake of simplicity, we assume that $1 \leq \subgparam \cdot \Delta = C' \cdot \subgparam^2 \cdot \lambda \cdot \sqrt{\log(2n)}$ for some $C' \geq C$. 
Combining the error bound \eqref{eq:appl:nonlinear:global:bound} with Proposition~\ref{prop:appl:nonlinear:1bit:dithering} then leads to
\begin{align}
\lnorm{\solu\latpv - \trulatpv} &\lesssim \Delta \cdot (\subgparam \cdot \delta + \tfrac{1}{\sqrt{n}}) \lesssim \subgparam \cdot \lambda \cdot \sqrt{\log(2n)} \cdot (\subgparam \cdot \delta + \tfrac{1}{\sqrt{n}}). \label{eq:appl:nonlinear:1bit:bounddithering}
\end{align}
Since the right-hand side of \eqref{eq:appl:nonlinear:1bit:bounddithering} tends to $0$ if $n \to \infty$ and $\delta$ is appropriately chosen (cf. \eqref{eq:results:bounds:global:bound-decay}), this error estimate indeed verifies our initial claim: By carefully choosing the dithering parameter $\Delta$, the generalized Lasso \eqref{eq:appl:nonlinear:klassoiso} yields a consistent estimator of $\trulatpv$. Moreover, there is a remarkable difference to \eqref{eq:appl:nonlinear:1bit:boundrndflip}, namely that \eqref{eq:appl:nonlinear:1bit:bounddithering} even allows us to retrieve the norm of $\trulatpv$, supposed that a (rough) upper bound $\lambda \geq \lnorm{\trulatpv}$ is available.

\begin{remark}[Relationship to the mismatch principle]
	The error bound of \eqref{eq:appl:nonlinear:1bit:bounddithering} essentially reproduces one of the main results of Thrampoulidis and Rawat in \cite{thrampoulidis2018dithering}; see Theorem~IV.1 therein.
	But let us point out once again that our general estimation framework enabled us to obtain this result in a very systematic way; most notably, the only technical ingredient of our approach is Proposition~\ref{prop:appl:nonlinear:1bit:dithering}.
	
	Although not mentioned explicitly, the above derivation corresponds to an application of the mismatch principle with $\tset = \{\trulatpv\}$ as target set, which turns \ref{rec:results:mismatch:modelcovar} of Recipe~\ref{rec:results:mismatch} into a trivial step.
	In fact, an argumentation such as in Proposition~\ref{prop:appl:nonlinear:sim:mp} is not useful in the context of Assumption~\ref{model:appl:nonlinear:1bit}, since one is interested in recovery of $\trulatpv$ and not just a scalar multiple of it.
\end{remark}

\subsubsection{Generalized Linear Models}
\label{subsec:appl:nonlinear:glm}

Generalized linear models form a natural extension of linear regression problems and are conceptually very similar to single-index models (cf. \cite{nelder1972glm,mccullagh1989generalized}).
Let us begin with a formal definition:\footnote{This is not exactly how generalized linear models are usually introduced in the literature. We adopt our definition from \cite[Subsec.~3.4]{plan2014highdim}, disregarding that $y$ also belongs to an exponential family.}
\begin{assumption}[Generalized linear models]\label{model:appl:nonlinear:glm}
	Let Assumption~\ref{model:results:observations} be satisfied with $\atoms = \I{\d}$ and let the output variable obey
	\begin{equation}\label{eq:appl:nonlinear:glm}
		\mean[y \suchthat \lat] = \fobs(\sp{\lat}{\trulatpv})
	\end{equation}
	where $\trulatpv \in \R^\d \setminus \{\vnull\}$ is an (unknown) \emph{parameter vector} and $\fobs \colon \R \to \R$ is a measurable scalar function (the inverse of $\fobs$ is typically referred to as the \emph{link function}).
\end{assumption}
In contrast to the single-index model in \eqref{eq:appl:nonlinear:sim}, the condition \eqref{eq:appl:nonlinear:glm} does not specify the output variable $y$ directly, but rather assumes a parametric structure for its conditional expectation given the input vector $\lat$.
In that way, one may incorporate different types of noise than with Assumption~\ref{model:appl:nonlinear:sim}. In particular, we wish to point out that --- despite their resemblance --- both models are not completely equivalent; see \cite[Sec.~6]{plan2014highdim} for a more detailed discussion.
However, the following result shows that the assertion of Proposition~\ref{prop:appl:nonlinear:sim:mp} remains literally true for generalized linear models, implying that Corollary~\ref{cor:appl:nonlinear:global} provides a similar recovery guarantee for the parameter vector $\trulatpv$:
\begin{proposition}\label{prop:appl:nonlinear:glm:mp}
	Let Assumption~\ref{model:appl:nonlinear:glm} be satisfied and set $\tset \coloneqq \spann\{\trulatpv\}$.
	Then $\targetlatpv \coloneqq \scalfac \trulatpv \in \tset$ with
	\begin{equation}
		\scalfac \coloneqq \tfrac{1}{\lnorm{\trulatpv}^2} \mean[\fobs(\sp{\lat}{\trulatpv})\sp{\lat}{\trulatpv}]
	\end{equation}
	minimizes the mismatch covariance on $\tset$ and we have
	\begin{equation}\label{eq:appl:nonlinear:glm:modelcovar}
		\modelcovar{\targetlatpv} = \lnorm[\big]{\mean[\fobs(\sp{\lat}{\trulatpv})\orthcompl{\proj}_{\trulatpv} \lat]} \ .
	\end{equation}
\end{proposition}
\begin{proof}
	Using the law of total expectation, we observe that
	\begin{align}
		\modelcovar{\targetlatpv} &= \lnorm[\big]{\mean[(y - \sp{\lat}{\scalfac\trulatpv}) \lat]} = \lnorm[\big]{\mean[\mean[(y - \sp{\lat}{\scalfac\trulatpv}) \lat \suchthat \lat]]} \\*
		&= \lnorm[\big]{\mean{} [\underbrace{\mean[y  \suchthat \lat]}_{= \fobs(\sp{\lat}{\trulatpv})} {} \cdot {} \lat]  - \mean[\sp{\lat}{\scalfac\trulatpv} \lat]} = \lnorm[\big]{\mean[(\fobs(\sp{\lat}{\trulatpv}) - \sp{\lat}{\scalfac\trulatpv}) \lat]} \ .
	\end{align}
	Now, we may proceed exactly as in the proof of Proposition~\ref{prop:appl:nonlinear:sim:mp}.
\end{proof}

\subsubsection{Multiple-Index Models and Variable Selection}
\label{subsec:appl:nonlinear:mim}

A natural extension of single-index models, which covers a much wider class of output rules, are so-called multiple-index models (e.g., see \cite{horowitz2009semi}).
They are formally defined as follows:
\begin{assumption}[Multiple-index models]\label{model:appl:nonlinear:mim}
	Let Assumption~\ref{model:results:observations} be satisfied with $\atoms = \I{\d}$ and let
	\begin{equation}\label{eq:appl:nonlinear:mim}
		y = \Fobs(\sp{\lat}{\latpv_1}, \dots, \sp{\lat}{\latpv_S})
	\end{equation}
	where $\latpv_1, \dots, \latpv_S \in \R^\d \setminus \{\vnull\}$ are (unknown) \emph{index vectors} and $\Fobs \colon \R^S \to \R$ is a measurable \emph{output function} which can be random (independently of $\lat$).
\end{assumption}

The case of $S = 1$ obviously coincides with the single-index model from Assumption~\ref{model:appl:nonlinear:sim}.
In general, the ultimate goal is to recover all index vectors $\latpv_1, \dots, \latpv_S$, or at least the spanned subspace $\spann{\{\latpv_1, \dots, \latpv_S\}}$.
The feasibility of this task, however, strongly depends on the output function $\Fobs$. If $\Fobs$ is linear for example, recovery of all index vectors is impossible, since \eqref{eq:appl:nonlinear:mim} would simply degenerate into an ordinary linear regression model. But even in the non-linear case one might have to deal with ambiguities, such as permutation invariance of the indices.

With this in mind, applying the generalized Lasso \eqref{eq:appl:nonlinear:klassoiso} to multiple-index observations might appear useless, as a linear predictor $\lat \mapsto \sp{\lat}{\latpv}$ certainly cannot capture the impact of all index components in \eqref{eq:appl:nonlinear:mim} at the same time.
Indeed, we cannot expect recovery of all index vectors, but it still makes sense to ask whether \eqref{eq:appl:nonlinear:klassoiso} does at least approximate \emph{any} vector in  $\spann{\{\latpv_1, \dots, \latpv_S\}}$.
Regarding the mismatch principle, this suggests selecting $\tset = \spann{\{\latpv_1, \dots, \latpv_S\}}$ as target set in step \ref{rec:results:mismatch:tset} of Recipe~\ref{rec:results:mismatch}, which leads to the following outcome of \ref{rec:results:mismatch:modelcovar}: 
\begin{proposition}\label{prop:appl:nonlinear:mim:mp}
	Let Assumption~\ref{model:appl:nonlinear:mim} be satisfied. Moreover, assume that the index vectors $\latpv_1, \dots, \latpv_S \in \R^\d$ form an orthonormal system and set $\tset \coloneqq \spann\{\latpv_1, \dots, \latpv_S\}$.
	Then $\targetlatpv \coloneqq \sum_{j = 1}^S \scalfac_j \latpv_j \in \tset$ with
	\begin{equation}
		\scalfac_j \coloneqq \mean[\Fobs(\sp{\lat}{\latpv_1},\dots,\sp{\lat}{\latpv_S})\sp{\lat}{\latpv_j}], \quad j = 1, \dots, S,
	\end{equation}
	minimizes the mismatch covariance on $\tset$ and we have
	\begin{equation}\label{eq:appl:nonlinear:mim:modelcovar}
		\modelcovar{\targetlatpv} = \lnorm[\big]{\mean[\Fobs(\sp{\lat}{\latpv_1},\dots,\sp{\lat}{\latpv_S})\orthcompl{\proj}_{\tset} \lat]} \ .
	\end{equation}
	In particular, if $\lat \distributed \Normdistr{\vnull}{\I{\d}}$, we have $\modelcovar{\targetlatpv} = 0$.
\end{proposition}
\begin{proof}
	We proceed analogously to the proof of Proposition~\ref{prop:appl:nonlinear:sim:mp}: First, let us make the parametric ansatz $\targetlatpv = \sum_{j = 1}^S \scalfac_j \latpv_j \in \tset$ with unspecified scalars $\scalfac_1, \dots, \scalfac_S \in \R$, and consider the orthogonal decomposition
	\begin{equation}
		\lat = \proj_{\tset} \lat + \orthcompl{\proj}_{\tset} \lat = \sum_{j = 1}^S \sp{\lat}{\latpv_j} \latpv_j + \orthcompl{\proj}_{\tset} \lat,
	\end{equation}
	where we have used that $\latpv_1, \dots, \latpv_S$ form an orthonormal system.
	Using the isotropy of $\lat$ now implies
	\begin{align}
		\modelcovar{\targetlatpv}^2 &= \lnorm[\Big]{\mean[\Big(\Fobs(\sp{\lat}{\latpv_1}, \dots, \sp{\lat}{\latpv_S}) - \sum_{j' = 1}^S \scalfac_{j'} \sp{\lat}{\latpv_{j'}}\Big) \Big( \sum_{j = 1}^S \sp{\lat}{\latpv_j} \latpv_j + \orthcompl{\proj}_{\tset} \lat \Big)][\Big]}^2 \\*
		&= \sum_{j = 1}^S \lnorm[\Big]{\mean[\Big(\Fobs(\sp{\lat}{\latpv_1}, \dots, \sp{\lat}{\latpv_S}) - \sum_{j' = 1}^S \scalfac_{j'} \sp{\lat}{\latpv_{j'}}\Big) \sp{\lat}{\latpv_j} \latpv_j][\Big]}^2 \\*
			& \qquad + \lnorm[\Big]{\mean[\Fobs(\sp{\lat}{\latpv_1}, \dots, \sp{\lat}{\latpv_S}) \orthcompl{\proj}_{\tset} \lat ] - \underbrace{\mean[\Big(\sum_{j' = 1}^S \scalfac_{j'} \sp{\lat}{\latpv_{j'}}\Big) \orthcompl{\proj}_{\tset} \lat][\Big] }_{= \vnull}}^2 \\
		&= \sum_{j = 1}^S \Big(\mean[\Fobs(\sp{\lat}{\latpv_1}, \dots, \sp{\lat}{\latpv_S}) \sp{\lat}{\latpv_j} ] - \scalfac_j \underbrace{\mean[\sp{\lat}{\latpv_j}^2]}_{\smash{= \lnorm{\latpv_j}^2 = 1}} \Big)^2 \cdot \underbrace{\lnorm{\latpv_j}^2}_{\smash{= 1}} \\*
			& \qquad + \lnorm[\big]{\mean[\Fobs(\sp{\lat}{\latpv_1}, \dots, \sp{\lat}{\latpv_S}) \orthcompl{\proj}_{\tset} \lat ]}^2 \\
		&= \sum_{j = 1}^S \Big(\mean[\Fobs(\sp{\lat}{\latpv_1}, \dots, \sp{\lat}{\latpv_S}) \sp{\lat}{\latpv_j} ] - \scalfac_j \Big)^2 + \lnorm[\big]{\mean[\Fobs(\sp{\lat}{\latpv_1}, \dots, \sp{\lat}{\latpv_S}) \orthcompl{\proj}_{\tset} \lat ]}^2 \ .
	\end{align}
	Thus, the mismatch covariance is minimized on $\tset$ if $\scalfac_j = \mean[\Fobs(\sp{\lat}{\latpv_1},\dots,\sp{\lat}{\latpv_S})\sp{\lat}{\latpv_j}]$ for all $j = 1, \dots, S$. The claim in the Gaussian case again follows from the fact that $\sp{\lat}{\latpv_1}, \dots, \sp{\lat}{\latpv_S}$, and $\orthcompl{\proj}_{\tset} \lat$ are independent if $\lat \distributed \Normdistr{\vnull}{\I{\d}}$.
\end{proof}
We note that the orthogonality assumption on the index vectors is not overly restrictive because one may incorporate an orthogonalizing transformation by redefining $\Fobs$.
The result of Proposition~\ref{prop:appl:nonlinear:mim:mp} allows for a similar conclusion as in the situation of single-index models in Subsection~\ref{subsec:appl:nonlinear:sim}: the generalized Lasso \eqref{eq:appl:nonlinear:klassoiso} approximates a certain target vector $\targetlatpv = \sum_{j = 1}^S \scalfac_j \latpv_j$ in the span of the index vectors, supposed that the ``compatibility parameter'' $\modelcovar{\targetlatpv}$ is sufficiently small. 
In fact, we cannot expect a stronger statement at this point, since the scalar factors $\scalfac_1,\dots,\scalfac_S$ do strongly depend on the (still) unknown index vectors and the possibly unknown output function $\Fobs$.
Consequently, a recovery strategy for the entire subspace $\tset = \spann{\{\latpv_1, \dots, \latpv_S\}}$ asks for a more sophisticated method, e.g., a lifting approach as recently proposed by Yang et al.\ \cite{yang2017stein,yang2017nips}; for further reading on multiple-index models in high-dimensional estimation, see also \cite{chen2010mim,tan2018mim}.
Finally, we wish to point out a relationship between multiple-index models and (shallow) neural networks:
\begin{remark}[Shallow neural networks]
	If $\Fobs(v_1, \dots, v_S) \coloneqq \sum_{j = 1}^S \alpha_j \fobs(v_j)$ for some (non-linear) scalar function $\fobs \colon \R \to \R$ and coefficients $\alpha_1, \dots, \alpha_S \in \R$, the output rule of \eqref{eq:appl:nonlinear:mim} corresponds to a \emph{shallow neural network}:
	\begin{equation}
	y = \sum_{j = 1}^S \alpha_j \fobs(\sp{\lat}{\latpv_j}).
	\end{equation}
	Due to the rise of \emph{deep learning} \cite{lecun2015deep,goodfellow2016deep}, there is a tremendous interest in theses types of models in the past few years, but note that the learning setting of this article fits only to a limited extent into the theory of neural networks: The key objective in neural network research is to accurately predict the output variable $y$, which is clearly not achieved by the least-squares-based estimator \eqref{eq:appl:nonlinear:klassoiso}. Instead, we are rather asking for an estimate of the underlying parameter vectors, which is (if at all) only of secondary relevance for prediction problems. An interesting theoretical study related to this issue can be found in a recent work by Mondelli and Montanari \cite{mondelli2018connection}.
\end{remark}

Despite the above mentioned limitations, Proposition~\ref{prop:appl:nonlinear:mim:mp} still allows us to treat an important special case of multiple-index models, namely \emph{variable selection}:
If the index vectors are unit vectors, say $\latpv_1 = \vunit_{k_1}, \dots, \latpv_S = \vunit_{k_S}$ for a certain set of \emph{active variables} $\actsupp \coloneqq \{k_1, \dots, k_S\} \subset \{1, \dots, \d\}$, then \eqref{eq:appl:nonlinear:mim} simplifies as follows:
\begin{equation}\label{eq:appl:nonlinear:mim:vsmodel}
	y = \Fobs(\sp{\lat}{\vunit_{k_1}},\dots,\sp{\lat}{\vunit_{k_S}}) = \Fobs(s_{k_1}, \dots, s_{k_S}),
\end{equation}
which exactly corresponds to the variable selection model introduced in \eqref{eq:intro:vs}.
Adapting Proposition~\ref{prop:appl:nonlinear:mim:mp} to this specific situation reveals that the optimal target vector is indeed supported on $\actsupp$:
\begin{corollary}[Variable selection]\label{cor:appl:nonlinear:mim:vs}
	Let $\actsupp \coloneqq \{k_1, \dots, k_S\} \subset \{1, \dots, \d\}$ be a set of active variables, let Assumption~\ref{model:appl:nonlinear:mim} be satisfied with $\latpv_j = \vunit_{k_j}$ for $j = 1, \dots, S$ (see \eqref{eq:appl:nonlinear:mim:vsmodel}), and set
	\begin{equation}
		\tset \coloneqq \spann\{ \vunit_{k_1}, \dots, \vunit_{k_S}\} = \{ \latpv \in \R^\d \suchthat \supp(\latpv) \subset \actsupp \}.
	\end{equation}
	Then $\targetlatpv \coloneqq \sum_{j = 1}^S \scalfac_j \vunit_{k_j} \in \tset$ with
	\begin{equation}\label{eq:appl:nonlinear:mim:vs:scalars}
		\scalfac_j \coloneqq \mean[\Fobs(s_{k_1}, \dots, s_{k_S})s_{k_j}], \quad j = 1, \dots, S,
	\end{equation}
	minimizes the mismatch covariance on $\tset$ and we have
	\begin{equation}
		\modelcovar{\targetlatpv} = \lnorm[\big]{\mean[\Fobs(s_{k_1}, \dots, s_{k_S})\orthcompl{\proj}_{\tset} \lat]} \ .
	\end{equation}
	If the feature variables of $\lat = (s_1, \dots, s_\d)$ are independent, we particularly have $\modelcovar{\targetlatpv} = 0$.
\end{corollary}
\begin{proof}
	This is an immediate consequence of Proposition~\ref{prop:appl:nonlinear:mim:mp} and \eqref{eq:appl:nonlinear:mim:vsmodel}.
	The `in particular part' is due to the fact that $\orthcompl{\proj}_{\tset} \lat$ only depends on the non-active feature variables in $\{1, \dots, \d\} \setminus \actsupp$, which are independent from $s_{k_1}, \dots, s_{k_S}$.
\end{proof}
Combining this result with Corollary~\ref{cor:appl:nonlinear:global} shows that variable selection via the generalized Lasso \eqref{eq:appl:nonlinear:klassoiso} becomes feasible, as long as the coefficients $\scalfac_1, \dots, \scalfac_S$ of $\targetlatpv = \sum_{j = 1}^S \scalfac_j \vunit_{k_j}$ are not too small in magnitude and the sample size $n$ is sufficiently large;
more specifically, if $\solu\latpv$ is a minimizer of \eqref{eq:appl:nonlinear:klassoiso} with $\solu\latpv \approx \targetlatpv$, one may extract the set of active variables~$\actsupp$ from $\solu\latpv$ by only keeping its $S$ largest entries, so that the resulting vector is supported on $\supp(\targetlatpv) = \actsupp$.
A remarkable conclusion of Corollary~\ref{cor:appl:nonlinear:mim:vs} is that --- in stark contrast to single-index models --- we even obtain a \emph{consistent} estimator of $\targetlatpv$ in the sub-Gaussian case, supposed that the input variables are independent in Assumption~\ref{model:results:observations}. 
This observation has some interesting implications for sparse recovery tasks:
\begin{remark}[Sparse recovery]\label{rmk:appl:nonlinear:mim:vs-sim}
	The complexity of the variable selection problem does particularly rely on the number of active variables $S$.
	In practice, one often has to face high-dimensional scenarios where $S$ is significantly smaller than~$\d$, implying that the target vector $\targetlatpv$ in Corollary~\ref{cor:appl:nonlinear:mim:vs} is \emph{sparse}.
	As we will see later in Subsection~\ref{subsec:appl:nonlinear:sset}, such a sparsity prior can be easily exploited by means of an $\l{1}$-based hypothesis set, which enables accurate estimation results even when $n \ll \d$.
	
	Moreover, it is worth emphasizing that a single-index model \eqref{eq:appl:nonlinear:sim} with $\lnorm{\trulatpv}[0] \leq S$ is a special case of \eqref{eq:appl:nonlinear:mim:vsmodel}. A comparison of Proposition~\ref{prop:appl:nonlinear:sim:mp} and Corollary~\ref{cor:appl:nonlinear:mim:vs} shows that recovery of $\supp(\trulatpv)$ is possible under relatively mild assumptions, whereas recovery of~$\trulatpv$ itself may fail for non-Gaussian inputs.
	Hence, if only $\supp(\trulatpv)$ is available, one can proceed as follows: after discarding all non-active variables $\setcompl{\supp(\trulatpv)} = [\d] \setminus \supp(\trulatpv)$, one could eventually perform a more sophisticated method (operating in a much lower dimensional space) to estimate the weights of $\trulatpv$; such a strategy is well known as \emph{dimension reduction} and is successfully applied in many real-world learning tasks.
	With this in mind, Corollary~\ref{cor:appl:nonlinear:mim:vs} provides more evidence of why Lasso-type estimators can serve as a reliable \emph{variable (feature) selector}.
	Let us finally point out that the literature on the general problem of variable selection is extensive, going far beyond the specific setup considered above. For further reading, the interested reader is referred to the theoretical studies of \cite{Comminges_2012,Lafferty_2008,Bertin_2008,kubkowski2019selection} and the references therein.
\end{remark}

\subsubsection{Superimposed Observations}
\label{subsec:appl:nonlinear:msense}

A different way to generalize Assumption~\ref{model:appl:nonlinear:sim} are superpositions of single-index observations in the following sense:
\begin{assumption}[Superimposed single-index models]\label{model:appl:nonlinear:superimp}
	Let $\lat^1, \dots, \lat^M$ be independent, isotropic, centered sub-Gaussian random vectors in $\R^\d$ with $\normsubg{\lat^j} \leq \subgparam$ for some $\subgparam > 0$. The output variable is given by
	\begin{equation}\label{eq:appl:nonlinear:superimp}
		y \coloneqq \tfrac{1}{\sqrt{M}} \sum_{j = 1}^M \fobs_j(\sp{\lat^j}{\trulatpv})
	\end{equation}
	where $\trulatpv \in \R^\d \setminus \{\vnull\}$ is an unknown \emph{index vector} and $\fobs_j \colon \R \to \R$ are a measurable scalar \emph{output functions}.
	Moreover, we define the input vector as $\lat \coloneqq \tfrac{1}{\sqrt{M}} \sum_{j = 1}^M \lat^j$.
\end{assumption}
Using the isotropy of $\lat^1, \dots, \lat^M$ and Hoeffding's inequality \cite[Lem.~5.9]{vershynin2012random}, it is not hard to see that the random pair $(\lat, y)$ indeed fulfills Assumption~\ref{model:results:observations} with $\atoms = \I{\d}$ (cf. \cite[Prop.~6.5(1)]{genzel2017msense}).
The observation scheme of Assumption~\ref{model:appl:nonlinear:superimp} was extensively studied by the first author in \cite{GJ17:sampta,genzel2017msense}, motivated by an application in distributed wireless sensing.
The basic idea behind this model is that $M$ autonomous \emph{nodes} transmit their linear measurements $\sp{\lat^j}{\trulatpv}$ simultaneously to a central receiver. Due to non-linear distortions during this transmission procedure, which are modeled by the output functions $\fobs_j$, the receiver eventually measures a \emph{superposition} of all distorted signals, as described in \eqref{eq:appl:nonlinear:superimp}.

Since \cite{genzel2017msense} actually relies on similar proof techniques as this work and considers a more general setup than Assumption~\ref{model:appl:nonlinear:superimp}, we omit a detailed discussion at this point. Nevertheless, for the sake of completeness, let us provide an adaption of Proposition~\ref{prop:appl:nonlinear:sim:mp} that verifies the  feasibility of the associated recovery task in the Gaussian case:
\begin{proposition}\label{prop:appl:nonlinear:superimp:mp}
	Let Assumption~\ref{model:appl:nonlinear:superimp} be satisfied with $\lat^j \distributed \Normdistr{\vnull}{\I{\d}}$ for $j = 1, \dots, M$ and $\lnorm{\trulatpv} = 1$. Set $\tset \coloneqq \spann\{\trulatpv\}$ as target set. 
	Then $\targetlatpv \coloneqq \bar\scalfac \trulatpv \in \tset$ with
	\begin{equation}
		\bar\scalfac \coloneqq \tfrac{1}{M} \sum_{j = 1}^M \mean[\fobs_j(\sp{\lat^j}{\trulatpv}) \sp{\lat^j}{\trulatpv}]
	\end{equation}
	minimizes the mismatch covariance on $\tset$ and we have $\modelcovar{\targetlatpv} = 0$.
\end{proposition}
\begin{proof}
	This works analogously to the proof of Proposition~\ref{prop:appl:nonlinear:sim:mp}. See \cite[Prop.~6.5(1)]{genzel2017msense} for more details.
\end{proof}

\subsubsection{High-Dimensional Problems and \texorpdfstring{$\l{1}$}{L1}-Constraints}
\label{subsec:appl:nonlinear:sset}

So far we did only marginally discuss the impact of the hypothesis set $\sset$ in the above examples, in particular, the issues raised by \ref{quest:intro:complexity} at the end of Subsection~\ref{subsec:intro:statlearn}.
In fact, the presence of a constraint set leads to fundamentally different estimation results than one would obtain from ordinary (unconstrained) least squares minimization.
In the context of Theorem~\ref{thm:results:bounds:conic} for example, one just ends up with $\effdim[\conic]{\atoms^\T\sset - \targetlatpv} \asymp \d$ in \eqref{eq:results:bounds:conic:meas} if $\sset = \R^\p$ and $\atoms$ is injective.
This reflects the well-known fact that ordinary least squares are often impractical in high-dimensional scenarios with $n \ll \d$, underpinning a key challenge of many estimation tasks: identify a hypothesis set that incorporates prior knowledge about the model of interest.

In the following, we demonstrate how an $\l{1}$-constraint may significantly improve the sampling rate when the target vector $\targetlatpv$ is known to be \emph{sparse}.
Regarding the above examples of single-index models (cf. Subsection~\ref{subsec:appl:nonlinear:sim}) and variable selection (cf. Subsection~\ref{subsec:appl:nonlinear:mim}), this assumption would be simply fulfilled if $\lnorm{\trulatpv}[0] \ll \d$ in \eqref{eq:appl:nonlinear:sim} or $S \ll \d$ in \eqref{eq:appl:nonlinear:mim:vsmodel}, respectively.
More generally, let us assume that $\targetlatpv$ is $S$-sparse for a certain $S > 0$, i.e., $\lnorm{\targetlatpv}[0] \leq S$.
The Cauchy-Schwarz inequality then implies an upper bound for the $\l{1}$-norm:
\begin{equation}
	\lnorm{\targetlatpv}[1] \leq \sqrt{\lnorm{\targetlatpv}[0]} \cdot \lnorm{\targetlatpv} \leq \sqrt{S} \cdot \lnorm{\targetlatpv} \ ,
\end{equation}
which in turn suggests to select $\sset = \sqrt{S} \cdot \lnorm{\targetlatpv} \cdot \ball[1][\d]$ as hypothesis set.
In order to apply Corollary~\ref{cor:appl:nonlinear:global}, we may use that $\meanwidth{\ball[1][\d]} \lesssim \sqrt{\log(2\d)}$ (cf. \cite[Ex.~3.8]{vershynin2014estimation}), so that \eqref{eq:appl:nonlinear:global:meas} yields the following condition:
\begin{equation}\label{eq:appl:nonlinear:sset:measglobal}
	n \gtrsim \subgparam^4 \cdot \delta^{-4} \cdot \lnorm{\targetlatpv}^2 \cdot S \cdot \log(2\d).
\end{equation}
A similar argument also works for the conic mean width in Theorem~\ref{thm:results:bounds:conic} (with $\atoms = \I{\d}$):
Choosing $\sset = \lnorm{\targetlatpv}[1] \cdot \ball[1][d]$ as hypothesis set, it holds that $\meanwidth[\conic]{\sset - \targetlatpv} \lesssim \sqrt{S \cdot \log(2\d / S)}$ (cf. \cite[Prop.~3.10]{chandrasekaran2012geometry}).
Consequently, the condition of \eqref{eq:results:bounds:conic:meas} is already satisfied if
\begin{equation}\label{eq:appl:nonlinear:sset:measconic}
	n \gtrsim \subgparam^4 \cdot \delta^{-2} \cdot S \cdot \log(\tfrac{2\d}{S}).
\end{equation}

The above bounds are consistent with the traditional theory of compressed sensing and sparse signal estimation, according to which the number of needed observations essentially scales linearly with the degree of sparsity $S$, and the ambient dimension $\d$ only appears in log-terms (cf. \cite{candes2006cs,candes2006stable,donoho2006cs}).
On the other hand, \eqref{eq:appl:nonlinear:sset:measglobal} and \eqref{eq:appl:nonlinear:sset:measconic} are both only of limited practical relevance, since the involved hypothesis sets require prior knowledge of the sparsity or $\l{1}$-norm of $\targetlatpv$, respectively.
Fortunately, using the concept of local mean width (see Definition~\ref{def:proofs:results:localmw} and Theorem~\ref{thm:proofs:results:local}), it is possible to derive stable recovery results that allow for different types of model uncertainties.
A more detailed discussion of this advanced approach is contained in \cite[Sec.~2.4]{genzel2017cosparsity} and \cite[Sec.~3.1]{genzel2018hinge}.
For comprehensive overviews of compressed sensing, high-dimensional estimation problems, and the benefit of sparsity, we refer to \cite{buhlmann2011statistics,davenport2012cs,foucart2013cs,hastie2015sparsity,vershynin2018hdp}.

Finally, let us point that there are many extensions of sparsity that can be incorporated by an appropriate choice of hypothesis set, e.g., \emph{group sparsity} \cite{huang2010group}, \emph{sparsifying transformations} \cite{Elad2006}, or \emph{weighted sparsity} \cite{khajehnejad2009weighted}.
Analyzing the mean width and the resulting error bounds in each of these situations could then provide a means to adapt certain model parameters; see \cite{oymak2012recovery} for an example in weighted $\l{1}$-minimization.

\subsection{Correlated Input Variables}
\label{subsec:appl:correlated}

In this subsection, we investigate several instances of Assumption~\ref{model:results:observations} where the mixing matrix $\atoms \in \R^{\p \times \d}$ is not the identity.
The first part (Subsection~\ref{subsec:appl:correlated:undercomplete}) deals with the case of $\p \geq \d$, meaning that there are more observed than latent variables. 
In this context, it will turn out that the statement of Theorem~\ref{thm:results:bounds:global} can be significantly strengthened if $\atoms$ is approximately known and the hypothesis set is carefully adapted.
Subsection~\ref{subsec:appl:correlated:overcomplete} is then devoted to the case of $\p < \d$, which is especially useful to model noisy data. Interestingly, the mismatch principle reveals in this situation that the mismatch covariance basically corresponds to the signal-to-noise ratio of the underlying observation process.
Finally, let us emphasize that, compared to the previous subsection, we now take a rather abstract viewpoint and do not further specify to the output variable $y$.

\subsubsection{The Underdetermined Case \texorpdfstring{($\p \geq \d$)}{(p \textgreater= d)}}
\label{subsec:appl:correlated:undercomplete}

We first note that the term `underdetermined' refers to the condition $\targetlatpv \in \atoms^\T \sset$ in Theorem~\ref{thm:results:bounds:global}.
Indeed, if $\p \geq \d$, the linear system $\targetlatpv = \atoms^\T \target\xpv$ typically has many solutions in $\sset$, implying that the associated parameter vector $\target\xpv \in \sset$ is not uniquely defined.
This ambiguity is in fact a particular reason why all theoretical guarantees from Section~\ref{sec:results} do only concern the estimation of a target vector $\targetlatpv = \atoms^\T \target\xpv$ but not $\target\xpv$.
The purpose of this subsection is therefore to show that, based on approximate knowledge of $\atoms$, one may adapt the hypothesis set $\sset$ in a such way that $\target\xpv$ becomes well-defined and the mean width in \eqref{eq:results:bounds:global:meas} does not depend on $\atoms$ anymore.

To simplify the following exposition, let us assume that the mixing matrix $\atoms \in \R^{\p \times \d}$ has full rank. Hence, $\atoms$ is injective and its pseudo-inverse $\psinv{\atoms} \in \R^{\d \times \p}$ satisfies $\psinv{\atoms}\atoms = \I{\d}$.
In principle, as long as $\atoms$ is known, the issue of correlated feature variables can be resolved by first computing
\begin{equation}\label{eq:appl:correlated:undercomplete:latentextraction}
	\psinv{\atoms} \x_i = \psinv{\atoms} \atoms \lat_i = \lat_i, \quad i = 1, \dots, n,
\end{equation}
and then proceeding with isotropic data as in Subsection~\ref{subsec:appl:nonlinear}.
Unfortunately, such a pre-processing step is often unstable in practice and $\atoms$ is usually not exactly known.
The following corollary of Theorem~\ref{thm:results:bounds:global} takes a different approach that involves a transformation of the hypothesis set in \eqref{eq:results:bounds:klasso} but not of the actual input data. 
Moreover, it just assumes that an estimated mixing matrix $\tilde{\atoms} \in \R^{\p\times \d}$ is available.
\begin{corollary}[Adaptive estimation via \refklasso{\sset}]\label{cor:appl:correlated:undercomplete:adapted}
	Let Assumption~\ref{model:results:observations} be satisfied with $\p \geq \d$ and assume that $\atoms$ has full rank.
	Moreover, assume that $\tilde{\atoms} \in \R^{\p \times \d}$ is a full rank matrix such that $\atrafo \coloneqq \psinv{\tilde{\atoms}} \atoms \in \R^{\d \times \d}$ is invertible.
	Let $\tilde\sset \subset \R^\d$ be a bounded, convex subset and fix a vector $\targetlatpv \in \atrafo^\T \tilde\sset$.
	Then there exists a numerical constant $C > 0$ such that for every $\delta \in \intvopcl{0}{1}$, the following holds true with probability at least $1 - 5\exp(- C \cdot \subgparam^{-4} \cdot \delta^2 \cdot n)$:
	If the number of observed samples obeys
	\begin{equation}\label{eq:appl:correlated:undercomplete:adapted:meas}
		n \gtrsim \subgparam^4 \cdot \delta^{-4} \cdot \opnorm{\atrafo}^2 \cdot \effdim{\tilde\sset},
	\end{equation}
	then every minimizer $\solu\xpv$ of \refklasso{\sset} with $\sset \coloneqq (\psinv{\tilde{\atoms}})^\T\tilde\sset$ satisfies
	\begin{equation}\label{eq:appl:correlated:undercomplete:adapted:bound}
		\lnorm{\tilde{\atoms}^\T\solu\xpv - \targetlatpv} \lesssim \opnorm{\atrafo^{-1}} \cdot \big( \max\{1, \subgparam \cdot \modeldev{\targetlatpv}\} \cdot \delta + \modelcovar{\targetlatpv} \big) + \lnorm{(\I{\d} - \atrafo^{-\T})\targetlatpv}
	\end{equation}
	and
	\begin{equation}\label{eq:appl:correlated:undercomplete:adapted:boundxpv}
		\lnorm{\solu\xpv - \target\xpv} \lesssim \opnorm{\atrafo^{-1}\psinv{\tilde{\atoms}}} \cdot \big( \max\{1, \subgparam \cdot \modeldev{\targetlatpv}\} \cdot \delta + \modelcovar{\targetlatpv} \big)
	\end{equation}
	where $\target\xpv \coloneqq (\psinv{\tilde{\atoms}})^\T \atrafo^{-\T} \targetlatpv \in \sset$.
\end{corollary}
\begin{proof}
	We simply apply Theorem~\ref{thm:results:bounds:global} with $\sset = (\psinv{\tilde{\atoms}})^\T\tilde\sset$.
	Let us first observe that
	\begin{equation}
		\targetlatpv \in \atrafo^\T \tilde\sset = \atoms^\T (\psinv{\tilde{\atoms}})^\T \tilde\sset = \atoms^\T \sset.
	\end{equation}
	Furthermore, Slepian’s inequality \cite[Lem.~8.25]{foucart2013cs} yields
	\begin{equation}
		\meanwidth{\atoms^\T \sset} = \meanwidth{\atrafo^\T \tilde\sset} \leq \opnorm{\atrafo} \cdot \meanwidth{\tilde\sset},
	\end{equation}
	so that the assumption \eqref{eq:appl:correlated:undercomplete:adapted:meas} implies \eqref{eq:results:bounds:global:meas}.
	On the event of Theorem~\ref{thm:results:bounds:global}, we therefore obtain
	\begin{align}
		\lnorm{\tilde{\atoms}^\T\solu\xpv - \targetlatpv} &\leq \lnorm{\tilde{\atoms}^\T\solu\xpv - \atrafo^{-\T} \targetlatpv} + \lnorm{(\I{\d} - \atrafo^{-\T})\targetlatpv} \\*
		&= \opnorm{\atrafo^{-\T}} \cdot \lnorm{\atrafo^{\T} \tilde{\atoms}^\T\solu\xpv - \targetlatpv} + \lnorm{(\I{\d} - \atrafo^{-\T})\targetlatpv} \\
		&= \opnorm{\atrafo^{-1}} \cdot \lnorm{\atoms^{\T} \underbrace{(\psinv{\tilde{\atoms}})^\T \tilde{\atoms}^\T\solu\xpv}_{\stackrel{(\ast)}{=} \solu\xpv} - \targetlatpv} + \lnorm{(\I{\d} - \atrafo^{-\T})\targetlatpv} \\
		&\stackrel{\eqref{eq:results:bounds:global:bound}}{\lesssim} \opnorm{\atrafo^{-1}} \cdot \big( \max\{1, \subgparam \cdot \modeldev{\targetlatpv}\} \cdot \delta + \modelcovar{\targetlatpv} \big) + \lnorm{(\I{\d} - \atrafo^{-\T})\targetlatpv} \ ,
	\end{align}
	where $(\ast)$ follows from $\solu\xpv \in \ran((\psinv{\tilde{\atoms}})^\T)$ and the identity $(\psinv{\tilde{\atoms}})^\T \tilde{\atoms}^\T (\psinv{\tilde{\atoms}})^\T = (\psinv{\tilde{\atoms}})^\T$.
	
	In order to establish \eqref{eq:appl:correlated:undercomplete:adapted:boundxpv}, we observe that
	\begin{equation}
		(\psinv{\tilde{\atoms}})^\T = (\psinv{\tilde{\atoms}})^\T \atrafo^{-\T} \atrafo^\T =  (\psinv{\tilde{\atoms}})^\T \atrafo^{-\T} \atoms^\T (\psinv{\tilde{\atoms}})^\T.
	\end{equation}
	Since $\solu\xpv \in \ran((\psinv{\tilde{\atoms}})^\T)$, this implies $\solu\xpv = (\psinv{\tilde{\atoms}})^\T \atrafo^{-\T} \atoms^\T \solu\xpv$.
	Using the definition of $\target\xpv$, we therefore end up with
	\begin{align}
		\lnorm{\solu\xpv - \target\xpv} = \lnorm{(\psinv{\tilde{\atoms}})^\T \atrafo^{-\T} (\atoms^\T \solu\xpv - \targetlatpv)} \leq \opnorm{\atrafo^{-1}\psinv{\tilde{\atoms}}} \cdot \lnorm{\atoms^\T \solu\xpv - \targetlatpv} \ ,
	\end{align}
	and \eqref{eq:appl:correlated:undercomplete:adapted:boundxpv} now follows from \eqref{eq:results:bounds:global:bound} again.
\end{proof}
If $\tilde{\atoms} = \atoms$ (and therefore $\atrafo = \I{\d}$), the error bound of \eqref{eq:appl:correlated:undercomplete:adapted:bound} still coincides with \eqref{eq:results:bounds:global:bound}, whereas \eqref{eq:appl:correlated:undercomplete:adapted:boundxpv} provides an approximation guarantee for the actual solution  vector $\solu\xpv$ of \eqref{eq:results:bounds:klasso}.
This achievement of Corollary~\ref{cor:appl:correlated:undercomplete:adapted} is due to the specific choice of $\sset$, which is the image of a lower dimensional subset $\tilde\sset$ that accounts for the effect of the mixing matrix $\atoms$ and thereby resolves the ambiguity pointed out above.
The only price to pay is the presence of an additional factor $\opnorm{\psinv{\atoms}}$ in \eqref{eq:appl:correlated:undercomplete:adapted:boundxpv}.
It is also noteworthy that the complexity in \eqref{eq:appl:correlated:undercomplete:adapted:meas} is now measured in terms of the hypothesis set $\tilde{\sset} \subset \R^\d$, allowing us to exploit structural assumptions on $\targetlatpv$ in a direct manner.
	
A practical feature of Corollary~\ref{cor:appl:correlated:undercomplete:adapted} is that it even remains valid when the estimators $\solu\xpv$ and $\solu\latpv \coloneqq \tilde{\atoms}^\T\solu\xpv$ are based on an approximate mixing matrix $\tilde{\atoms}$.
The quality of this approximation is essentially captured by the transformation matrix $\atrafo = \psinv{\tilde{\atoms}} \atoms$, in the sense that $\atrafo \approx \I{\d}$ if $\tilde{\atoms} \approx \atoms$.
This particularly affects the sample size in \eqref{eq:appl:correlated:undercomplete:adapted:meas} as well as the error bound in \eqref{eq:appl:correlated:undercomplete:adapted:bound}, which is in fact only significant if $\tilde{\atoms}$ is sufficiently close to $\atoms$.
Remarkably, the output vector $\solu\xpv$ may still constitute a consistent estimator of $\target\xpv$, since $\atrafo$ just appears in terms of a multiplicative constant in \eqref{eq:appl:correlated:undercomplete:adapted:boundxpv}. 
However, we point out that \eqref{eq:appl:correlated:undercomplete:adapted:boundxpv} is actually a weaker statement than \eqref{eq:appl:correlated:undercomplete:adapted:bound}, as the former does not necessarily address our initial concern of learning semi-parametric output rules.

Let us conclude with some possible applications of Corollary~\ref{cor:appl:correlated:undercomplete:adapted}, and in particular \eqref{eq:appl:correlated:undercomplete:adapted:boundxpv}:
\begin{remark}
\begin{rmklist}
\item\label{rmk:appl:correlated:underdetermined:applications}
	\emph{Clustered variables.}
	A situation of practical interest are \emph{clustered features}, where the components of $\x$ are divided into (disjoint) groups of strongly correlated variables.
	In this case, the associated hypothesis set $\sset \coloneqq (\psinv{\tilde{\atoms}})^\T\tilde\sset$ would basically turn into a group constraint, which leads to  a block-like support of the parameter vector $\target\xpv \in \sset$.
	A formal study of such a problem scenario is in principle straightforward, but however requires a certain technical effort and therefore goes beyond the scope of this paper. 
\item\label{rmk:appl:correlated:underdetermined:unknownmixing}
	\emph{Unknown mixing matrix.}
	The above strategy clearly relies on (partial) knowledge of $\atoms$.
	The same problem persists in the general setup of Theorem~\ref{thm:results:bounds:global}, as the mixing matrix is needed to construct the estimator $\solu\latpv \coloneqq \atoms^\T \solu\xpv$.
	Thus, if $\atoms$ is unknown, the error bound \eqref{eq:results:bounds:global:bound} is rather of theoretical interest, merely indicating that --- in principle --- estimation via \eqref{eq:results:bounds:klasso} is feasible.
	Nevertheless, one may still draw heuristic conclusions from Theorem~\ref{thm:results:bounds:global}. For instance, if $\atoms$ generates clustered features as in \ref{rmk:appl:correlated:underdetermined:applications}, the support pattern of the (possibly non-unique) parameter vector $\target\xpv$ would at least indicate what clusters are active. This approach was already discussed in a previous work by the authors \cite{genzel2016fs}, which is inspired by the example of \emph{mass spectrometry data} (see also Remark~\ref{rmk:appl:correlated:overdetermined}).
\item\label{rmk:appl:correlated:underdetermined:covarmatrix}
	\emph{The case of $\p = \d$.} In this situation, the mixing matrix $\atoms$ is invertible and one may simply select $\tilde{\atoms} = \I{\d}$ in Corollary~\ref{cor:appl:correlated:undercomplete:adapted} (implying that $\atrafo = \atoms$ and $\sset = \tilde\sset$).
	For example, this setup proves very useful when considering a single-index model with \emph{unknown covariance structure}: 
	Let Assumption~\ref{model:results:observations} be satisfied with $y = \fobs(\sp{\x}{\tru\xpv})$ for some index vector $\tru\xpv \in \sset \subset \R^\d$ and output function $\fobs \colon \R \to \R$. Thus, compared to Assumption~\ref{model:appl:nonlinear:sim}, the input vector $\x = \atoms \lat$ is not necessarily isotropic and its covariance matrix $\Covmatr \coloneqq \atoms \atoms^\T \in \R^{\d \times \d}$ is unknown.
	An application of Corollary~\ref{cor:appl:correlated:undercomplete:adapted} with $\targetlatpv = \atoms^\T \tru\xpv$ then shows that, according to \eqref{eq:appl:correlated:undercomplete:adapted:boundxpv}, the minimizer $\solu\xpv$ is an estimator of the unknown index vector $\tru\xpv = \target\xpv$. Noteworthy, the generalized Lasso \eqref{eq:results:bounds:klasso} does not require any knowledge of $\Covmatr$ in this case.
	\qedhere
\end{rmklist}
\label{rmk:appl:correlated:underdetermined}
\end{remark}

\subsubsection{The Overdetermined Case \texorpdfstring{($\p < \d$)}{(p \textless \ d)} and Noisy Data}
\label{subsec:appl:correlated:overcomplete}

In most real-world applications, the input data is corrupted by noise.
With regard to our statistical model setup, this basically means that the input vector $\x = \atoms \lat$ is generated from two types of latent factors, namely those affecting the output variable $y$ and those that do not.
Let us make this idea more precise:
\begin{assumption}[Noisy data]\label{model:appl:correlated:overcomplete}
	Let Assumption~\ref{model:results:observations} be satisfied and assume that the latent factors split into two parts, $\lat = (\latsig,\latnoise)$, where $\latsig = (v_1, \dots, v_{\d_1}) \in \R^{\d_1}$ are referred to as \emph{signal variables} and $\latnoise = (n_1, \dots, n_{\d_2}) \in \R^{\d_2}$ as \emph{noise variables} ($\implies \ \d = \d_1 + \d_2$).
	The output variable does only depend on the signal variables, i.e.,
	\begin{equation}\label{eq:appl:correlated:overcomplete}
		y = \Fobs(\latsig) = \Fobs(v_1, \dots, v_{\d_1})
	\end{equation}
	for a measurable \emph{output function} $\Fobs \colon \R^{\d_1} \to \R$ which can be random (independently of $\latsig$), and we also assume that $y$ and $\latnoise$ are independent.
	According to the above partition, the mixing matrix takes the form $\atoms = [\atoms_\latsig, \atoms_\latnoise] \in \R^{\p \times \d}$ where $\atoms_\latsig \in \R^{\p \times \d_1}$ is associated with $\latsig$ and $\atoms_\latnoise \in \R^{\p \times \d_2}$ with $\latnoise$.
	Thus, the input vector can be decomposed as $\x = \atoms \lat = \atoms_\latsig \latsig + \atoms_\latnoise \latnoise$.
\end{assumption}
The key concern of Assumption~\ref{model:appl:correlated:overcomplete} is that the factors of $\latnoise$ do not contribute to the output variable $y$ and are therefore regarded as \emph{noise}.
In this context, it could easily happen that $\d > \p$, so that $\atoms$ is not injective anymore.
Retrieving the latent factors $\lat$ from $\x$ then becomes impossible, even if $\atoms$ is exactly known (cf. \eqref{eq:appl:correlated:undercomplete:latentextraction}).
A typical example case would be component-wise ``background'' noise within the input vector $\x$, i.e., $\atoms_\latnoise \in \R^{\p \times \p}$ is a diagonal matrix and we have $\d = \d_1 + \d_2 = \d_1 + \p > \p$.
Moreover, we emphasize that Assumption~\ref{model:appl:correlated:overcomplete} should not be confused with the variable selection problem from \eqref{eq:appl:nonlinear:mim:vsmodel}. 
The latter comes along with the task to detect the \emph{unknown} subset of active variables, whereas the above partition into signal and noise variables is known to a certain extent.\footnote{This does not necessarily mean that $\latsig$ and $\latnoise$ are directly accessible, but rather that there exists (partial) knowledge of the mixing matrix $\atoms = [\atoms_\latsig, \atoms_\latnoise]$ and the associated noise pattern.}
For that reason, the order of the variables in $\lat = (\latsig,\latnoise)$ is rather a matter of convenience and does not restrict the generality.

Let us now apply the mismatch principle (Recipe~\ref{rec:results:mismatch}) to investigate the estimation capacity of the generalized Lasso \eqref{eq:results:bounds:klasso} under the noise model of Assumption~\ref{model:appl:correlated:overcomplete}. 
According to the output rule of \eqref{eq:appl:correlated:overcomplete}, it is natural to select a target set of the form $\tset = \tset_\latsig \times \R^{\d_2} \subset \R^{\d_1 + \d_2}$. Indeed, given any $\targetlatpv = (\targetlatpv_\latsig, \targetlatpv_\latnoise) \in \tset$, we are actually only interested in the signal component $\targetlatpv_\latsig \in \tset_\latsig$ --- encoding the desired (parametric) information about \eqref{eq:appl:correlated:overcomplete} in terms of $\tset_\latsig \subset \R^{\d_1}$ --- while the noise component $\targetlatpv_\latnoise \in \R^{\d_2}$ is irrelevant for our purposes.
This important observation is particularly consistent with the following decomposition of the mismatch covariance:
\begin{proposition}\label{prop:appl:correlated:overcomplete:modelcovar}
	Let Assumption~\ref{model:appl:correlated:overcomplete} be satisfied. 
	For every $\targetlatpv = (\targetlatpv_\latsig, \targetlatpv_\latnoise) \in \R^{\d_1+\d_2}$, we have that
	\begin{equation}\label{eq:appl:correlated:overcomplete:modelcovar}
		\modelcovar[\lat,y]{\targetlatpv}^2 = \lnorm[\big]{\mean[(y - \sp{\latsig}{\targetlatpv_\latsig}) \latsig]}^2 + \lnorm{\targetlatpv_\latnoise}^2 = \modelcovar[\latsig,y]{\targetlatpv_\latsig}^2 + \lnorm{\targetlatpv_\latnoise}^2 \ .
	\end{equation}
\end{proposition}
\begin{proof}
	Using the isotropy of $\lat = (\latsig, \latnoise)$ and the independence of $y$ and $\latnoise$, we obtain
	\begin{align}
		\modelcovar[\lat,y]{\targetlatpv}^2 &= \lnorm[\big]{\mean[(y - \sp{\lat}{\targetlatpv}) \lat]}^2 \\*
		&= \lnorm[\big]{\mean[(y - \sp{\latsig}{\targetlatpv_\latsig} - \sp{\latnoise}{\targetlatpv_\latnoise}) \latsig]}^2 + \lnorm[\big]{\mean[(y - \sp{\latsig}{\targetlatpv_\latsig} - \sp{\latnoise}{\targetlatpv_\latnoise}) \latnoise]}^2 \\
		&= \lnorm[\big]{\mean[(y - \sp{\latsig}{\targetlatpv_\latsig}) \latsig]}^2 + \lnorm[\big]{\mean[\sp{\latnoise}{\targetlatpv_\latnoise} \latnoise]}^2 = \lnorm[\big]{\mean[(y - \sp{\latsig}{\targetlatpv_\latsig}) \latsig]}^2 + \lnorm{\targetlatpv_\latnoise}^2  \ . \qedhere
	\end{align}
\end{proof}
The identity \eqref{eq:appl:correlated:overcomplete:modelcovar} asks us to select a target vector $\targetlatpv = (\targetlatpv_\latsig, \targetlatpv_\latnoise) \in \tset$ such that the ``signal-related'' mismatch covariance $\modelcovar[\latsig,y]{\targetlatpv_\latsig}$ is sufficiently small, and at the same time, the magnitude of noise component $\targetlatpv_\latnoise$ must not be too large.
This task is unfortunately not straightforward because the transformed hypothesis set $\atoms^\T\sset \subset \R^{\d}$ is not full-dimensional if $\d > \p$, so that the simplification of Remark~\ref{rmk:results:bounds}\ref{rmk:results:bounds:mismatchpractical} is not applicable anymore.
For example, solely enforcing $\targetlatpv_\latnoise = \vnull$ could eventually lead to a large value of $\modelcovar[\latsig,y]{\targetlatpv_\latsig}$, as the condition $\targetlatpv \in \atoms^\T\sset$ implies an `overdetermined' equation system with $\targetlatpv_\latsig = \atoms_\latsig^\T \target\xpv$ and $\targetlatpv_\latnoise = \atoms_\latnoise^\T \target\xpv$.
More generally speaking, the minimization problem \eqref{eq:results:mismatch:modelcovar:minimization} in Recipe~\ref{rec:results:mismatch}\ref{rec:results:mismatch:modelcovar} is not separable in the signal and noise component, which can cause difficulties.

For that reason, we do not exactly apply Recipe~\ref{rec:results:mismatch}\ref{rec:results:mismatch:modelcovar} in the following, but rather construct an admissible target vector $\targetlatpv = (\targetlatpv_\latsig, \targetlatpv_\latnoise) \in \tset \intersec \atoms^\T \sset$ that approximately solves \eqref{eq:results:mismatch:modelcovar:minimization}.
Motivated by the decomposition of \eqref{eq:appl:correlated:overcomplete:modelcovar}, our basic idea is to first minimize $\latpv_\latsig \mapsto \modelcovar[\latsig,y]{\latpv_\latsig}$ and then to choose among all \emph{admissible} vectors in $\atoms_\latnoise^\T \sset$ the one of minimal magnitude.
Such an approach appears in fact quite natural, since our primary goal is to accurately estimate the signal component $\targetlatpv_\latsig$ by $\atoms_\latsig^\T \solu\xpv$, with $\solu\xpv$ being a minimizer of \eqref{eq:results:bounds:klasso}.
Thus, let us simply define
\begin{equation}\label{eq:appl:correlated:overcomplete:modelcovarsig}
	\targetlatpv_\latsig \coloneqq \argmin_{\latpv_\latsig \in \tset_\latsig \intersec \atoms_\latsig^\T\sset} \modelcovar[\latsig,y]{\latpv_\latsig}.
\end{equation}
Denoting the preimage of $\targetlatpv_\latsig$ by $\sset_{\latsig} \coloneqq \{ \xpv \in \sset \suchthat \atoms_\latsig^\T \xpv = \targetlatpv_\latsig \} \subset \R^\p$, we are now allowed to select any vector $\targetlatpv_\latnoise \in \atoms_\latnoise^\T \sset_{\latsig} \subset \R^{\d_2}$.
The one of minimal $\l{2}$-norm is therefore
\begin{equation}
	\targetlatpv_\latnoise \coloneqq \argmin_{\latpv_\latnoise \in \atoms_\latnoise^\T \sset_{\latsig}} \lnorm{\latpv_\latnoise} = \argmin_{\xpv \in  \sset_{\latsig}} \lnorm{\atoms_\latnoise^\T \xpv} \ ,
\end{equation}
and we also set $\modelcovar[\latnoise]{\targetlatpv_\latsig} \coloneqq \lnorm{\targetlatpv_\latnoise}$, in order to indicate the dependence $\targetlatpv_\latsig$.

According to Recipe~\ref{rec:results:mismatch}\ref{rec:results:mismatch:bound}, this choice of $\targetlatpv$ yields the following bound:
\begin{multline}
	\lnorm{\atoms_\latsig^\T\solu\xpv - \targetlatpv_\latsig} \leq \lnorm{\atoms^\T\solu\xpv - \targetlatpv} \\ \stackrel{\mathllap{\eqref{eq:results:bounds:global:bound-decay}}}{\lesssim} \max\{\subgparam, \subgparam^2 \cdot \modeldev{\targetlatpv}\} \cdot \Big(\frac{\effdim{\atoms^\T\sset}}{n}\Big)^{1/4} + \Big( \modelcovar[\latsig,y]{\targetlatpv_\latsig}^2 + \modelcovar[\latnoise]{\targetlatpv_\latsig}^2 \Big)^{1/2}. \label{eq:appl:correlated:overcomplete:errorbnd}
\end{multline}
The significance of this error estimate clearly depends on the size of the bias terms $\modelcovar[\latsig,y]{\targetlatpv_\latsig}$ and $\modelcovar[\latnoise]{\targetlatpv_\latsig}$.
The former parameter corresponds to the mismatch covariance of the noiseless data model $(\latsig, y)$ and can be treated as in the isotropic case of Subsection~\ref{subsec:appl:nonlinear}. 
On the other hand, $\modelcovar[\latnoise]{\targetlatpv_\latsig}$ captures the impact of the noise variables $\latnoise$ on the input data $\x$.
As already pointed out above, the presence of such a noise term is inevitable to some extent, since the ``ideal'' target vector $\target{\tilde{\latpv}} \coloneqq (\targetlatpv_\latsig, \vnull) \in \R^{\d_1 + \d_2}$ is usually not contained in $\atoms^\T\sset$.\footnote{Or equivalently, there does not exist a linear hypothesis function $\func = \sp{\cdot}{\target\xpv}$ with $\target\xpv \in \sset$ and $\target{\tilde{\latpv}} = \atoms^\T \target\xpv$.}
From a more practical perspective, $\modelcovar[\latnoise]{\targetlatpv_\latsig}$ serves as a measure of the \emph{noise power} of the underlying observation model $(\x,y)$ and its reciprocal, $1 / \modelcovar[\latnoise]{\targetlatpv_\latsig}$, can be interpreted as the signal-to-noise ratio.

\begin{remark}
\begin{rmklist}
\item\label{rmk:appl:correlated:overdetermined:fspaper}
	\emph{Previous approaches.} The above methodology is inspired by \cite{genzel2016fs}, where the authors investigated noisy data in conjunction with Gaussian single-index models.
	The key idea of that work is to specify a parameter vector $\target\xpv \in \R^{\p}$ such that the linear mapping $\x \mapsto \sp{\x}{\target\xpv}$ imitates the output model for $y$ as well as possible, coining the notion of \emph{optimal representations}.
	Interestingly, the mismatch principle essentially leads us to the same results as in \cite{genzel2016fs}, but in a much more systematic and less technical way.
	In particular, the error bound of \eqref{eq:appl:correlated:overcomplete:errorbnd} is not exclusively restricted to single-index models.
\item\label{rmk:appl:correlated:overdetermined:massspec}
	\emph{Real-world data.} A natural way to think of Assumption~\ref{model:appl:correlated:overcomplete} is that the columns of $\atoms_\latsig$ form building blocks (atoms) of the input vector and the signal variables $\latsig_i$ determine their individual contributions to each sample $\x_i$, $i = 1, \dots, n$.
	The same interpretation applies to the noise variables $\latnoise_i$, but they do not affect the actual output $y_i$.
		
	A prototypical example of real-world data that was extensively studied in \cite{genzel2016fs} is so-called \emph{mass spectrometry data}. In this case, the columns of $\atoms_\latsig$ do simply form (discretized) Gaussian-shaped peaks, each one representing a certain type of \emph{molecule} (or a compound). The signal variables are in turn proportional to the molecular concentration of these molecules.
	On the other hand, the output variable $y_i$ only takes values in $\{-1,+1\}$, indicating the \emph{health status} of a human individual (e.g., healthy or suffering from a specific disease).
	The key task is now to use empirical risk minimization to identify those molecules (features) that are relevant to the health status, eventually enabling for early diagnostics and a deeper understanding of pathological mechanisms.
	In the context of statistical learning, this challenge is closely related to the variable selection model considered in \eqref{eq:appl:nonlinear:mim:vsmodel}.
	For more details about this specific problem setup, we refer to \cite{genzel2016fs,genzel2015master}, focusing on a theoretical analysis, as well as to \cite{conrad2015spa} for practical aspects and numerical experiments.
	\qedhere
\end{rmklist} \label{rmk:appl:correlated:overdetermined}
\end{remark}

\section{Related Literature and Approaches}
\label{sec:literature}

This part gives a brief overview of several recent results from the literature that are related to the main ideas and conclusions of this work.
Our intention here is not to provide a complete list of existing works but rather to point out important conceptual similarities and differences of our approach to others.
In particular, we do only marginally address the specific examples of output models that were already investigated in Subsection~\ref{subsec:appl:nonlinear}.

\subsection{Statistical Learning Theory}
\label{subsec:literature:statlearn}

The setting and terminology of this work is clearly based on statistical learning theory; see \cite{vapnik1998learning,cucker2007learning,hastie2009elements,shalev2014understanding} for comprehensive overviews.
As already pointed out in the introduction, one of the key objectives in this research area is to study empirical risk minimization \eqref{eq:intro:emplossmin-general} as a means to approximate the inaccessible problem of expected risk minimization \eqref{eq:intro:explossmin-general}.
Despite many different variants, there are basically two common ways to assess the quality of a solution $\solu\func \in \funcclass$ to \eqref{eq:intro:emplossmin-general}:
\begin{alignat}{2}
	&\text{Prediction:} && \qquad \mean{}_{(\x,y)}[(y - \solu\func(\x))^2] - \mean{}_{(\x,y)}[(y - \soluexp\func(\x))^2], \label{eq:literature:statlearn:prederr} \\*[.5\baselineskip] 
	&\text{Estimation:} && \qquad \mean{}_{\x}[(\solu\func(\x) - \soluexp\func(\x))^2], \label{eq:literature:statlearn:esterr}
\end{alignat}
where $\soluexp\func \in \funcclass$ is a minimizer of \eqref{eq:intro:explossmin-general}.
The main purpose of \eqref{eq:literature:statlearn:prederr} is to compare the predictive power of the empirical risk minimizer $\solu\func$ to the best possible in $\funcclass$, which is $\soluexp\func$ by definition.
In contrast, the estimation error \eqref{eq:literature:statlearn:esterr} measures how well $\solu\func$ approximates $\soluexp\func$, while not depending on the true output variable $y$.

Among a large amount of works on estimation and prediction problems, we think that the most related ones are those by Mendelson \cite{mendelson2014learning,mendelson2014learninggeneral}, which establish a very general learning framework for empirical risk minimization.
His results are based on a few complexity parameters that relate the noise level and geometric properties of the hypothesis set to the estimation error and sample size.
A crucial feature of Mendelson's approach is that it does not rely on concentration inequalities but rather on a mild \emph{small ball condition}.
In that way, the theoretical bounds of \cite{mendelson2014learning,mendelson2014learninggeneral} remain even valid when the input and output variables are heavy-tailed.
This method turned out to be very useful for controlling non-negative empirical processes in general, and therefore found many applications beyond statistical learning.

While the problem setting of Mendelson certainly resembles ours, there are several crucial differences.
Firstly, we focus on much more restrictive model assumptions, in particular, sub-Gaussian data and \emph{linear} hypothesis functions. 
This makes our results less general, but at the same time, more explicit and accessible for concrete observation models.
The second important difference comes along with the first one: Instead of simply approximating the expected risk minimizer as in \eqref{eq:literature:statlearn:esterr}, we refine the classical estimation problem by the notion of target sets.
This allows us to encode the desired (parametric) information about the underlying output rule, which in turn forms the basis of the mismatch principle (Recipe~\ref{rec:results:mismatch}).
For a more technical comparison of our approach to \cite{mendelson2014learning,mendelson2014learninggeneral} see Remark~\ref{rmk:proofs:local:techniques} as well as Remark~\ref{rmk:results:outofbox}.

Finally, it is worth noting that the prediction error \eqref{eq:literature:statlearn:prederr} is only of minor interest in this work.
Indeed, since we restrict ourselves to linear hypothesis functions in $\funcclass$, one cannot expect that $\func(\x) = \sp{\x}{\xpv}$ yields a reliable predictor of $y$, unless it linearly depends on $\x$.

\subsection{Signal Processing and Compressed Sensing}
\label{subsec:literature:cs}

A considerable amount of the recent literature on (high-dimensional) signal processing and compressed sensing deals with \emph{non-linear} distortions of linear sampling schemes, such as quantized, especially $1$-bit compressed sensing \cite{boufounos2008onebit,zymnis2010compressed}.
In this context, single-index models proved very useful (cf. Assumption~\ref{model:appl:nonlinear:sim}), as they permit many different types of perturbations by means of the output function.
A remarkable line of research by Plan, Vershynin, and collaborators \cite{plan2013onebit,plan2013robust,plan2014tessellation,ai2014onebitsubgauss,plan2014highdim,plan2015lasso} as well as related works \cite{thrampoulidis2015lasso,knudson2016dithering,genzel2016estimation,baraniuk2017dithering,goldstein2018nongauss,thrampoulidis2018dithering,goldstein2016structured} has made significant progress in the statistical analysis of these model situations.
Somewhat surprisingly, it turned out that recovery of the index vector is already achievable by convex programming, even when the output function is unknown.
Of particular relevance to our approach is \cite{plan2015lasso}, where the generalized Lasso \eqref{eq:results:bounds:klasso} was studied for the first time with Gaussian single-index observations. Although not mentioned explicitly, the key argument of \cite{plan2015lasso} is based on an application of the orthogonality principle (cf. \eqref{eq:intro:orthogprinciple}), eventually leading to the same conclusions as in Subsection~\ref{subsec:appl:nonlinear:sim}.
In that light, our results show that the original ideas of \cite{plan2015lasso} apply to a much wider class of semi-parametric models.
Going beyond the case of single-index models is also a key aspect of a very recent work by Sattar and Oymak \cite{sattar2019quickly}. Similarly to our approach, they do not assume a \emph{realizable model}, i.e., a specific functional relationship between $\x$ and $y$. But on the other hand, their results do only yield estimation guarantees for the expected risk minimizer, with a much stronger focus on algorithmic issues and heavier tailed (sub-exponential) input data.

Next, it is worth pointing out that the discussion of Proposition~\ref{prop:appl:nonlinear:sim:mp} reveals several shortcomings of learning single-index models via the generalized Lasso, most importantly, the inconsistency of \eqref{eq:results:bounds:klasso} in the non-Gaussian case.
A recent series of papers by Yang et al.\ \cite{yang2015sparse,yang2017high,yang2017stein,yang2017nips} tackles this issue by investigating adaptive, possibly non-convex estimators.
While these findings are very interesting and actually include more complicated estimation tasks, such as phase retrieval (cf. Example~\ref{ex:appl:nonlinear:sim}\ref{ex:appl:nonlinear:sim:evenfunc}) and multiple-index models (cf. Subsection~\ref{subsec:appl:nonlinear:mim}), the proposed algorithms either assume knowledge of the output function or of the density function of the input vector.
Such prerequisites are in fact somewhat different from the conception of this work, but nevertheless, we think that these approaches are a promising direction of future research (see also Section~\ref{sec:conlcusion}).

In the situation of sparse index vectors and $\l{1}$-constraints, our results allow us to reproduce classical recovery guarantees from compressed sensing (cf. Subsection~\ref{subsec:appl:nonlinear:sset}). But despite obvious similarities, the theoretical foundations of compressed sensing rely on quite different concepts, most prominently the \emph{null space property} and \emph{restricted isometry property} \cite{foucart2013cs}.
This methodology indeed enables the treatment of many practically relevant sensing schemes that are by far not included in Assumption~\ref{model:results:observations}, e.g., Fourier subsampling or circular matrices.

\subsection{Correlated Features and Factor Analysis}
\label{subsec:literature:correl}

An important practical concern of this work is to allow for correlated variables in the input vector $\x$.
However, most traditional approaches in signal estimation theory and variable selection rely on assumptions that preclude stronger correlations between features, such as restricted isometry, restricted eigenvalue conditions, or irrepresentability, e.g., see \cite{geer2009lasso}.
For that reason, various strategies have been proposed in the literature to deal with these delicate situations, for instance, \emph{hierarchical clustering} \cite{buehlmann2013clustering} or \mbox{\emph{OSCAR/OWL}} \cite{bondell2008oscar,figueiredo2014owl,figueiredo2016owl}.
The recent work of Li et al.\ \cite{li2018graph} provides a general methodology to incorporate correlations between covariates by \emph{graph-based regularization}; see also Section~2 therein for a discussion of related literature in that direction.
Although the aforementioned results are limited to more specific model settings like noisy linear regression, they bear a resemblance to the basic idea of Corollary~\ref{cor:appl:correlated:undercomplete:adapted}, namely using the (estimated) mixing matrix $\tilde{\atoms}$ to adapt the constraint set of \eqref{eq:results:bounds:klasso}.

If the correlation structure of $\x = \atoms \lat$ is unknown, one rather aims at recovering the mixing matrix $\atoms$ based on the available sample set $\{\x_i\}_{i = 1}^n$.
This task is usually referred to as \emph{factor analysis}, or more generally, \emph{matrix factorization}, e.g., see \cite[Chap.~14]{hastie2009elements}.
Unfortunately, such a factor analysis is often unstable, especially when the problem is high-dimensional ($n \ll \p, \d$) and noise is present.
It was already succinctly emphasized by Vapnik in \cite[p.~12]{vapnik1998learning} that a direct approach is preferable in these situations:
\begin{quote}\itshape
	``If you possess a restricted amount of information for solving some problem, try to solve the problem directly and never solve the more general problem as an intermediate step. It is possible that the available information is sufficient for a direct solution but is insufficient for solving a more general intermediate problem.''
\end{quote}
The procedure of the generalized Lasso \eqref{eq:results:bounds:klasso} precisely reflects this perspective, as it only takes the ``unspoiled'' data $\{(\x_i,y_i)\}_{i = 1}^n$ as input.
If $\solu\xpv$ is a minimizer of \eqref{eq:results:bounds:klasso}, our theoretical findings show that  computing the actual estimator $\solu\latpv = \atoms^\T\solu\xpv$ still requires $\atoms$, but this merely concerns a simple post-processing step. 
In that way, one may hope that the entire estimation process is less sensitive to noise and a small sample size.
Moreover, as mentioned in Remark~\ref{rmk:appl:correlated:underdetermined}\ref{rmk:appl:correlated:underdetermined:unknownmixing}, it is often possible to extract the information of interest directly from $\solu\xpv$ by means of domain knowledge, rather than to explicitly compute $\solu\latpv = \atoms^\T\solu\xpv$.

\section{Conclusion and Outlook}
\label{sec:conlcusion}

The key achievements of this work are the mismatch principle (Recipe~\ref{rec:results:mismatch}) and the associated error bounds (Theorem~\ref{thm:results:bounds:global}, Theorem~\ref{thm:results:bounds:conic}, and Theorem~\ref{thm:proofs:results:local}).
As demonstrated in Section~\ref{sec:appl}, these results allow for the systematic derivation of theoretical guarantees for the generalized Lasso \eqref{eq:results:bounds:klasso} in a variety of model situations and thereby to explore its capability of fitting linear models to non-linear output rules.
For each of these applications, a crucial step is to select an appropriate \emph{target set}, containing all those parameter vectors that could be used to solve the estimation problem under investigation.
In this way, we were able to meet our initial concerns \ref{quest:intro:estimation}--\ref{quest:intro:complexity} at the end of Subsection~\ref{subsec:intro:statlearn}, most importantly, that an estimation result should not only be accurate but also \emph{interpretable} in a suitable manner.
In particular, this strategy allows us to a certain extent to ``bridge the gap'' between specific estimation tasks, such as learning single-index models, and the abstract setting of statistical learning.

From the application side, it has turned out that the estimator \eqref{eq:results:bounds:klasso} is surprisingly robust against different types of model uncertainties, such as an unknown output function or a misspecified mixing matrix.
Our findings indicate that the estimated parameter vector often carries the desired information (if correctly interpreted), without performing any complicated pre-processing steps.
A general practical conclusion is therefore that, despite its simplicity, constrained least squares minimization often yields a reliable guess of the true parameters. Hence, the outcome of \eqref{eq:results:bounds:klasso} could at least serve as a good initialization for a more sophisticated method that is specifically tailored to the problem of interest.

Let us close our discussion with several open issues, possible extensions as well as some interesting future research directions:
\begin{listing}
\item
	\emph{Hypothesis classes.} 
	As long as the hypothesis class $\funcclass \subset L^2(\R^\p, \mu_{\x})$ is a convex set (cf. \eqref{eq:intro:explossmin-general} and \eqref{eq:intro:emplossmin-general}), the arguments of Section~\ref{sec:proofs} essentially remain valid. More precisely, the applied concentration inequalities from Theorem~\ref{thm:proofs:local:multiplproc} and Theorem~\ref{thm:proofs:local:quadrproc} hold true in a much more general setting (see \cite{mendelson2016multiplier}).
	Nevertheless, such an adaption would require a certain technical effort and is actually not the major concern of this paper.
	For that reason, we stick to case of linear hypothesis functions in order to work out the key ideas.
	
	On the other hand, it is well known that the use of \emph{non-convex} hypothesis classes can substantially increase the learning capacity of empirical risk minimization, e.g., in \emph{deep learning} \cite{lecun2015deep,goodfellow2016deep}. But this gain may come along with very challenging non-convex optimization tasks.
	An interesting approach related to single-index models (cf. Assumption~\ref{model:appl:nonlinear:sim}) is taken by Yang et al.\ in \cite{yang2015sparse}, where $\funcclass$ basically corresponds to a subset of $\{ \x \mapsto \fobs(\sp{\x}{\xpv}) \suchthat \xpv \in \sset \}$.
	In that way, they obtain a consistent estimator even in the sub-Gaussian case but it only applies under very restrictive assumption on $\fobs \colon \R \to \R$. This drawback is mainly due to the fact that the landscape of the empirical risk in \eqref{eq:intro:emplossmin-general} becomes very complicated for non-linear output functions (see also \cite{mei2016landscape}).
	However, these results give evidence that the mismatch principle could be extended to non-linear hypothesis functions.
\item
	\emph{Consistent estimators.} 
	The analysis of Subsection~\ref{subsec:appl:nonlinear} reveals that \eqref{eq:results:bounds:klasso} is typically inconsistent if the input data is non-Gaussian.
	Interestingly, one can resolve this issue by introducing an adaptive estimator, such as in \cite{yang2017stein,yang2017nips,yang2017high}. 
	But we rather aim at a generic strategy that does not leverage explicit knowledge of the probability distribution of $\x$.
	The observation of Remark~\ref{rmk:appl:nonlinear:mim:vs-sim} could be useful at this point, as it indicates that \eqref{eq:results:bounds:klasso} successfully detects the active variables (or the support of an index vector) under a weaker sub-Gaussian assumption.
\item
	\emph{Structured input data.} 
	As already pointed out in Subsection~\ref{subsec:literature:cs}, many problems in signal processing require to deal with structured data vectors, including \emph{non-i.i.d.}\ samples $\{\x_i\}_{i = 1}^n$ and \emph{heavy-tailed} feature variables.
	While the latter concern may be addressed by Mendelson's small ball method \cite{mendelson2014learning,mendelson2014learninggeneral}, the assumption of independent sampling is still crucial to many results in empirical process theory. 
	It is especially not clear how to appropriately extend the statistical tools from Subsection~\ref{subsec:proofs:local}, which form the technical basis of our error bounds.
	Nevertheless, some promising progress has been recently made with the works of \cite{dirksen2017onebitcirc,dirksen2018robust,dirksen2018ditherstruct} in the area of $1$-bit compressed sensing, showing first theoretical guarantees for subsampled Gaussian circulant matrices and heavy-tailed measurements.
\item
	\emph{Loss functions.} A natural variation of \eqref{eq:results:bounds:klasso} is to consider a different convex \emph{loss function} $\loss\colon \R \times \R \to \R$:
	\begin{equation}\label{eq:conclusion:klassoloss}
		\min_{\xpv \in \sset} \ \tfrac{1}{n} \sum_{i = 1}^n \loss(y_i,\sp{\x_i}{\xpv}).
	\end{equation}
	Using the concept of \emph{restricted strong convexity} according to \cite{genzel2016estimation}, one can establish similar error bounds for the estimator \eqref{eq:conclusion:klassoloss} as in Section~\ref{sec:results}.
	However, this asks for a careful adaption of the mismatch covariance (cf. Definition~\ref{def:results:parameters:mismatch}), which is based on a tighter estimate in \eqref{eq:proofs:local:bndmismatchcovar}.
	In that way, the adapted mismatch covariance can even take negative values, thereby improving the estimation performance in certain situations.
	Although not stated explicitly, this phenomenon is exploited in \cite{genzel2018hinge}, where \eqref{eq:conclusion:klassoloss} is studied for the so-called \emph{hinge loss function} in the context of $1$-bit compressed sensing.
	This problem setting comes along with a series of technical challenges, since the hinge loss is not coercive and the expected risk minimizer is not necessarily uniquely defined (e.g., if $y = \sign(\sp{\lat}{\trulatpv})$).
	Therefore, the geometry of the hypothesis set $\sset$ plays a much larger role in \cite{genzel2018hinge}, and in fact, the simplification of Remark~\ref{rmk:results:bounds}\ref{rmk:results:bounds:mismatchpractical} is not applicable anymore. 
	This particularly shows that extending the mismatch principle with regard to \eqref{eq:conclusion:klassoloss} is not straightforward in general.
\item
	\emph{Regularized estimators.} From an optimization perspective, it is often beneficial to solve the \emph{regularized} analog of \eqref{eq:results:bounds:klasso}, that is
	\begin{equation}\label{eq:conclusion:klassoreg}
		\min_{\xpv \in \R^\p} \ \tfrac{1}{n} \sum_{i = 1}^n (y_i - \sp{\x_i}{\xpv})^2 + \lambda \norm{\xpv}_{\sset}
	\end{equation}
	where $\lambda > 0$ is a (tunable) regularization parameter and $\norm{\cdot}_{\sset}$ the Minkowski functional associated with $\sset$.
	Even though \eqref{eq:conclusion:klassoreg} is conceptually strongly related to \eqref{eq:results:bounds:klasso}, it is far from obvious how one can adapt our results to these types of estimators. The interested reader is referred to \cite{lecue2016regularization,lecue2017regularization} for recent
advances in that direction.
\end{listing}

\section{Proofs}
\label{sec:proofs}

The proofs of Theorem~\ref{thm:results:bounds:global} and Theorem~\ref{thm:results:bounds:conic} in Subsection~\ref{subsec:proofs:results} are consequences of a more general error bound (cf. Theorem~\ref{thm:proofs:results:local}) which is based on the so-called local mean width as refined complexity parameter. 
Theorem~\ref{thm:proofs:results:local} is then proven in Subsection~\ref{subsec:proofs:local}. This part is in fact the heart of our analysis and relies on two very recent concentration bounds for empirical stochastic processes (Theorem~\ref{thm:proofs:local:multiplproc} and Theorem~\ref{thm:proofs:local:quadrproc}), allowing for \emph{uniformly} controlling the linear and quadratic term of the excess risk functional.

\subsection{Proofs of the Main Results (Theorem~\ref{thm:results:bounds:global} and Theorem~\ref{thm:results:bounds:conic})}
\label{subsec:proofs:results}

Let us begin with defining the local mean width, which, in a certain sense, generalizes the notion of conic mean width from Definition~\ref{def:results:parameters:meanwidth}:
\begin{definition}[Local mean width]\label{def:proofs:results:localmw}
	Let $\ssetalt \subset \R^\d$ be a subset.
	The \emph{local mean width} of $\ssetalt$ at scale $t > 0$ is defined as
	\begin{equation}
		\meanwidth[t]{\ssetalt} \coloneqq \meanwidth{\tfrac{1}{t}\ssetalt \intersec \S^{\d-1}} = \tfrac{1}{t} \meanwidth{\ssetalt \intersec t\S^{\d-1}}.
	\end{equation}
\end{definition}
The purpose of this geometric parameter is to measure the complexity of a set $\ssetalt$ in a ``local'' neighborhood of $\vnull$, whose size is determined by the scale $t$.
Thus, intuitively speaking, $\meanwidth[t]{\ssetalt}$ captures certain features of $\ssetalt$ at different resolution levels as $t$ varies.
The following lemma relates the local mean width to its conic counterpart, showing that the latter essentially corresponds to the limit case $t \to 0$:
\begin{lemma}\label{lem:proofs:results:mwlocalconic}
Let $\ssetalt \subset \R^\d$ be convex and $\vnull \in \ssetalt$. Then, the mapping $t \mapsto \meanwidth[t]{\ssetalt}$ is non-increasing, and it holds that $\meanwidth[t]{\ssetalt} \to \meanwidth[\conic]{\ssetalt}$ as $t \to 0$. In particular, for all $t > 0$, we have that $\meanwidth[t]{\ssetalt} \leq \meanwidth[\conic]{\ssetalt}$ and $\meanwidth[t]{\ssetalt} \leq \tfrac{1}{t} \meanwidth{\ssetalt}$.
\end{lemma}
\begin{proof}
	This follows from the definition and the basic properties of the (global) mean width as well as the inclusion $\tfrac{1}{t}\ssetalt \subset \cone{\ssetalt}$ for all $t > 0$.
\end{proof}
With regard to the problem setting of this work, the most important conclusion from Lemma~\ref{lem:proofs:results:mwlocalconic} is that the local complexity of a set may be significantly smaller if the considered neighborhood is small but not infinitesimal (cf. \cite[Sec.~2.1]{genzel2018hinge}).
Indeed, as we will see next, this relaxation greatly pays off when investigating the performance of the estimator \eqref{eq:results:bounds:klasso}.
The following result basically shows that, by accepting an approximation error of order $t$, the required sample size is determined by the value of $\meanwidth[t]{\atoms^\T\sset - \targetlatpv}$, rather than $\meanwidth[\conic]{\atoms^\T\sset - \targetlatpv}$:
\begin{theorem}[Estimation via \eqref{eq:results:bounds:klasso} -- Local version]\label{thm:proofs:results:local}
	Let Assumption~\ref{model:results:observations} be satisfied.
	Let $\sset \subset \R^\p$ be a convex subset and fix a vector $\targetlatpv \in \atoms^\T\sset \subset \R^\d$.
	Then there exists numerical constants $C_1, C_2, C_3 > 0$ such that for every $u > 0$ and $t > 0$, the following holds true with probability at least $1 - 2\exp(- C_1 \cdot u^2) - 2\exp(- C_1 \cdot n) - \exp(- C_1 \cdot \subgparam^{-4} \cdot n)$:
	If the number of observed samples obeys
	\begin{equation}\label{eq:proofs:results:local:meas}
		n \geq C_2 \cdot \subgparam^4 \cdot \effdim[t]{\atoms^\T\sset - \targetlatpv},
	\end{equation}
	and
	\begin{equation}\label{eq:proofs:results:local:implbound}
		t > C_3 \cdot \Big( \subgparam \cdot \modeldev{\targetlatpv} \cdot \frac{\meanwidth[t]{\atoms^\T\sset - \targetlatpv} + u}{\sqrt{n}} + \modelcovar{\targetlatpv} \Big),
	\end{equation}
	then every minimizer $\solu\xpv$ of \eqref{eq:results:bounds:klasso} satisfies $\lnorm{\atoms^\T\solu\xpv - \targetlatpv} < t$.
\end{theorem}
Compared to our main results from Section~\ref{sec:results}, the above guarantee is quite implicit, as both sides of the condition \eqref{eq:proofs:results:local:implbound} depend on $t$.
A convenient way to read Theorem~\ref{thm:proofs:results:local} is therefore as follows: First, fix an estimation accuracy $t$ that one is willing to tolerate. Then adjust the sample size $n$ and the target vector $\targetlatpv$ such that \eqref{eq:proofs:results:local:implbound} is fulfilled (if possible at all).
In particular, by allowing for a larger approximation error, the conditions of \eqref{eq:proofs:results:local:meas} and \eqref{eq:proofs:results:local:implbound} become weaker according to Lemma~\ref{lem:proofs:results:mwlocalconic}.
Consequently, the assertion of Theorem~\ref{thm:proofs:results:local} enables a compromise between the desired accuracy and the budget of available samples, which is controlled by means of the local mean width.
The interested reader is referred to \cite[Sec.~III.D]{genzel2016estimation} for a more geometric interpretation of this important trade-off.

The remainder of this part is devoted to an application of Theorem~\ref{thm:proofs:results:local} in order to derive the statements of Theorem~\ref{thm:results:bounds:global} and Theorem~\ref{thm:results:bounds:conic}.
The basic idea is to set $t$ equal to the right-hand side of the (simplified) error bounds in \eqref{eq:results:bounds:global:bound-decay} and \eqref{eq:results:bounds:conic:bound-decay}, respectively, and then to verify that the assumptions \eqref{eq:proofs:results:local:meas} and \eqref{eq:proofs:results:local:implbound} hold true in either case.

Let us start with proving Theorem~\ref{thm:results:bounds:conic}, which in fact requires less technical effort than the proof of Theorem~\ref{thm:results:bounds:global} below.
\begin{proof}[Proof of Theorem~\ref{thm:results:bounds:conic}]
	Following the above roadmap, we apply Theorem~\ref{thm:proofs:results:local} with
	\begin{equation}
		t = 2C_3 \cdot \Big( \subgparam \cdot \modeldev{\targetlatpv} \cdot \frac{\meanwidth[\conic]{\atoms^\T\sset - \targetlatpv} + u}{\sqrt{n}} + \modelcovar{\targetlatpv} \Big)
	\end{equation}
	where $u > 0$ is specified later on.
	Let us first assume that $t > 0$. By Lemma~\ref{lem:proofs:results:mwlocalconic}, we have that
	\begin{align}
		t &> C_3 \cdot \Big( \subgparam \cdot \modeldev{\targetlatpv} \cdot \frac{\meanwidth[\conic]{\atoms^\T\sset - \targetlatpv} + u}{\sqrt{n}} + \modelcovar{\targetlatpv} \Big) \\*
		&\geq C_3 \cdot \Big( \subgparam \cdot \modeldev{\targetlatpv} \cdot \frac{\meanwidth[t]{\atoms^\T\sset - \targetlatpv} + u}{\sqrt{n}} + \modelcovar{\targetlatpv} \Big),
	\end{align}
	implying that \eqref{eq:proofs:results:local:implbound} is satisfied. Moreover, the condition \eqref{eq:proofs:results:local:meas} follows from \eqref{eq:results:bounds:conic:meas}:
	\begin{equation}
		n \gtrsim \subgparam^4 \cdot \underbrace{\delta^{-2}}_{\geq 1} \cdot \effdim[\conic]{\atoms^\T\sset - \targetlatpv} \geq \subgparam^4 \cdot \effdim[t]{\atoms^\T\sset - \targetlatpv}.
	\end{equation}
	Consequently, Theorem~\ref{thm:proofs:results:local} states that, with probability at least $1 - 2\exp(- C_1 \cdot u^2) - 2\exp(- C_1 \cdot n) - \exp(- C_1 \cdot \subgparam^{-4} \cdot n)$, the following error bound holds true:
	\begin{equation}
		\lnorm{\atoms^\T\solu\xpv - \targetlatpv} < t = 2C_3 \cdot \Big( \subgparam \cdot \modeldev{\targetlatpv} \cdot \frac{\meanwidth[t]{\atoms^\T\sset - \targetlatpv} + u}{\sqrt{n}} + \modelcovar{\targetlatpv} \Big).
	\end{equation}
	To obtain \eqref{eq:results:bounds:conic:bound}, we observe that
	\begin{equation}
		\subgparam \cdot \frac{\meanwidth[\conic]{\atoms^\T\sset - \targetlatpv}}{\sqrt{n}} \stackrel{\eqref{eq:results:bounds:conic:meas}}{\lesssim} \subgparam^{-1} \cdot \delta,
	\end{equation}
	and set $u = \subgparam^{-2} \cdot \delta \cdot \sqrt{n}$. Note that the desired probability of success $1 - 5\exp(- C \cdot \subgparam^{-4} \cdot \delta^2 \cdot n)$ is achieved by adjusting the constant $C$ and using that $\delta \leq 1$ and $\subgparam \gtrsim 1$, where the latter is due to
	\begin{equation}\label{eq:proofs:results:conic:subgparam}
		\subgparam \geq \normsubg{\lat} = \sup_{\latpv \in \S^{\d-1}} \normsubg{\sp{\lat}{\latpv}} \stackrel{\eqref{eq:intro:subgnorm}}{\geq} \sup_{\latpv \in \S^{\d-1}} 2^{-1/2} \underbrace{\mean[\abs{\sp{\lat}{\latpv}}^2]^{1/2}}_{\stackrel{\eqref{eq:intro:isotr}}{=} 1} = 2^{-1/2}.
	\end{equation}
	
	Finally, let us consider the case of $t = 0$ (i.e., exact recovery), which is equivalent to $\modeldev{\targetlatpv} = \modelcovar{\targetlatpv} = 0$.
	While we cannot apply Theorem~\ref{thm:proofs:results:local} directly in this situation, the proof steps from Subsection~\ref{subsec:proofs:local} can be still easily adapted:
	If $0 = \modeldev{\targetlatpv} = \normsubg{y - \sp{\lat}{\targetlatpv}}$, the observation model is actually linear and noiseless: $y = \sp{\lat}{\targetlatpv} = \sp{\x}{\target\xpv}$. This immediately implies that the multiplier term $\multiplterm{\cdot, \targetlatpv}$ vanishes, so that $\exloss(\xpv, \target\xpv) = \quadrterm{\atoms^\T\xpv, \targetlatpv}$ for all $\xpv \in \R^\p$. 
	Repeating the argument of Step 3 in the proof of Theorem~\ref{thm:proofs:results:local} with $t = 1$ and $\ssetalt = \cone{\atoms^\T \sset - \targetlatpv} \intersec \S^{\d-1}$, we conclude that, with probability at least $1 - \exp(- C_1 \cdot \subgparam^{-4} \cdot n)$, it holds that
	\begin{equation}\label{eq:proofs:results:conic:noiseless}
		\tfrac{1}{n} \sum_{i = 1}^n \abs[\big]{\sp{\lat_i}{\h}}^2 > 0 
	\end{equation}
	for all $\h \in \cone{\atoms^\T \sset - \targetlatpv} \intersec \S^{\d-1}$.
	Now, if $\xpv' \in \sset_{\targetlatpv,>0} = \{\xpv \in \sset \suchthat \atoms^\T\xpv \neq \targetlatpv \}$, we have $\tfrac{\atoms^\T\xpv' - \targetlatpv}{\lnorm{\atoms^\T\xpv' - \targetlatpv}} \in \cone{\atoms^\T \sset - \targetlatpv} \intersec \S^{\d-1}$ by the convexity of $\atoms^\T \sset - \targetlatpv$. Therefore, on the event of \eqref{eq:proofs:results:conic:noiseless}, it follows that \enlargethispage{\baselineskip}
	\begin{align}
		\exloss(\xpv', \target\xpv) = \quadrterm{\atoms^\T\xpv', \targetlatpv} = \lnorm{\atoms^\T\xpv' - \targetlatpv}^2 \cdot \tfrac{1}{n} \sum_{i = 1}^n \abs[\big]{\sp{\lat_i}{\tfrac{\atoms^\T\xpv' - \targetlatpv}{\lnorm{\atoms^\T\xpv' - \targetlatpv}}}}^2 > 0.
	\end{align}
	Hence, $\xpv'$ cannot be a minimizer of \eqref{eq:results:bounds:klasso}, so that every minimizer $\solu\xpv$ belongs to $\sset_{\targetlatpv,0}$, i.e., $\atoms^\T\solu\xpv = \targetlatpv$.	
\end{proof}

The proof of Theorem~\ref{thm:results:bounds:global} is slightly more involved. We loosely follow the argumentation from \cite[Thm.~3.6]{genzel2017msense}.
\begin{proof}[Proof of Theorem~\ref{thm:results:bounds:global}]
	We apply Theorem~\ref{thm:proofs:results:local} with
	\begin{equation}
		t = D \cdot \Big[ \subgparam \cdot \Big(\frac{\meanwidth{\atoms^\T\sset}}{\sqrt{n}}\Big)^{1/2} + \frac{u}{\sqrt{n}}\Big] + D' \cdot \modelcovar{\targetlatpv} \label{eq:proofs:results:global:t}
	\end{equation}
	where the values of $D, D' \gtrsim 1$ are specified later on.
	Let us first establish an upper bound on the local mean width:
	\begin{align}
		\effdim[t]{\atoms^\T\sset-\targetlatpv} &\leq \tfrac{1}{t^2}\effdim{\atoms^\T\sset-\targetlatpv} = \tfrac{1}{t^2} \effdim{\atoms^\T\sset} \\*
		&\stackrel{\eqref{eq:proofs:results:global:t}}{\leq} \frac{1}{D^2 \cdot \subgparam^2} \cdot \frac{\sqrt{n}}{\meanwidth{\atoms^\T\sset}} \cdot \effdim{\atoms^\T\sset} = \frac{1}{D^2 \cdot \subgparam^2} \cdot \sqrt{n} \cdot \meanwidth{\atoms^\T\sset}, \label{eq:proofs:results:global:estmeas}
	\end{align}
	where we have also used Lemma~\ref{lem:proofs:results:mwlocalconic} and the translation invariance of the mean width (cf. \cite[Prop.~2.1]{plan2013robust}).
	Next, we adjust $D$ such that $D \gtrsim \max\{1, \subgparam \cdot \modeldev{\targetlatpv} \}$ and use \eqref{eq:proofs:results:global:estmeas} to obtain the following bound on the right-hand side of \eqref{eq:proofs:results:local:implbound}:
	\begin{align}
		& C_3 \cdot \Big( \underbrace{ \subgparam \cdot \modeldev{\targetlatpv}}_{\lesssim D \lesssim D^2} \cdot \frac{\meanwidth[t]{\atoms^\T\sset - \targetlatpv} + u}{\sqrt{n}} + \modelcovar{\targetlatpv}\Big) \\*
		\lesssim{} & D^2 \cdot \frac{\meanwidth[t]{\atoms^\T\sset - \targetlatpv}}{\sqrt{n}} + D \cdot \frac{u}{\sqrt{n}} + D' \cdot \modelcovar{\targetlatpv} \\
		\stackrel{\eqref{eq:proofs:results:global:estmeas}}{\leq} {} & D^2 \cdot \frac{1}{D \cdot \subgparam} \cdot \frac{\sqrt{\meanwidth{\atoms^\T\sset}} \cdot n^{1/4}}{\sqrt{n}} + D \cdot \frac{u}{\sqrt{n}} + D' \cdot \modelcovar{\targetlatpv} \\
		={} & D \cdot \underbrace{\subgparam^{-1}}_{\stackrel{\eqref{eq:proofs:results:conic:subgparam}}{\leq} 2 \subgparam} \cdot \Big(\frac{\meanwidth{\atoms^\T\sset}}{\sqrt{n}}\Big)^{1/2} + D \cdot \frac{u}{\sqrt{n}} + D' \cdot \modelcovar{\targetlatpv} \\
		\lesssim {} & D \cdot \Big[ \subgparam \cdot \Big(\frac{\meanwidth{\atoms^\T\sset}}{\sqrt{n}}\Big)^{1/2} + \frac{u}{\sqrt{n}}\Big] + D' \cdot \modelcovar{\targetlatpv}  = t. \label{eq:proofs:results:global:estimplbound}
	\end{align}
	Consequently, if $D' \gtrsim 1$ is large enough and $D = \tilde{C} \cdot \max\{1, \subgparam \cdot \modeldev{\targetlatpv} \}$, with $\tilde{C} > 0$ being an appropriate numerical constant, we conclude that \eqref{eq:proofs:results:global:estimplbound} implies the condition \eqref{eq:proofs:results:local:implbound} of Theorem~\ref{thm:proofs:results:local}. Moreover, combining \eqref{eq:proofs:results:global:estmeas} and \eqref{eq:results:bounds:global:meas}, we obtain	\begin{equation}
		\effdim[t]{\atoms^\T\sset-\targetlatpv} \leq \frac{1}{D^2 \cdot \subgparam^2} \cdot \sqrt{n} \cdot \meanwidth{\atoms^\T\sset} \stackrel{\eqref{eq:results:bounds:global:meas}}{\lesssim} \frac{1}{D^2 \cdot \subgparam^4} \cdot \underbrace{\delta^{2}}_{\leq 1} {}\cdot{} n \leq \frac{1}{\tilde{C}^2 \cdot \subgparam^4} \cdot n, 
	\end{equation}
	which implies \eqref{eq:proofs:results:local:meas}, where $\tilde{C}$ might have to be slightly enlarged again.
	
	Consequently, Theorem~\ref{thm:proofs:results:local} with $u = \delta \cdot \sqrt{n}$ states that, with probability at least $1 - 2 \exp(-C_1 \cdot \delta^2 \cdot n) - 2 \exp(-C_1 \cdot n) - \exp(-C_1 \cdot \subgparam^{-4} \cdot n)$, every minimizer $\solu\xpv$ of \eqref{eq:results:bounds:klasso} satisfies the error bound
	\begin{align} 
		\lnorm{\atoms^\T\solu\xpv - \targetlatpv} < t &\lesssim \max\{1, \subgparam \cdot \modeldev{\targetlatpv} \} \cdot \Big[ \underbrace{\subgparam \cdot \Big(\frac{\meanwidth{\atoms^\T\sset}}{\sqrt{n}}\Big)^{1/2}}_{\stackrel{\eqref{eq:results:bounds:global:meas}}{\lesssim} \delta} + \frac{u}{\sqrt{n}}\Big] + \modelcovar{\targetlatpv}  \\*
		&\lesssim \max\{1, \subgparam \cdot \modeldev{\targetlatpv} \} \cdot \delta + \modelcovar{\targetlatpv}.
	\end{align}
	
	Finally, since $\delta \leq 1$ and $\subgparam \gtrsim 1$ by \eqref{eq:proofs:results:conic:subgparam}, we achieve the desired probability of success $1 - 5\exp(- C \cdot \subgparam^{-4} \cdot \delta^2 \cdot n)$ by adjusting $C$.
\end{proof}

\subsection{Proof of Theorem~\ref{thm:proofs:results:local}}
\label{subsec:proofs:local}

Throughout this subsection, we assume that Assumption~\ref{model:results:observations} is satisfied, and recall the notations from Table~\ref{tab:results:terminology}.
For the sake of readability, let us also introduce some additional notation:
The objective function of \eqref{eq:results:bounds:klasso} --- typically referred to as the \emph{empirical risk} --- is denoted by
\begin{equation}
	\lossemp{}(\xpv) \coloneqq \tfrac{1}{n} \sum_{i = 1}^n (y_i - \sp{\x_i}{\xpv})^2, \quad \xpv \in \R^\p,
\end{equation}
and the associated \emph{excess risk} by
\begin{equation}
	\exloss(\xpv, \xpv') \coloneqq \lossemp{}(\xpv) - \lossemp{}(\xpv'), \quad \xpv, \xpv' \in \R^\p.
\end{equation}
Moreover, we define the following subsets of $\sset$:
\begin{align}
	\sset_{\targetlatpv,t} &\coloneqq \{\xpv \in \sset \suchthat \lnorm{\atoms^\T\xpv - \targetlatpv} = t \}, \\*
	\sset_{\targetlatpv,>t} &\coloneqq \{\xpv \in \sset \suchthat \lnorm{\atoms^\T\xpv - \targetlatpv} > t \}, \\*
	\sset_{\targetlatpv,<t} &\coloneqq \{\xpv \in \sset \suchthat \lnorm{\atoms^\T\xpv - \targetlatpv} < t \}.
\end{align}
Finally, by the assumption $\targetlatpv \in \atoms^\T \sset$ in Theorem~\ref{thm:proofs:results:local}, there exists $\target\xpv \in \sset$ such that $\targetlatpv = \atoms^\T \target\xpv$. Note that this parameter vector could be highly non-unique, but the arguments below apply to every choice of $\target\xpv$ with $\targetlatpv = \atoms^\T \target\xpv$.

\begin{figure}
	\centering
	\tikzstyle{blackdot}=[shape=circle,fill=black,minimum size=1mm,inner sep=0pt,outer sep=0pt]
		\begin{tikzpicture}[scale=2]
			\coordinate (K1) at (-.5,-.3);
			\coordinate (K2) at (.5,-1);
			\coordinate (K3) at (.3,-1.7);
			\coordinate (K4) at (-.2,-2.1);
			\coordinate (K5) at (-1,-1.2);
			\coordinate[below right=.35cm and .15 of K1] (muX0);
			
			\draw[fill=gray!20!white, name path = K] (K1) -- (K2) -- (K3) -- (K4) -- (K5) -- cycle;
			\node[draw,circle,fill=gray!20!white,minimum size=1.4cm] at (muX0) {}; 
			\begin{scope}
				\clip (K1) -- (K2) -- (K3) -- (K4) -- (K5) -- cycle;
				\node[draw=red,circle,fill=gray!60!white,minimum size=1.4cm,ultra thick] at (muX0) {};
			\end{scope}
			\draw[thick] (K1) -- (K2) -- (K3) -- (K4) -- (K5) -- cycle;
			
			\node at (barycentric cs:K1=1,K2=1,K3=1,K4=1,K5=1) {$\atoms^\T\sset$};
			\node[blackdot] at (muX0) {};
			
			\path (muX0) -- ++(-40:.35cm) coordinate (anchorSphere);
			\node[above right=.2 and 1 of anchorSphere] (anchorSpherelabel) {$\atoms^\T\sset \intersec (t\S^{\d-1} + \targetlatpv) = \atoms^\T \sset_{\targetlatpv,t}$};
			\path[<-,shorten <=3pt,>=stealth,red,bend right] (anchorSphere) edge (anchorSpherelabel);
			
			\node[left=1.2 of muX0] (muX0label) {$\targetlatpv$};
			\path[<-,shorten <=3pt,>=stealth,bend right] (muX0) edge (muX0label);
			
		\end{tikzpicture}
	\caption{Illustration of the main proof argument. The key difficulty is to show that $\exloss(\xpv, \target\xpv) > 0$ for all $\xpv \in \sset_{\targetlatpv,t}$ (see red arc). Hence, if $\solu\xpv$ is a minimizer of \protect\eqref{eq:results:bounds:klasso}, we can conclude that $\atoms^\T\solu\xpv \in \R^\d$ must belong to the dark gray intersection $\atoms^\T \sset_{\targetlatpv,<t}$, which simply means that $\lnorm{\atoms^\T\solu\xpv - \targetlatpv} < t$.}
	\label{fig:proofs:local:argument}
\end{figure}
The proof of Theorem~\ref{thm:proofs:results:local} is based on a classical localization idea from learning theory (cf. \cite{mendelson2002improving,bartlett2005local,mendelson2007reconstruction}):
The claim follows if we can show that $\solu\xpv \in \sset_{\targetlatpv,<t}$ for every minimizer $\solu\xpv$ of \eqref{eq:results:bounds:klasso}.
The key step of our proof below is therefore to verify that, with high probability, we have $\exloss(\xpv, \target\xpv) > 0$ uniformly for all $\xpv \in \sset_{\targetlatpv,t}$.
This particularly implies that $\sset_{\targetlatpv,t}$ cannot contain a minimizer $\solu\xpv$ of \eqref{eq:results:bounds:klasso}, since $\exloss(\solu\xpv, \target\xpv) \leq 0$.
Finally, by the convexity of $\lossemp{}$ and $\sset$, it even follows that $\exloss(\xpv, \target\xpv) > 0$ for all $\xpv \in \sset_{\targetlatpv,>t}$. In turn, every minimizer of \eqref{eq:results:bounds:klasso} must belong to $\sset_{\targetlatpv,<t}$, which concludes the argument; see Figure~\ref{fig:proofs:local:argument} for a visualization.
	

\subsubsection*{Step 1: Decomposing the Excess Risk}

We first decompose the excess risk functional into its linear and quadratic part, corresponding to the first and second order Taylor term, respectively:
\begin{align}
	\exloss(\xpv, \target\xpv) &= \lossemp{}(\xpv) - \lossemp{}(\target\xpv) \\*
	&= \tfrac{1}{n} \sum_{i = 1}^n (y_i - \sp{\x_i}{\xpv})^2 - \tfrac{1}{n} \sum_{i = 1}^n (y_i - \sp{\x_i}{\target\xpv})^2 \\
	&= \tfrac{2}{n} \sum_{i = 1}^n (\sp{\x_i}{\target\xpv} - y_i) \sp{\x_i}{\xpv - \target\xpv} + \tfrac{1}{n} \sum_{i = 1}^n \abs[\big]{\sp{\x_i}{\xpv - \target\xpv}}^2 \\
	&= \tfrac{2}{n} \sum_{i = 1}^n \underbrace{(\sp{\atoms\lat_i}{\target\xpv} - y_i)}_{= \sp{\lat_i}{\targetlatpv} - y_i \eqqcolon \xi_i(\targetlatpv)} \sp{\atoms\lat_i}{\xpv - \target\xpv} + \tfrac{1}{n} \sum_{i = 1}^n \abs[\big]{\sp{\atoms\lat_i}{\xpv - \target\xpv}}^2 \\
	&= \underbrace{\tfrac{2}{n} \sum_{i = 1}^n \xi_i(\targetlatpv) \sp{\lat_i}{\atoms^\T\xpv - \targetlatpv}}_{\eqqcolon \multiplterm{\atoms^\T\xpv, \targetlatpv}} + \underbrace{\tfrac{1}{n} \sum_{i = 1}^n \abs[\big]{\sp{\lat_i}{\atoms^\T\xpv - \targetlatpv}}^2}_{\eqqcolon \quadrterm{\atoms^\T\xpv, \targetlatpv}}. \label{eq:proofs:local:excessdecomp}
\end{align}
The factors $\xi_i(\targetlatpv) \coloneqq \sp{\lat_i}{\targetlatpv} - y_i$ can be regarded as \emph{multipliers} that do not depend on $\xpv$, but are \emph{not} independent of $\lat_i$. For that reason, $\multiplterm{\atoms^\T\xpv, \targetlatpv}$ is referred to as the \emph{multiplier term}, whereas $\quadrterm{\atoms^\T\xpv, \targetlatpv}$ is called the \emph{quadratic term}.
Moreover, we emphasize that both terms do actually only depend on $\atoms^\T\xpv$ and $\targetlatpv = \atoms^\T\target\xpv$, implying that $\exloss(\tilde{\xpv}, \target{\tilde{\xpv}}) = \exloss(\xpv, \target{\xpv})$ for all $\tilde{\xpv}, \target{\tilde{\xpv}} \in \R^\p$ with $\atoms^\T\tilde{\xpv} = \atoms^\T\xpv$ and $\atoms^\T\target{\tilde{\xpv}} = \atoms^\T\target\xpv$.

The goal of the following two steps is to establish \emph{uniform} lower bounds for both empirical processes, $\xpv \mapsto \multiplterm{\atoms^\T\xpv, \targetlatpv}$ and $\xpv \mapsto \quadrterm{\atoms^\T\xpv, \targetlatpv}$, on the neighborhood set $\sset_{\targetlatpv,t}$.
Under the hypothesis of \eqref{eq:proofs:results:local:meas} and \eqref{eq:proofs:results:local:implbound}, this eventually leads to the desired positivity of the excess risk in Step~4.

\subsubsection*{Step 2: Bounding the Multiplier Term}

Let us start with the multiplier term $\multiplterm{\cdot, \targetlatpv}$.
For this purpose, we apply the following recent concentration inequality on empirical multiplier processes due to Mendelson \cite{mendelson2016multiplier}, which is based on a sophisticated chaining argument:
\begin{theorem}[\protect{\cite[Thm.~4.4]{mendelson2016multiplier}}]\label{thm:proofs:local:multiplproc}
	Let Assumption~\ref{model:results:observations} be satisfied and let $\ssetalt \subset t \S^{\d-1}$ for some $t > 0$.
	Let $\xi$ be a sub-Gaussian random variable (not necessarily independent of $\lat$) and let $\xi_1, \dots, \xi_n$ be independent copies of $\xi$.\footnote{More precisely, the independent copies of $\xi$ are generated \emph{jointly} with $\lat$ and $y$, i.e., $(\lat_i,y_i,\xi_i)$ is an independent copy of $(\lat,y,\xi)$.}
	Then there exist numerical constants $C_1, C' > 0$ such that for every $u > 0$, the following holds true with probability at least $1 - 2 \exp(-C_1 \cdot u^2) - 2 \exp(- C_1 \cdot n)$:
	\begin{equation}\label{eq:proofs:local:multiplproc:concentr}
		\sup_{\h \in \ssetalt} \abs[\Big]{\Big(\tfrac{1}{n}\sum_{i = 1}^n \xi_i \sp{\lat_i}{\h} \Big) - \mean[\xi \sp{\lat}{\h}]} \leq C' \cdot \subgparam \cdot \normsubg{\xi} \cdot  \frac{\meanwidth{\ssetalt} + u \cdot t}{\sqrt{n}} \ .
	\end{equation}
\end{theorem}

The definition of $\multiplterm{\cdot,\targetlatpv}$ implies that we can apply Theorem~\ref{thm:proofs:local:multiplproc} with $\ssetalt = \atoms^\T \sset_{\targetlatpv,t} - \targetlatpv \subset t\S^{\d-1}$ and $\xi = \sp{\lat}{\targetlatpv} - y$.
Accordingly, we have that, with probability at least $1 - 2 \exp(-C_1 \cdot u^2) - 2 \exp(- C_1 \cdot n)$, the following bound holds true for every $\xpv \in \sset_{\targetlatpv,t}$ ($\implies \atoms^\T\xpv - \targetlatpv \in \ssetalt$):\pagebreak
\begin{align}
    \tfrac{1}{2} \cdot \multiplterm{\atoms^\T\xpv, \targetlatpv} &= \tfrac{1}{n} \sum_{i = 1}^n \xi_i(\targetlatpv) \sp{\lat_i}{\atoms^\T\xpv - \targetlatpv} \\*
		&\geq \mean[\xi \sp{\lat}{\atoms^\T\xpv - \targetlatpv}] - C' \cdot \subgparam \cdot \normsubg{\xi} \cdot \frac{\meanwidth{\ssetalt} + u \cdot t}{\sqrt{n}} \\
		&= - t \cdot \Big( \underbrace{\mean[\xi \sp{\lat}{\tfrac{\targetlatpv - \atoms^\T\xpv}{t}}][\Big]}_{\stackrel{(\ast)}{\leq} \modelcovar{\targetlatpv}} {}+{} C' \cdot \subgparam \cdot \underbrace{\normsubg{\xi}}_{= \modeldev{\targetlatpv}} \cdot \frac{\tfrac{1}{t}\meanwidth{\atoms^\T \sset_{\targetlatpv,t} - \targetlatpv} + u}{\sqrt{n}} \Big) \\
		&\stackrel{\mathllap{\atoms^\T \sset_{\targetlatpv,t} = \atoms^\T\sset \intersec (t \S^{\d-1} + \targetlatpv)}}{\geq} 
			- t \cdot \Big( \modelcovar{\targetlatpv} + C' \cdot \subgparam \cdot \modeldev{\targetlatpv} \cdot \frac{\meanwidth[t]{\atoms^\T\sset - \targetlatpv} + u}{\sqrt{n}} \Big) \\
		&\geq
			- t \cdot \max\{1, C'\} \cdot \underbrace{\Big( \modelcovar{\targetlatpv} + \subgparam \cdot \modeldev{\targetlatpv} \cdot \frac{\meanwidth[t]{\atoms^\T\sset - \targetlatpv} + u}{\sqrt{n}} \Big)}_{\eqqcolon t_0},
\end{align}
where $(\ast)$ is a consequence of the Cauchy-Schwarz inequality:
\begin{align}
	\mean[\xi \sp{\lat}{\tfrac{\targetlatpv - \atoms^\T\xpv}{t}}][\Big] &= \mean[(\sp{\lat}{\targetlatpv} - y) \sp{\lat}{\tfrac{\targetlatpv - \atoms^\T\xpv}{t}}][\Big] \\*
	&= \sp[\Big]{\mean[(\sp{\lat}{\targetlatpv} - y)\lat]}{\tfrac{\targetlatpv - \atoms^\T\xpv}{t}} \\
	&\leq \lnorm[\big]{\mean[(\sp{\lat}{\targetlatpv} - y)\lat]} \cdot \underbrace{\lnorm[\Big]{\tfrac{\targetlatpv - \atoms^\T\xpv}{t}}}_{= 1} = \modelcovar{\targetlatpv}.\label{eq:proofs:local:bndmismatchcovar}
\end{align}
Hence, we end up with
\begin{equation}\label{eq:proofs:local:lowerbndmultipl}
	\multiplterm{\atoms^\T\xpv, \targetlatpv} \geq -2 \cdot \max\{1, C'\} \cdot t \cdot t_0 \quad \text{for all $\xpv \in \sset_{\targetlatpv,t}$ \ .}
\end{equation}

\subsubsection*{Step 3: Bounding the Quadratic Term}

The quadratic term $\quadrterm{\cdot, \targetlatpv}$ can be handled by another chaining-based concentration inequality from random matrix theory:
\begin{theorem}[\protect{\cite[Thm.~1.3]{liaw2016randommat}}]\label{thm:proofs:local:quadrproc}
	Let Assumption~\ref{model:results:observations} be satisfied and let $\ssetalt \subset t \S^{\d-1}$ for some $t > 0$.
	Then there exists a numerical constant $C'' > 0$ such that for every $u \geq 0$, the following holds true with probability at least $1 - \exp(- u^2)$:
	\begin{equation}\label{eq:proofs:local:quadrproc:concentr}
		\sup_{\h \in \ssetalt} \abs[\Big]{\Big(\tfrac{1}{n}\sum_{i = 1}^n \abs{\sp{\lat_i}{\h}}^2 \Big)^{1/2} - t} \leq C'' \cdot \subgparam^2 \cdot  \frac{\meanwidth{\ssetalt} + u \cdot t}{\sqrt{n}} \ .
	\end{equation}
\end{theorem}

Similarly to Step~2, we apply Theorem~\ref{thm:proofs:local:quadrproc} with $\ssetalt = \atoms^\T \sset_{\targetlatpv,t} - \targetlatpv \subset t\S^{\d-1}$ and $u = \sqrt{C_1 \cdot \subgparam^{-4} \cdot n}$. Hence, with probability at least $1 - \exp(- C_1 \cdot \subgparam^{-4} \cdot n)$, the following holds true for every $\xpv \in \sset_{\targetlatpv,t}$:
\begin{align}
	\sqrt{\quadrterm{\atoms^\T\xpv, \targetlatpv}} &= \Big( \tfrac{1}{n} \sum_{i = 1}^n \abs[\big]{\sp{\lat_i}{\atoms^\T\xpv - \targetlatpv}}^2 \Big)^{1/2} \\*
	&\geq t - C'' \cdot \subgparam^2 \cdot \frac{\meanwidth{\ssetalt} + u \cdot t}{\sqrt{n}} \\
	&= t \cdot \Big( 1 - C'' \cdot \frac{\subgparam^2 \cdot \tfrac{1}{t}\meanwidth{\atoms^\T \sset_{\targetlatpv,t} - \targetlatpv}}{\sqrt{n}} {}-{} C'' \cdot \frac{\subgparam^2 \cdot u}{\sqrt{n}} \Big) \\
	&= t \cdot \Big( 1 - C'' \cdot \underbrace{\frac{\subgparam^2 \cdot \meanwidth[t]{\atoms^\T \sset - \targetlatpv}}{\sqrt{n}}}_{\stackrel{\eqref{eq:proofs:results:local:meas}}{\leq} 1 / \sqrt{C_2}} {} - {} C'' \sqrt{C_1} \Big) \\
	&\geq t \cdot \underbrace{(1 - C'' / \sqrt{C_2} - C'' \sqrt{C_1})}_{\eqqcolon C_0}.
\end{align}
Finally, we adjust the numerical constants $C_1$ and $C_2$ such that $C_0 > 0$, implying that
\begin{equation}\label{eq:proofs:local:lowerbndquadr}
	\quadrterm{\atoms^\T\xpv, \targetlatpv} \geq C_0^2 \cdot t^2 \quad \text{for all $\xpv \in \sset_{\targetlatpv,t}$ \ .}
\end{equation}

\subsubsection*{Step 4: Bounding the Excess Risk and Conclusion}

We now assume that the events of Step~2 and Step~3 have indeed occurred (with probability at least $1 - 2\exp(- C_1 \cdot u^2) - 2\exp(- C_1 \cdot n) - \exp(- C_1 \cdot \subgparam^{-4} \cdot n)$). 
Hence, the lower bounds of \eqref{eq:proofs:local:lowerbndmultipl} and \eqref{eq:proofs:local:lowerbndquadr} yield
\begin{align}
	\exloss(\xpv, \target\xpv) &= \quadrterm{\atoms^\T\xpv, \targetlatpv} + \multiplterm{\atoms^\T\xpv, \targetlatpv} \\*
	&\geq C_0^2 \cdot t^2 - 2 \cdot \max\{1, C'\} \cdot t \cdot t_0 \\
	&= t \cdot \underbrace{(C_0^2 \cdot t - 2 \cdot \max\{1, C'\} \cdot t_0)}_{\stackrel{\eqref{eq:proofs:results:local:implbound}}{>} 0} > 0 \quad \text{for all $\xpv \in \sset_{\targetlatpv,t}$ \ ,} \label{eq:proofs:local:excesspos}
\end{align}
where $C_3$ needs to be appropriately adjusted in \eqref{eq:proofs:results:local:implbound}, depending on the numerical constants $C_0$ and $C'$.

Since $\exloss(\solu\xpv, \target\xpv) \leq 0$ for every minimizer $\solu\xpv$ of \eqref{eq:results:bounds:klasso}, we can conclude that $\solu\xpv$ cannot belong to $\sset_{\targetlatpv,t}$.
Next, on the same event, assume that there is a minimizer $\solu\xpv$ of \eqref{eq:results:bounds:klasso} with $\solu\xpv \in \sset_{\targetlatpv,>t}$. By the convexity of $\sset$, it is not hard to see that the line segment $\lambda \mapsto \operatorname{Seg}(\lambda) \coloneqq \lambda \target\xpv + (1 - \lambda)\solu\xpv$ intersects $\sset_{\targetlatpv,t}$ at a point $\xpv' \coloneqq \lambda' \target\xpv + (1 - \lambda')\solu\xpv$ with $\lambda' \in \intvcl{0}{1}$. Thus,
\begin{align}
	&\exloss(\operatorname{Seg}(0), \target\xpv) = \exloss(\solu\xpv, \target\xpv) \leq 0, \\*
	&\exloss(\operatorname{Seg}(\lambda'), \target\xpv) = \exloss(\xpv', \target\xpv) \stackrel{\eqref{eq:proofs:local:excesspos}}{>} 0, \\*
	&\exloss(\operatorname{Seg}(1), \target\xpv) = \exloss(\target\xpv, \target\xpv) = 0, 
\end{align}
which contradicts the convexity of the mapping $\lambda \mapsto \exloss(\operatorname{Seg}(\lambda), \target\xpv)$.
In turn, every minimizer $\solu\xpv$ of \eqref{eq:results:bounds:klasso} must belong to $\sset_{\targetlatpv,<t}$, which implies the desired error bound $\lnorm{\atoms^\T\solu\xpv - \targetlatpv} < t$. \hfill \qedsymbol

\begin{remark}[Related approaches]\label{rmk:proofs:local:techniques}
	The above proof strategy loosely follows the learning framework of Mendelson from \cite{mendelson2014learning,mendelson2014learninggeneral}.
	Based on a similar decomposition of the excess risk as in \eqref{eq:proofs:local:excessdecomp}, the key idea of Mendelson's approach is to show that the quadratic term dominates the multiplier term except for a small neighborhood of the expected risk minimizer.
	Thus, according to the convexity argument of Step~4, one can conclude that the empirical risk minimizer belongs to this small neighborhood, which in turn yields an error bound.
	
	While we are able to invoke a concentration result for the quadratic term due to the sub-Gaussianity of $\lat$, a major concern of \cite{mendelson2014learning,mendelson2014learninggeneral} is that one can already establish lower bounds under a so-called small ball condition.
	This technique is referred to as \emph{Mendelson's small ball method} and allows for proving learning guarantees under much weaker moment assumptions on the input data.
	
	On the other hand, as pointed out in Subsection~\ref{subsec:literature:statlearn}, a key difference to \cite{mendelson2014learning,mendelson2014learninggeneral} is that we do not restrict to the expected risk minimizer as target vector but permit any choice $\targetlatpv \in \atoms^\T\sset$.
	This particularly explains why the multiplier term $\multiplterm{\cdot, \targetlatpv}$ needs to be treated somewhat differently in our setup.
	Apart from that, it is worth mentioning that the above analysis also improves the related approaches of \cite{plan2015lasso,genzel2016estimation}. These works on \emph{Gaussian} single-index models handle the multiplier term by a simpler argument based on Markov's inequality, which eventually leads to a more pessimistic probability of success.
\end{remark}

\subsection{Proof of Proposition~\ref{prop:results:setup:modeldecomp}}
\label{subsec:proofs:modeldecomp}

\begin{proof}[Proof of Proposition~\ref{prop:results:setup:modeldecomp}]
	Since the covariance matrix of $\x$ is positive semi-definite and of rank $\d$, there exists $\U \in \R^{\p \times \d}$ with orthonormal columns and a diagonal matrix $\D \in \R^{\d \times \d}$ with positive entries such that $\mean[\x\x^\T] = \U \D \D \U^\T$.
	We set $\lat \coloneqq \D^{-1}\U^\T \x$ and $\atoms \coloneqq \U \D \in \R^{\p \times \d}$.
	
	It is not hard to see that $\lat$ is a centered isotropic random vector in $\R^\d$. Indeed, we have $\mean[\lat] = \D^{-1}\U^\T \mean[\x] = \vnull$ and
	\begin{equation}
		\mean[\lat \lat^\T] = \D^{-1}\U^\T \mean[\x \x^\T] \U \D^{-1} = \D^{-1}\U^\T \U \D \D \U^\T \U \D^{-1} = \I{\d} \ .
	\end{equation}
	Finally, let us compute the covariance matrix of $\x - \atoms\lat$:
	\begin{align}
		&\mean[(\x - \atoms\lat)(\x - \atoms\lat)^\T] = \mean[\x\x^\T]  + \mean[\atoms\lat\lat^\T\atoms^\T] - \mean[\x \lat^\T\atoms^\T] - \mean[\atoms\lat\x^\T] \\*
		={} & \U \D \D \U^\T + \U \D \D \U^\T - \mean[\x\x^\T] \U \D^{-1} \D \U^\T - \U\D \D^{-1} \U^\T \mean[\x\x^\T] \\
		={} & 2\U \D \D \U^\T - \U \D \D \U^\T \U \D^{-1} \D \U^\T - \U\D \D^{-1} \U^\T \U \D \D \U^\T \\
		={} & 2\U \D \D \U^\T - 2\U \D \D \U^\T = \vnull.
	\end{align}
	Hence, we conclude that $\x = \atoms\lat$ almost surely.
\end{proof}
	


\section*{Acknowledgments}
{\smaller
The authors thank Chandrajit Bajaj, Peter Jung, Maximillian März, and Alexander Stollenwerk for fruitful discussions, in particular for pointing out potential applications.
M.G.\ is supported by the European Commission Project DEDALE (contract no. 665044) within the H2020 Framework Program and by the Bundesministerium für Bildung und Forschung
(BMBF) through the Berliner Zentrum for Machine Learning (BZML), Project AP4.
G.K. acknowledges partial support by the Bundesministerium für Bildung und Forschung (BMBF) through the Berliner Zentrum for Machine Learning (BZML), Project AP4, RTG DAEDALUS (RTG 2433), Projects P1 and P3, RTG BIOQIC (RTG 2260), Projects P4 and P9, and by the Berlin Mathematics Research Center MATH+, Projects EF1-1 and EF1-4.

}

\renewcommand*{\bibfont}{\smaller}
\printbibliography[heading=bibintoc]

\newpage
\listoftodos

\end{document}